\documentclass[10pt]{amsart}

\IfFileExists{srcltx.sty}{\usepackage{srcltx}}

\numberwithin{equation}{section}

\usepackage[latin1]{inputenc}
\usepackage{xspace,amssymb,amsfonts,euscript}
\usepackage{amsthm,amsmath}
\usepackage{palatino}
\usepackage{euscript}
\input xy \xyoption {all}
\usepackage{tikz}

\RequirePackage{color}
\definecolor{myred}{rgb}{0.75,0,0}
\definecolor{mygreen}{rgb}{0,0.5,0}
\definecolor{myblue}{rgb}{0,0,0.65}

\RequirePackage{ifpdf}
\ifpdf
  \IfFileExists{pdfsync.sty}{\RequirePackage{pdfsync}}{}
  \RequirePackage[pdftex,
   colorlinks = true,
   urlcolor = myblue, 
   citecolor = mygreen, 
   linkcolor = myred, 
   pagebackref,
   bookmarksopen=true]{hyperref}
\else
  \RequirePackage[hypertex]{hyperref}
\fi

    \def\BM{{\mathbb{B}}}
    \def\CM{{\mathbb{C}}}
\def\DG{{\mathfrak D}}    
    
\def\FG{{\mathfrak F}}

\def\MG{{\mathfrak M}}

    \def\QM{{\mathbb{Q}}}
    \def\RM{{\mathbb{R}}}

    \def\ZM{{\mathbb{Z}}}

    \def\AC{{\mathcal{A}}}
    \def\BC{{\mathcal{B}}}
    \def\CC{{\mathcal{C}}}

    \def\FC{{\mathcal{F}}}

    \def\IC{{\mathcal{I}}}

    \def\LC{{\mathcal{L}}}
    \def\MC{{\mathcal{M}}}
    \def\NC{{\mathcal{N}}}
    \def\OC{{\mathcal{O}}}
\def\PB{{\mathbf P}}    \def\PC{{\mathcal{P}}}
    
    \def\RC{{\mathcal{R}}}
    
    \def\TC{{\mathcal{T}}}
    
    \def\VC{{\mathcal{V}}}

\def\a{\alpha}
\def\b{\beta}
\def\g{\gamma}
\def\G{\Gamma}
\def\d{\delta}
\def\D{\Delta}
\def\e{\varepsilon}

\def\O{\Omega}
\def\r{\rho}
\def\s{\sigma}

\def\z{\zeta}

\newcommand{\nc}{\newcommand} \newcommand{\renc}{\renewcommand}

\newcommand{\rdots}{\mathinner{ \mkern1mu\raise1pt\hbox{.}
    \mkern2mu\raise4pt\hbox{.}
    \mkern2mu\raise7pt\vbox{\kern7pt\hbox{.}}\mkern1mu}}

\def\grdim{{\mathrm{grdim}}}

\def\wt{\widetilde}
\def\ov{\overline}
\def\un{\underline}

\def\to{\rightarrow}

\def\longto{\longrightarrow}

\nc{\triright}{\stackrel{[1]}{\to}}
\nc{\longtriright}{\stackrel{[1]}{\longto}}

\nc{\Br}{\mathcal{B}}
\nc{\HotRR}{{}_R\mathcal{K}_R}
\nc{\HotR}{\mathcal{K}_R}
\nc{\excise}[1]{}
\nc{\defect}{\text{df}}
\nc{\h}[1]{\underline{H}_{#1}}

\nc{\Ga}{\mathbb{G}_a} 
\nc{\Gm}{\mathbb{G}_m} 

\nc{\Perv}{{\mathbf{P}}}

\nc{\IH}{{\mathrm{IH}}}

\nc{\ic}{\mathbf{IC}}

\nc{\gl}{{\mathfrak{gl}}}
\renc{\sl}{{\mathfrak{sl}}}
\renc{\sp}{{\mathfrak{sp}}}

\nc{\HBM}{H^{BM}}

\newtheorem{thm}{Theorem}[section]
\newtheorem{lem}[thm]{Lemma}

\newtheorem{prop}[thm]{Proposition}
\newtheorem{cor}[thm]{Corollary}

\theoremstyle{definition}
\newtheorem{defi}[thm]{Definition}
\newtheorem{ex}[thm]{Example}

\theoremstyle{remark}
\newtheorem{remark}[thm]{Remark}
\newtheorem{question}[thm]{Question}

\DeclareMathOperator{\Tor}{Tor}

\newcommand{\into}{\hookrightarrow}

\newcommand{\fib}{\twoheadrightarrow}

\def\iff{\Leftrightarrow}

\def\op{{\mathrm{op}}}

\newcommand{\eq}[1]{\begin{align*} #1 \end{align*}}
\DeclareMathOperator{\Sym}{Sym}
\newcommand{\ext}{\mathrm{ext}}
\newcommand{\norm}[1]{\lVert#1\rVert}
\DeclareMathOperator{\ord}{ord}
\newcommand{\fl}{\FG}
\newcommand{\flem}{\fl_{em}}
\newcommand{\fllv}{\fl_{lv}}

\begin{document}

\begin{abstract}
We associate to a sufficiently generic oriented matroid program and choice of linear system of parameters a finite dimensional algebra, whose representation theory is analogous to blocks of Bernstein--Gelfand--Gelfand category $\OC$.  When the data above comes from a generic linear program for a hyperplane arrangement, we recover the algebra defined by Braden--Licata--Proudfoot--Webster.

Applying our construction to nonlinear oriented matroid programs provides a large new class of algebras.  For Euclidean oriented matroid programs, the resulting algebras are quasi-hereditary and Koszul, as in the linear setting.  In the non-Euclidean case, we obtain algebras that are not quasi-hereditary and not known to be Koszul, but still have a natural class of standard modules and satisfy numerical analogues of quasi-heredity and Koszulity on the level of graded Grothendieck groups.
\end{abstract}

\title{A category $\OC$ for oriented matroids}

\author{Ethan Kowalenko}\address{ Department of Mathematics, University of California, Riverside, CA 92521, USA}
\email{kowalenko@math.ucr.edu}

\author{Carl Mautner} \address{ Department of Mathematics, University of California, Riverside, CA 92521, USA}
\email{mautner@math.ucr.edu}


\maketitle


\section{Introduction}

In \cite{GDKD, BLPWtorico}, Braden--Licata--Proudfoot--Webster introduced a class of finite-dimensional algebras related to the combinatorics of hyperplane arrangements, whose representation theory is closely analogous to the integral blocks of Bernstein--Gelfand--Gelfand (BGG) category $\OC$.  Recall that BGG category $\OC$ plays an important role in Lie theory and can be described using the geometry of the Springer resolution.  Braden--Licata--Proudfoot--Webster discovered their algebras by analogy, motivated by the geometry of toric hyperk\"ahler (or hypertoric) varieties, but the algebras can be defined from basic linear algebra data.  The input for their definition was the data of a \textit{polarized arrangement} $\VC = (V,\eta,\xi)$, where $V \subset \RM^n$ is a $d$-dimensional linear subspace, $\eta \in \RM^n/V$ is a (generic) vector and $\xi\in V^*$ is a (generic) covector. 
 Braden--Licata--Proudfoot--Webster~\cite{BLPWgco} and others (e.g., \cite{Losev}) have since introduced and studied other such geometric categories $\OC$ associated to conical symplectic resolutions.

In this paper we extend the definition of Braden--Licata--Proudfoot--Webster in a different, more combinatorial direction: from the setting of polarized arrangements to the combinatorics of oriented matroids.  More precisely, the role of $V \subset \RM^n$ is replaced by a rank $d$ orientable matroid $M$ with parameter space $U$ and the role of $\eta$ and $\xi$ by an oriented matroid program $(\wt\MC,g,f)$ that extends and lifts an orientation of $M$.  One motivation for our work was a desire to categorify and better understand the matroidal Schur algebras of~\cite{BMMatroid,BMHyperRingel}.

To explain our results and motivation, we first recall in more detail the results of Braden--Licata--Proudfoot--Webster.

\subsection{Hypertoric category $\OC$}

In \cite{GDKD}, Braden--Licata--Proudfoot--Webster defined a quadratic algebra $A(\VC)$.  One motivation was a description of a regular block of category $\OC$ as arising from a quantization of the structure sheaf of $T^*(G/B)$, the cotangent bundle of a flag variety.  When $\VC$ is rational (meaning that $V,\eta,$ and $\xi$ are defined over $\QM$), one may associate to $\VC$ a hyperk\"ahler variety $\MG$\footnote{More precisely, the subspace $V$ and vector $\eta$ alone determine the variety $\mathfrak M$. The covector $\xi$ can be used to endow $\MG$ with a $\CM^*$-action.}, sometimes called a hypertoric variety, which behaves in various ways like the cotangent bundle of a flag variety.  Braden--Licata--Proudfoot--Webster show that in this case the category of representations of $A(\VC)$ is equivalent to that obtained by applying the same sort of quantization construction to the hypertoric variety for $\VC$.  Moreover they show:
\begin{thm}[Braden--Licata--Proudfoot--Webster]\label{BLPW}
Let $\VC = (V,\eta,\xi)$ be a polarized arrangement where $\eta$ and $\xi$ are generic.\footnote{See the following paragraph for the meaning of the word generic used here.}
\begin{enumerate}
\item The algebra $A(\VC)$ is quadratic with quadratic dual $A(\VC^\vee)$, where $\VC^\vee = (V^\perp,-\xi,-\eta)$ denotes the \textit{Gale dual} polarized arrangement.
\item The algebra $A(\VC)$ is quasi-hereditary.
\item The algebra $A(\VC)$ is Koszul (and thus by the first result, Koszul dual to $A(\VC^\vee)$).
\item Up to derived Morita equivalence, the algebra $A(\VC)$ depends only on $V \subset \RM^n$ and not on $\eta$ or $\xi$.
\end{enumerate}
\end{thm}

To give a feeling for the representation theory of these algebras, we will describe a labelling of the simple modules for $A(\VC)$.  It is convenient to consider the following hyperplane arrangement defined by $\VC= (V,\eta,\xi)$.  Note that $\eta \in \RM^n/V$ can be viewed as the affine subspace $\eta+V \subset \RM^n$, and we consider the arrangement of hyperplanes in $\eta+V$ cut out by the coordinate hyperplanes of $\RM^n$.
The genericity condition on $\eta$ is the requirement that the resulting arrangement be \emph{simple}, meaning that the nonempty intersection of $m$ hyperplanes has codimension $m$.  The covector $\xi \in V^*$ lifts to an affine linear functional on $\eta+V$.  The genericity condition on $\xi$ is the requirement that $\xi$ be nonconstant on any positive dimensional intersection of $V$ and a coordinate subspace.

\begin{ex}\label{ex-GDKD} The polarized arrangement from Example 2.2 of~\cite{GDKD} consists of a two-dimensional subspace $V \subset \RM^4$ together with some chosen $\eta$ and $\xi$. These choices produce the hyperplane arrangement depicted in Figure~\ref{fig-arrangement}(a).
\end{ex}

\begin{figure}
\begin{center}
\begin{tikzpicture}
 \def\x{1.6}
\draw (-1*\x,0) -- (2*\x,0);
\draw (-1*\x,.5*\x) -- (2*\x,.5*\x);
\draw (0,-1*\x) -- (0,2*\x);
\draw (-.5*\x,2*\x) -- (2*\x,-.5*\x);
\draw[red,thick] (-1*\x,1.6*\x) -- (2*\x,.1*\x);
\draw[->,red,thick] (-.8*\x,1.5*\x) -- (-.7*\x,1.7*\x);
\path (-1*\x,1.6*\x) node [left,red] {$\xi$};
\path (-.5*\x,-.5*\x) node {$\epsilon$};
\path (.75*\x,-.5*\x) node {$\delta$};
\path (-.5*\x,.25*\x) node {$\gamma$};
\path (.5*\x,.25*\x) node {$\beta$};
\path (.375*\x,.825*\x) node {$\alpha$};
\path (.5*\x,-1.5*\x) node {(a)};
\draw[fill=gray, opacity=.2] (-1*\x,.5*\x) -- (0,.5*\x) -- (0,1.5*\x) -- (2*\x,-.5*\x) -- (2*\x,-1*\x) -- (-1*\x,-1*\x) -- cycle;
\end{tikzpicture}\quad
\begin{tikzpicture}
  \def\radius{2.4cm}
  \draw[blue] circle[radius=\radius];
\path (0:\radius) coordinate (1+) node [circle,fill,inner sep=1pt]{};
\path (90:\radius) coordinate (2+) node [circle,fill,inner sep=1pt]{};
\path (135:\radius) coordinate (3+) node [circle,fill,inner sep=1pt]{};
\path (180:\radius) coordinate  (1-) node [circle,fill,inner sep=1pt]{};
\path (270:\radius) coordinate (2-) node [circle,fill,inner sep=1pt]{};
\path (315:\radius) coordinate (3-) node [circle,fill,inner sep=1pt]{};
\path (160:\radius) coordinate (f1) node [circle,fill,inner sep=1pt,red] {} node [left,red] {f};
\draw[->,red,thick] (f1) arc (160:152:\radius);
\path (250:\radius) coordinate (g) node [below,blue] {g};
\draw[->,blue,thick] (g) -- (250:.85*\radius);
\path (340:\radius) coordinate (f2) node [circle,fill,inner sep=1pt,red] {};
\draw[->,red,thick] (f2) arc (340:348:\radius);
\draw[red,thick,rotate=-20] (f2) arc [start angle=0, end angle=180, x radius=\radius, y radius=.7cm];

\draw (1-) arc [start angle=180, end angle=360, x radius=\radius, y radius=1cm];
\draw (1+) -- (1-);
\draw (2-) -- (2+);
\draw[rotate=-45] (3-) arc [start angle=0, end angle=180, x radius=\radius, y radius=1cm];

\draw[fill=gray, opacity=.2, rotate=-45] (3-) arc [radius = \radius, start angle =  0,  end angle = 112.5, x radius=\radius, y radius=1cm] -- (0,0) -- (1-) arc [radius = \radius, start angle =  225,  end angle = 360] --  cycle;

\path (240:.7*\radius) node {$\epsilon$};
\path (300:.7*\radius) node {$\delta$};
\path (215:.4*\radius) node {$\gamma$};
\path (-35:.4*\radius) node {$\beta$};
\path (45:.25*\radius) node {$\alpha$};
\path (270:1.35*\radius) node {(b)};
\end{tikzpicture}

\end{center}
\caption{Hyperplane arrangement and corresponding pseudosphere arrangement.}
\label{fig-arrangement}
\end{figure}
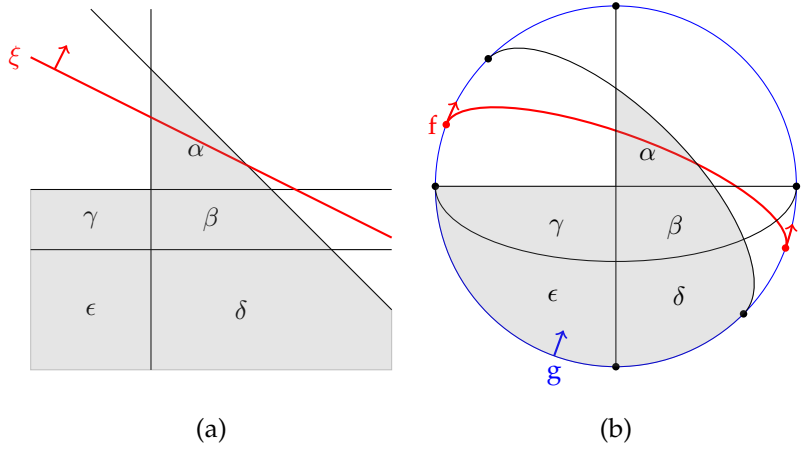

The set $\PC$ of chambers of the hyperplane arrangement in $\eta+V$ that are bounded with respect to $\xi$ parametrize the simple modules $\{L_\a\}_{\a \in \PC}$ for $A(\VC)$.  
  In Example~\ref{ex-GDKD}, we can label these chambers $\alpha,\beta,\gamma,\delta,\epsilon$ as in Figure~\ref{fig-arrangement}(a).  
  
  For each bounded chamber $\alpha \in \PC$, let $\beta \preceq \alpha$ if $\beta$ is contained in the cone generated by $\alpha$ originating from its maximal vertex.  The transitive closure of this relation gives the highest weight partial order on simple objects for the quasihereditary structure in the theorem.  In the example above, this produces the poset described by the following Hasse diagram:
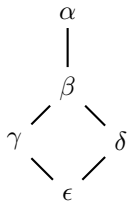
\begin{center}
\begin{tikzpicture}
    \node (a) at (0,0) {$\alpha$};
    \node [below of=a] (b) {$\beta$};
    \node [below left  of=b] (g)  {$\gamma$};
\node [below right of=b] (d) {$\delta$};
\node [below right of=g] (e) {$\epsilon$};

\draw [thick] (a) -- (b);
\draw [thick] (b) -- (g);
\draw [thick] (b) -- (d);
\draw [thick] (g) -- (e);
\draw [thick] (d) -- (e);
\end{tikzpicture}
\end{center}

More precisely, Braden--Licata--Proudfoot--Webster define standard modules $V_\a$ for every $\a \in \PC$ and prove (see the proof of~\cite[Theorem 5.23]{GDKD}):

\begin{thm}[Braden--Licata--Proudfoot--Webster]\label{thm-BLPWfilt}
For any $\a \in \PC$ the indecomposable projective cover $P_\a$ of $L_\a$ has a filtration with successive subquotients isomorphic to $V_\b$ for each $\b \succeq \a$ and each such standard module appears exactly once.
\end{thm}

\subsection{Matroidal setting}

Fix a field $k$ and a finite index set $E$.  In this paper we will begin with an orientable matroid $M$ of rank $d$ and a choice of parameter space $U \subset k^E$ for $M$.  By parameter space, we mean a subspace $U\subset k^E$ such that the composition $U\into k^E \fib \textrm{Span}\{t_i\mid i\in b\}$ is an isomorphism for any basis $b$ of $M$.

\begin{ex}\label{ex-polar1} Note that the subspace $V \subset \RM^n$ in a polarized arrangement of Braden--Licata--Proudfoot--Webster provides such a pair for $k=\RM$: let $M$ be the matroid on the index set $E=\{1,\ldots,n\}$ represented by the coordinate functions of $\RM^n$ restricted to $V$, viewed as vectors $x_1,\ldots,x_n \in V^*$, and let $U=V$.
\end{ex}

Let $\MC$ be an orientation of $M$, meaning an oriented matroid $\MC$ such that $\un\MC = M$, where $\un\MC$ denotes the underlying unoriented matroid.  (In the polarized arrangement example, there is a natural choice for $\MC$, as $M$ is represented by vectors in a real vector space.) 

The remaining input data we need is the structure of a oriented matroid program $\PB=(\wt\MC,g,f)$, meaning $\wt\MC$ is an oriented matroid on the underlying set $E \sqcup\{g,f\}$ such that $g$ is not a loop, $f$ is not a coloop, and $(\wt\MC\backslash f )/g = \MC$.  Like we did for $\eta$ and $\xi$, we ask that $g$ and $f$ be sufficiently generic (see Definiton~\ref{def-generic}).

The matroid $M$ is determined by $\PB$, so we can and will omit it from our notation and consider pairs $(\PB,U)$ where $\PB=(\wt\MC,g,f)$ is a sufficiently generic oriented matroid program and $U\subset k^E$ is a parameter space for the underlying (unoriented) matroid $M= \un{(\wt\MC\backslash f )/g}$.

\begin{ex}\label{ex-polar2}
Polarized arrangements give a natural class of examples.  For a $d$-dimensional polarized arrangement $(V,\eta,\xi)$, consider the $(d+1)$-dimensional subspace $\wt V$ of $\RM^n \times \RM_g \times \RM_f$ spanned by the graph of $\xi: V \to \RM_f$ and the vector $(\overline\eta,1,0) \in \RM^n \times \RM_g \times \RM_f$, where $\overline\eta$ is any representative of the coset $\eta \in \RM^n/V$.  Let $\wt\MC$ be the oriented matroid on the set $\{1,\ldots,n\} \sqcup \{g,f\}$ defined by the coordinate functions $x_1,\ldots,x_n,x_g,x_f \in \wt V^*$.
\end{ex}

Not every oriented matroid program $\PB$ comes from a polarized arrangement, but by the Topological Realization Theorem of Folkman--Lawrence, every loop-free program $\PB$ can be expressed as a \textit{pseudosphere arrangement} - a topological representation generalizing the notion of a hyperplane arrangement.

\begin{ex}
Figure~\ref{fig-arrangement}(b) shows the feasible region of the pseudosphere arrangement corresponding to the polarized arrangement from Example~\ref{ex-GDKD}.
\end{ex}

\begin{ex} \label{ex-ringel}
Figure~\ref{fig-ringel} depicts the feasible part of a pseudosphere arrangement, where $|E|=8$ and $M$ is the uniform rank 2 matroid on 8 points, that defines a non-realizable oriented matroid program $\PB=(\wt\MC,g,f)$.  Here the oriented submatroid $\wt\MC\backslash f$, a rank 3 oriented matroid on 9 points, is the non-stretchable simple arrangement of 9 pseudolines defined by Ringel~\cite{Ringel} as a perturbation the Pappus matroid.
\end{ex}

\begin{figure}
\begin{center}
\begin{tikzpicture}
  \def\radius{2.8cm}
  \draw[blue] circle[radius=\radius];
\path (0:\radius) coordinate (2+) node [circle,fill,inner sep=1pt]{};
\path (20:\radius) coordinate (7+) node [circle,fill,inner sep=1pt]{};
\path (40:\radius) coordinate (8+) node [circle,fill,inner sep=1pt]{};
\path (60:\radius) coordinate  (6+) node [circle,fill,inner sep=1pt]{};
\path (80:\radius) coordinate (5+) node [circle,fill,inner sep=1pt]{};
\path (100:\radius) coordinate (9+) node [circle,fill,inner sep=1pt]{};
\path (120:\radius) coordinate (4+) node [circle,fill,inner sep=1pt]{};
\path (140:\radius) coordinate (3+) node [circle,fill,inner sep=1pt]{};
\path (160:\radius) coordinate (f1) node [circle,fill,inner sep=1pt,red] {} node [left,red] {f};
\draw[->,red,thick] (f1) arc (160:152:\radius);
\path (180:\radius) coordinate (2-)  node [circle,fill,inner sep=1pt]{};
\path (200:\radius) coordinate (7-)  node [circle,fill,inner sep=1pt]{};
\path (220:\radius) coordinate (8-)  node [circle,fill,inner sep=1pt]{};
\path (240:\radius) coordinate (6-)  node [circle,fill,inner sep=1pt]{};
\path (250:\radius) coordinate (g) node [below,blue] {g};
\draw[->,blue,thick] (g) -- (250:2.4cm);
\path (260:\radius) coordinate (5-)  node [circle,fill,inner sep=1pt]{};
\path (280:\radius) coordinate (9-)  node [circle,fill,inner sep=1pt]{};
\path (300:\radius) coordinate (4-)  node [circle,fill,inner sep=1pt]{};
\path (320:\radius) coordinate (3-)  node [circle,fill,inner sep=1pt]{};
\path (340:\radius) coordinate (f2) node [circle,fill,inner sep=1pt,red] {};
\draw[->,red,thick] (f2) arc (340:348:\radius);
\draw[red] (f1) .. controls (125:1.4cm) and (-30:1.4cm) .. (f2);

\coordinate (a) at (10:2.1cm);
\coordinate (79) at (32.5:2.1cm);
\coordinate (89) at (55:2.1cm);
\coordinate (69) at (77.5:2.1cm);
\coordinate (59) at (100:2.1cm);
\coordinate (45) at (122.5:2.1cm);
\coordinate (b) at (145:2.1cm);
\coordinate (36) at (167.5:2.1cm);
\coordinate (26) at (190:2.1cm);
\coordinate (67) at (212.5:2.1cm);
\coordinate (68) at (235:2.1cm);
\coordinate (c) at (257.5:2.1cm);
\coordinate (25) at (280:2.1cm);
\coordinate (29) at (302.5:2.1cm);
\coordinate (24) at (325:2.1cm);
\coordinate (23) at (347.5:2.1cm);

\coordinate (39) at (0:1.4cm);
\coordinate (78) at (40:1.4cm);
\coordinate (46) at (80:1.4cm);
\coordinate (56) at (120:1.4cm);
\coordinate (37) at (160:1.4cm);
\coordinate (27) at (200:1.4cm);
\coordinate (28) at (240:1.4cm);
\coordinate (58) at (280:1.4cm);
\coordinate (49) at (320:1.4cm);

\coordinate (38) at (5:.7cm);
\coordinate (47) at (77:.7cm);
\coordinate (57) at (149:.7cm);
\coordinate (35) at (221:.7cm);
\coordinate (48) at (293:.7cm);

\draw [rounded corners] (2+) .. controls (23) .. (24) .. controls (29) .. (25) .. controls (28) .. (27) .. controls (26) .. (2-);
\draw [rounded corners] (7+) .. controls (79) .. (78) .. controls (47) .. (57) .. controls (37) .. (27) .. controls (67) .. (7-);
\draw [rounded corners] (8+) .. controls (89) .. (78) .. controls (38) .. (48) .. controls (58) .. (28) .. controls (68) .. (8-);
\draw [rounded corners] (6+) .. controls (69) .. (46) .. controls (56) .. (36) .. controls (26) .. (67) .. controls (68) .. (6-);
\draw [rounded corners] (5+) .. controls (59) .. (45) .. controls (56) .. (57) .. controls (35) .. (58) .. controls (25) .. (5-);
\draw [rounded corners] (9+) .. controls (59) .. (69) .. controls (89) .. (79) .. controls (39) .. (49) .. controls (29) .. (9-);
\draw [rounded corners] (4+) .. controls (45) .. (46) .. controls (47) .. (0,0) .. controls (48) .. (49) .. controls (24) .. (4-); 
\draw [rounded corners] (3+) .. controls (36) .. (37) .. controls (35) .. (0,0) .. controls (38) .. (39) .. controls (23) .. (3-); 

\draw[fill=gray, opacity=.2] (2-) arc [radius = \radius, start angle =  180,  end angle = 320] .. controls (23) .. (39) .. controls (79) .. (89) .. controls (69) .. (98:2.14cm) -- (45) -- (56) -- (36) -- (26) -- cycle;

\end{tikzpicture}

\end{center}
\caption{Ringel example}
\label{fig-ringel}
\end{figure}
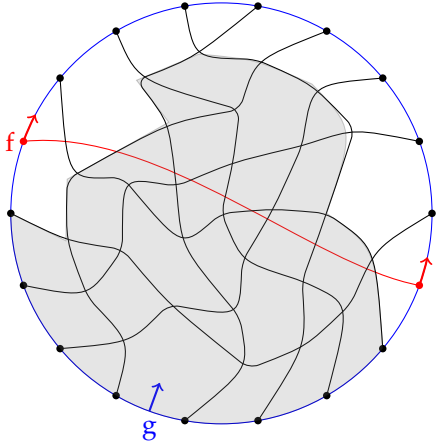

\begin{remark}
Every oriented matroid program $\PB$ where $d=2$ and $|E| \leq 7$ is realizable, so the program described in Example~\ref{ex-ringel} is a minimal non-realizable example.
\end{remark}

For a pair $(\PB=(\wt\MC,g,f),U)$ as above we define the dual pair $(\PB^\vee,U^\perp)$, where $\PB^\vee = (\wt\MC^\vee,f,g)$ is the dual oriented matroid program (here the roles of $f$ and $g$ are swapped),  and $U^\perp \subset k^E$ is the orthogonal complement.  

\begin{remark}
It is an exercise in linear algebra to check that when the oriented matroid program $\PB=(\wt\MC,g,f)$ comes from a polarized arrangement $\VC = (V,\eta,\xi)$ as in  Example~\ref{ex-polar2}, this duality agrees with the standard Gale duality of linear programming.  In other words, the dual program $\PB^\vee=(\wt\MC^\vee,f,g)$ is the oriented matroid program associated to the Gale dual polarized arrangement $\VC^\vee = (V^\perp,-\xi,-\eta)$.
\end{remark}

\subsection{Main results}

As above, let the pair $(\PB,U)$ consist of a sufficiently generic oriented matroid program $\PB=(\wt\MC,g,f)$ together with a parameter space $U\subset k^E$ for the (unoriented) matroid $M= \un{(\wt\MC\backslash f )/g}$.  Modifying the definition of Braden--Licata--Proudfoot--Webster to this setting, we introduce a finite-dimensional algebra $A(\PB,U)$  over $k$.  In particular, in the realizable case of Example~\ref{ex-polar1} and~\ref{ex-polar2}, one recovers the original algebra $A(\PB,V) = A(\VC)$.

In the more general setting, we show that part (1) of Theorem~\ref{BLPW} extends without modification:

\begin{thm}\label{thm-quad-intro}
Let $(\PB,U)$ be a pair as above. The algebra $A(\PB,U)$ is quadratic with quadratic dual $A(\PB^\vee,U^\perp)$ corresponding to the dual pair.
\end{thm}

Similarly to the realizable case, the simple modules for $A(\PB,U)$ are labelled by the set $\PC$ of bounded, feasible topes. For example, in Example~\ref{ex-ringel} the bounded, feasible topes correspond to shaded regions in  Figure~\ref{fig-ringel}.  Again one can define a cone relation on $\PC$ and standard modules $V_\a$ for each $\a \in \PC$.  

However, unlike in the realizable case, the transitive closure of the cone relation need not define a poset.  An oriented matroid program $\PB = (\wt\MC,g,f)$ is said to be \textit{Euclidean} if the transitive closure of the cone relation on bounded, feasible topes of $\PB$ is a poset.

Using this condition, we obtain the following analogue of Theorems~\ref{BLPW}(2), \ref{BLPW}(3) and~\ref{thm-BLPWfilt}.

\begin{thm}\label{thm-qherKos}
For a pair $(\PB,U)$ as above with the additional assumption that the program $\PB$ is \textbf{Euclidean}, the algebra $A(\PB,U)$ is quasi-hereditary and Koszul.

Moreover, for any $\a \in \PC$ the indecomposable projective cover $P_\a$ of $L_\a$ has a filtration with successive subquotients isomorphic to $V_\b$ for each $\b \succeq \a$ and each such standard module appears exactly once.
\end{thm}

\begin{remark}\label{rem-Euclid}
 While oriented matroid programs are not always Euclidean, every oriented matroid program of rank at most 3 (equivalently $d$ at most 2) is Euclidean.  Thus there are plenty of Euclidean, non-realizable programs, such as Example~\ref{ex-ringel}.

We do not know whether or not every non-realizable oriented matroid $\MC$ admits a Euclidean program $\PB=(\wt\MC,g,f)$ such that $\wt\MC/g\backslash f=\MC$.  For connections to a well-known conjecture of Las Vergnas, see the discussion surrounding Proposition~\ref{prop-lasvergnas}.
\end{remark}

We observe in Example~\ref{ex-EFM(8)} that in the non-Euclidean case, $A(\PB,U)$ need not be quasi-hereditary.  In particular, we give an example of a non-Euclidean program $\PB$ and projective $A(\PB,U)$-module which does not admit a standard filtration.

However, in Theorem~\ref{thm-Kgroup} we do prove that for any oriented matroid program $\PB$ the following analogue of Theorem~\ref{thm-BLPWfilt} holds on the level of the Grothendieck group of graded $A(\PB,U)$-modules.

\begin{thm}\label{thm-Kgroup-intro}
For any generic oriented matroid program $\PB$ and any $\a \in \PC$, the class of the indecomposable projective $P_\a$ in the Grothendieck group can be expressed as the sum:
\[ [P_\a] = \sum_{\g \succeq \a} q^{d_{\a\g}}[V_\g],\]
where $d_{\a\g}$ denotes the distance between the topes $\a$ and $\g$.
\end{thm}

While our proof of Koszulity in the Euclidean case relies on $A(\PB,U)$ being quasi-hereditary, it is conceivable that $A(\PB,U)$ is Koszul more generally.  As evidence in this direction, in Theorem~\ref{thm-numKoszul} we prove the Hilbert series of $A(\PB,U)$ and $A(\PB^\vee,U^\perp)$ satisfy the numerical identity discussed in~\cite[Lemma 2.11.1]{BGS}.

\subsection{Derived Morita equivalence}\label{subsec-morita} 
In light of Theorem~\ref{BLPW}(4) it seems natural to ask:

\begin{question}\label{q-Morita}
Let $M$ be an orientable matroid and $U$ a choice of parameter space for $M$.  For any two orientations $\MC_1, \MC_2$ of $M$ and generic oriented matroid programs $\PB_1 = (\wt\MC_1,g_1,f_1), \PB_2=(\wt\MC_2,g_2,f_2)$ such that $\wt\MC_\ell/g_\ell \backslash f_\ell=\MC_\ell,~\ell=1,2$, are the algebras
$A(\PB_1,U)$ and $A(\PB_2,U)$ derived Morita equivalent?
\end{question}

If the answer to this question is yes, it would appear to give a rather interesting algebraic invariant of the matroid $M$.  Or weaker, one might still hope for an affirmative answer under the assumption that $\MC_1=\MC_2$:

\begin{question}\label{q-MoritaOr}
Let $\MC$ be an oriented matroid and $U$ a choice of parameter space for $M=\un\MC$.  For any two  generic oriented matroid programs $\PB_1 = (\wt\MC_1,g_1,f_1), \PB_2=(\wt\MC_2,g_2,f_2)$ such that $\MC=\wt\MC_1/g_1 \backslash f_1=\wt\MC_2/g_2 \backslash f_2$, are the algebras
$A(\PB_1,U)$ and $A(\PB_2,U)$ derived Morita equivalent?
\end{question}

If the answer to one or both of these questions is no, the number of derived Morita equivalence classes could also provide a interesting invariant of $M$ or $\MC$.

As a partial result in this direction, following the strategy of Braden--Licata--Proudfoot--Webster, we prove the following theorem in Section~\ref{sec-derived}.

\begin{thm}\label{thm-derivedequiv}
Fix $\MC$ and let $\PB_1=(\wt\MC_{1},g_1,f_1)$ and $\PB_2=(\wt\MC_{2},g_2,f_2)$ be Euclidean such that $\wt\MC_i/g_i \backslash f_i=\MC$ for $i=1,2$.  Suppose in addition that the oriented matroid program $\PB_{\text{mid}}=(\wt\MC_{\text{mid}},g_2,f_1)$ is also Euclidean, where $\PB_{\text{mid}}$ is a generic oriented matroid program\footnote{Note that such an oriented matroid program $(\wt\MC_{\text{mid}},g_2,f_1)$ always exists and there will typically be many such oriented matroid programs.  However the particular choice will not matter for us, because, as mentioned in Remark~\ref{rem-eft}, all of our constructions depend only on the contraction and restriction oriented matroids $\wt\MC_{\text{mid}}/g_2$ and $\wt\MC_{\text{mid}}\backslash f_1$.}
 such that $$\wt\MC_{\text{mid}}/g_2=\wt\MC_1/g_1, \quad \wt\MC_{\text{mid}} \backslash f_1=\wt\MC_2 \backslash f_2.$$
Then the bounded derived categories of graded finitely generated $A(\PB_1,U)$- and $A(\PB_2,U)$-modules are equivalent.
\end{thm} 

This allow us to answer Questions~\ref{q-Morita} and~\ref{q-MoritaOr} in some simple cases.

\begin{cor}\label{cor-rank2}
Question~\ref{q-MoritaOr} has an affirmative answer for any oriented matroid $\MC$ of rank 2.
\end{cor}

\begin{proof}
Recall from Remark~\ref{rem-Euclid}, that any oriented matroid program of rank 3 is Euclidean, so for any $\PB_1$ and $\PB_2$, the three oriented matroid programs $\PB_1,\PB_2$ and $\PB_{\text{mid}}$ are all Euclidean and the result follows from Theorem~\ref{thm-derivedequiv}.
\end{proof}

\begin{cor}
Question~\ref{q-Morita} has an affirmative answer for $M=U_{2,n}$, the uniform matroid of rank 2 defined on a set $E$ of $n\geq2$ elements.
\end{cor}

\begin{proof}
By Corollary~\ref{cor-rank2}, it suffices to show that for any two orientations $\MC_1$ and $\MC_2$ of $U_{2,n}$, there are generic oriented matroid programs $\PB_1 = (\wt\MC_1,g_1,f_1), \PB_2=(\wt\MC_2,g_2,f_2)$ such that $\MC_1=\wt\MC_1/g_1 \backslash f_1$ and $\MC_2=\wt\MC_2/g_2 \backslash f_2$, for which $A(\PB_1,U)$ and $A(\PB_2,U)$ are derived equivalent.

But any two orientations of $U_{2,n}$ are related by a relabeling of $E$ and reorientation.  Note that relabelling and reorientation each induce a canonical isomorphism between the associated algebras.
\end{proof}

The same sort of argument gives a handful of similar examples.

\subsection{Matroidal Schur algebras}

Motivated in part by~\cite{GDKD}, Braden and the second author defined a \textit{hypertoric Schur algebra}~\cite{BMHyperRingel} - an analogue of the Schur algebra associated to affine hypertoric varieties.  Recall that one can construct an affine hypertoric variety $\MG_0$ with the input of a rational subspace $V\subset \RM^n$.  In this setting the resulting hypertoric Schur algebra $R(V)$ can be interpreted as a convolution algebra for a union of resolutions of stratum slices of $\MG_0$.  In particular, for a rational polarized arrangement $(V,\eta,\xi)$ with the same underlying subspace $V \subset \RM^n$, the convolution algebra for the resolution $\MG \to \MG_0$ is a subalgebra of the associated hypertoric Schur algebra.  Braden--Proudfoot--Webster showed in~\cite[Proposition 6.16, Example 6.18]{BPW1} that the convolution algebra of the resolution $\MG \to \MG_0$ is categorified by Harish-Chandra bimodules for hypertoric category $\OC$.  One expects the entire hypertoric Schur algebra to be similarly categorified by Harish-Chandra bimodules with more general support and similarly to obtain a natural $q$-deformation of the hypertoric Schur algebra, or $q$-hypertoric Schur algebra.

In~\cite{BMMatroid}, Braden and the second author observed that the hypertoric Schur algebras studied in~\cite{BMHyperRingel} can be defined in terms of the underlying matroid.  Following this observation, they defined a \textit{matroidal Schur algebra} $R(M)$ associated to any  matroid $M$.

One motivation for defining the category $\OC$ for oriented matroid programs in the present paper was to provide the foundation to categorify and find natural $q$-deformations of matroidal Schur algebras for orientable, but non-realizable matroids using an appropriate category of Harish-Chandra bimodules.

\subsection{Outline of paper}
In Section~\ref{sec-comb} we describe the combinatorial set-up of oriented matroid programs and parameter spaces.  In Section~\ref{sec-defA} we define the algebra $A(\PB,U)$ and in Section~\ref{sec-quaddual} we prove Theorem~\ref{thm-quad-intro} (Lemma~\ref{lem-quad} and Theorem~\ref{thm-quad}). In Section~\ref{sec-defB} we define the algebra $B(\PB,U)$ and prove Theorem~\ref{thm-A-B}, which is a key ingredient in the proof of Theorem~\ref{thm-qherKos}. Section~\ref{section-center} develops more topology, resulting in a nice description of the center of $B(\PB,U)$. 

In Section~\ref{sec-modules}, we study the category of finitely-generated (right) $A(\PB,U)$-modules and prove Theorem~\ref{thm-qherKos}. In particular, under the Euclidean assumption, we show $A(\PB,U)$ is quasihereditary (Theorem~\ref{thm-projfilt}) and Koszul (Theorem~\ref{thm-Koszul}). In the non-Euclidean setting, we prove Theorem~\ref{thm-Kgroup-intro} (Theorem~\ref{thm-Kgroup}), prove the Koszulity condition on Hilbert series (Theorem~\ref{thm-numKoszul}) and show that Theorem~\ref{thm-projfilt} requires the Euclidean assumption (Example~\ref{ex-EFM(8)}). In Section~\ref{sec-derived}, we study the derived categories of graded finitely-generated $A(\PB,U)$-modules for varying Euclidean $\PB$ and a fixed $\MC$ and prove Theorem~\ref{thm-derivedequiv}.

\subsection{Acknowledgements}
We would especially like to thank Tom, Tony, Nick and Ben, whose paper~\cite{GDKD} was a major source of inspiration for us and forms the foundation of this paper.

The second author is grateful to Jens Eberhardt for introducing him to the world of oriented matroids.  Thanks to Federico Ardila who put us in touch with Jim Lawrence and thanks to Jim for answering our silly questions.  The second author is also grateful to Catharina Stroppel for helpful conversations about finite-dimensional algebras.

The second author was supported in part by a Simons Foundation Collaboration Grant and NSF grant DMS-1802299.  He would also like to thank the Mathematics department at Dartmouth College for its hospitality.

\section{Combinatorial Setup}
\label{sec-comb}

In this section we briefly introduce the notation we will need to work with oriented matroids, but assume some familiarity with the basic notions.  To the uninitiated reader, we recommend \cite{OMbook} (particularly the first chapter) for a survey and as a useful reference.

\subsection{Generic oriented matroid programs}

For an index set $I$ and any function $Z\colon I\to \{0,+,-\}$ , let $\un Z:=\{i\mid Z(i)\neq0\}\subset I$ be the \emph{support} of $Z$ and let $z(Z):=I\backslash\un Z$ be the \emph{zero set} of $Z$.

Let $M$ be an orientable matroid of rank $d$ on the finite set $E$.  Let $\MC$ be an oriented matroid such that $\un\MC=M$ is its underlying unoriented matroid.

Let $\BM$ denote the set of bases for $M$.   We let $C=C(\MC)$ denote the set of signed circuits and $C^*=C^*(\MC)$ the set of signed cocircuits, both regarded as subsets of the set of functions $E \to\{0,+,-\}$. Note the unoriented matroid $M=\un\MC$ has circuits $\{\un X\mid X\in C\}$ and cocircuits $\{\un Y\mid Y\in C^*\}$. The dual oriented matroid $\MC^\vee$ is given by switching the roles of the circuits and cocircuits (i.e. $C(\MC^\vee)=C^*(\MC)$ and $C^*(\MC^\vee)=C(\MC)$), while the bases $\BM^\vee$ of the underlying matroid $M^\vee = \un\MC^\vee$ are the complements in $E$ of the elements of $\BM$.

Let $S\subset E$. 
Then the set 
  \[\{X\in C(\MC)\mid \underline{X}\subset E\backslash S\}\]
is the set of circuits of an oriented matroid $\MC\backslash S$ on $E\backslash S$, called the \emph{deletion} of $S$ from $\MC$.
The set 
  \[\{X|_S:S\to\{0,+,-\}\mid X\in C(\MC)\text{ and }\un{X}\cap S\text{ is inclusion minimal and nonempty}\}\]
gives the set of circuits of an oriented matroid $\MC/(E\backslash S)$ on $S$, called the \emph{contraction} of $\MC$ to $S$. Duality exchanges contraction and deletion:
  \[(\MC/S)^\vee=\MC^\vee\backslash S \quad \text{and}\quad (\MC\backslash S)^\vee=\MC^\vee/S.\]

An element $i\in E$ is a \textit{loop} of $\MC$ if $\{i\}$ is the support of a circuit of $\MC$.  Dually, $i\in E$ is a \textit{coloop} of $\MC$ if $i$ is not contained in the support of any circuit of $\MC$.

An \textit{oriented matroid program} $\PB = (\wt\MC,g,f)$ is an oriented matroid $\wt\MC$ on the set $E\sqcup \{f,g\}$ such that $(\wt\MC\backslash f )/g = \MC$, $g$ is not a loop and $f$ is not a coloop.  In particular, the rank of $\wt\MC$ is $d+1$, and this is defined to be the rank of $\PB$.

The deletion $\NC = \wt\MC \backslash f$ of $f$ from $\wt\MC$ is called the corresponding \textit{affine oriented matroid}.

\begin{remark}\label{rem-eft}
Our constructions will only depend on the contraction $\wt\MC /g$, which is a single element extension of $\MC$ on $E \sqcup \{f\}$, and the deletion $\wt\MC \backslash f$, which is a single element lift of $\MC$ on $E \sqcup \{g\}$.  Thus for our purposes it would be more natural to define an oriented matroid program as a pair, which we have taken to affectionately calling an \textit{eft}, of a single element extension and single element lift of $\MC$.  We will refrain from doing so in this paper as the original notion appears to be standard in the field.
\end{remark}

\begin{defi}\label{def-generic} We say that the oriented matroid program $\PB = (\wt\MC,g,f)$ is \textit{generic} for $\MC$ if 
\begin{enumerate}
\item for any cocircuit $Y$ of $\NC = \wt\MC\backslash f$, if $|z(Y)|>d$, then $Y(g)=0$, and
\item for any circuit $X$ of $\wt\MC/g$, if $|z(X)|>n-d$, then $X(f)=0$.
\end{enumerate} 
\end{defi}

\begin{remark}\label{rem-equality}
As the rank of the oriented matroid $\wt\MC/g$ on $E \sqcup \{f\}$ is $d$, for any circuit $X$ of $\wt\MC/g$, $|\un{X}|\leq d+1$ and so $X$ has \textit{at least} $n-d$ zero entries.  In the case of equality, $\un{X}$ contains a basis of $\un{\wt\MC/g}$ and so $z(X)$ is independent in $(\un{\wt\MC/g})^\vee$.  

Dually, for any cocircuit $Y$ of $\NC$, $|z(Y)|\geq d$ and if equality holds $z(Y)$ is independent in $\un\NC$.
\end{remark}

\begin{ex} \label{ex-linear}
As explained in Example~\ref{ex-polar2}, every polarized arrangement $\VC=(V,\eta,\xi)$ naturally gives rise to an oriented matroid program $\PB$.  If $\eta$ and $\xi$ are generic in the sense of Theorem~\ref{BLPW}, then $\PB$ is generic as well.
\end{ex}

We now deduce some simple consequences of genericity.

\begin{lem}\label{lem-loops}
Suppose $\PB = (\wt\MC,g,f)$ is generic.  Then $\NC = \wt\MC\backslash f$ has no loops and $\wt\MC/g$ has no coloops.  
\end{lem}

\begin{proof}
We prove the first statement and the second follows by duality.  By our assumption in the definition of an oriented matroid program, $g$ is not a loop of $\NC$, so $g \in \un{Y}$ for some cocircuit $Y$ of $\NC$.  By Definition~\ref{def-generic}(1), $|z(Y)|=d$ and so Remark~\ref{rem-equality} implies $z(Y)$ is independent.  If there were a loop $i$ of $\NC$, then $i \in z(Y)$ contradicting the fact that $z(Y)$ is independent.
\end{proof}

\begin{lem} \label{lem-BasFCocirc}
Assume $\PB$ is generic.  Then there is a natural bijection between the set $\BM$ of bases for $\un\MC$ and the set of feasible cocircuits for $\NC = \wt\MC\backslash f$.
\end{lem}

\begin{proof}
Consider the map that takes a feasible cocircuit $Y$ for $\NC$ to its zero set $b:=z(Y)$.
As $Y$ is feasible, $Y(g)=+$ and so by condition (1), $Y$ must have $d$ zero entries.  Then $Y$ has $n+1-d$ non-zero entries and is a circuit of $\NC^\vee$ (which has rank $n-d$), so any subset of $Y$ of size $n-d$ is a basis for $\un\NC^\vee$.  In particular, $\un Y \backslash \{g\}$ is a basis for $\un\NC^\vee$, so its complement $b \sqcup \{g\}$ is basis for $\un\NC$.  Thus $b$ is a basis for $\un\MC = \un\NC/g$.

To show that this is a bijection, suppose $b$ is a basis for $\un\MC$.  Then $b \sqcup \{g\}$ is a basis for $\un\NC$, its complement $E \backslash b$ is a basis for $\un\NC^\vee$ and so $(E \backslash b)\cup g$ must contain a cocircuit $Y$ for $\NC$.  By condition (1), either $\un Y=(E \backslash b)\cup g$ or $\un Y \subset E \backslash b$.  But the latter is not possible as $E \backslash b$ is a basis for $\un\NC^\vee$.  We conclude that there is a unique choice of feasible cocircuit $Y$ with support $\un Y = (E \backslash b)\cup g$.
\end{proof}

For $b\in \BM$, we let $Y_b$ be denote the corresponding feasible cocircuit.

\medskip

We will often use three constructions to obtain new generic oriented matroid programs from a generic oriented matroid program $\PB$: duality, deletion, and contraction.
Recall that duality for oriented matroid programs takes the program $\PB=(\wt\MC,g,f)$ to the program $\PB^\vee = (\wt\MC^\vee,f,g)$ with underlying oriented matroid 
 \eq{(\wt\MC^\vee \backslash g)/f = ((\wt\MC /g)\backslash f)^\vee = \MC^\vee.}  
For any $S\subset E$, we denote the contraction and deletion of $S$ by
 \eq{\PB/S:=(\wt\MC/S,g,f) \quad\text{and}\quad \PB\backslash S:=(\wt\MC\backslash S,g,f),}
respectively. 

Note that $\PB$ is generic if and only if $\PB^\vee$ is generic. If $\PB$ is generic and $S\subset b$ for some $b\in\BM$, then $\PB/S$ is generic and has rank $d+1-|S|$. If $\PB$ is generic and $S\cap b=\varnothing$ for some $b\in\BM$, then $\PB\backslash S$ is also generic of rank $d+1$.

\begin{lem} For any oriented matroid $\MC$ there exists a generic oriented matroid program $\PB=(\wt\MC,g,f)$ such that $(\wt\MC/g)\backslash f = \MC$.  
\end{lem}

\begin{proof}
For example, for any order on $E$, consider the lexicographic extension $\MC' = \MC[E]$ by a point $f$ with respect to this order (Note that this is the same as taking the extension $\MC[b_{\min}]$ where $b_{\min}$ is a lexicographically minimal basis of $\MC$).  By a Lemma of Todd~\cite[Lemma 7.2.6]{OMbook}, any circuit $X$ of $\MC'$ with more than $n-d$ zero entries satisfies $X(f)=0$.  It then remains to define $\wt\MC$ as a single element of lifting of $\MC'$ by a point $g$, such that $\NC = \wt\MC\backslash f$ satisfies property (1) above.  This can be done via the dual construction: consider the \textit{colocalization} $\tau : C(\MC') \to \{+,-,0\}$ defined for any $X \in C(\MC')$ by $\tau(X)=X_i$, where $i$ is the minimal element of $E$ such that $X_i \neq 0$.  Let $\wt\MC$ be the lexicographic lifting of $\MC'$ defined by $\tau$ (in other words the dual of the lexicographic extension of $(\MC')^\vee$ associated to $\tau$).
\end{proof}

For the rest of the paper we assume that $\PB$ is generic.

\subsection{Bounded feasible topes and sign vectors}\label{subsec-P}
In this section we recall the notions of bounded and feasible topes and show in Corollary~\ref{cor-bij} that when $\PB$ is generic there is a natural bijection between bases $\BM$ of $M$ and bounded feasible topes for $\PB$. 

Let $I$ be any index set. For any functions $Z,Z'\colon I\to\{0,+,-\}$, their \emph{composition} $Z\circ Z'\colon I\to\{0,+,-\}$ is defined by
  \[ Z\circ Z'(i)=\begin{cases} Z(i) & \text{if}\; Z(i)\neq0\\ Z'(i) &\text{otherwise}\end{cases} \]
We say that $Z$ is a \emph{face} of $Z'$ if $Z\circ Z'=Z'$.

The nonzero \emph{covectors} of an oriented matroid on the set $E$ are the functions $E\to \{0,+,-\}$ which can be written as the composition of cocircuits. The set of covectors of $\MC$ is denoted by $\LC(\MC)$, and includes the zero function $0$. It has a natural poset structure defined by $Y \leq X$ if $Y$ is a face of $X$. 
The poset $\LC(\MC)$ is pure with minimal element $0$, the zero function. 
The maximal elements of $\LC(\MC)$ are called \emph{topes}, while the minimal elements of $\LC(\MC)\backslash \{0\}$ are the cocircuits of $\MC$.

The rank of $Y\in\LC(\MC)$ is given as $\r(Y)=d-r(z(Y))$, where $r$ is the rank function of the underlying matroid $\un\MC$. For $Y_1,Y_2 \in \LC(\MC)$, the \emph{join} $Y_1\vee Y_2$ is the minimal covector that has both $Y_1$ and $Y_2$ as faces, which only exists if there is a tope $T$ with both $Y_1$ and $Y_2$ as faces. The \emph{meet} $Y_1\wedge Y_2$ is the maximal covector face of both $Y_1$ and $Y_2$. Note that the meet of $Y_1,Y_2\in\LC(\MC)$ always exists, but is the zero function when $Y_1$ and $Y_2$ do not have a common cocircuit face.

\begin{defi}\label{def-covect} (Feasible covectors and affine space)
Let $\PB=(\wt\MC,g,f)$ be a generic program and let $\LC=\LC(\NC)$ denote the set of covectors for $\NC = \wt\MC\backslash f$.  The \textit{affine space} of $\PB$ is:
\[\AC = \{Y \in \LC ~ | ~ Y(g) = +\}.\]
We call elements of $\AC$ \textit{feasible} covectors.

We say that the \textit{boundary of affine space} is:
\[ \AC^\infty = \{Y \in \LC ~|~ Y(g) = 0 \} \]
\end{defi}

Notice that $\AC^\infty$ defines an oriented matroid on $E \cup \{g\}$ which is equal to $\MC$ with $g$ adjoined as a loop. Also, notice that the join of covectors in $\AC$ is also in $\AC$ if it exists, while their meet is in $\AC$ if and only if they share a common cocircuit face in $\AC$; Otherwise, their meet is in $\AC^\infty$.

\begin{defi} (Feasible topes)
Let $\TC(\LC)$ denote the set of topes of $\NC$.  We let \[\FC = \AC \cap \TC(\LC) \] denote the set of feasible topes. 
\end{defi}

Notice that the definition of feasible topes does not depend on $f$. 

\begin{remark}
By Lemma~\ref{lem-loops} the topes of $\NC$ are the covectors $T$ such that $z(T)=\varnothing$.   
\end{remark}

A \textit{sign vector} is a function $\a\colon E\to \{+,-\}$, usually written as $\a\in\{+,-\}^E$. There is an obvious injective map from $\FC$ to $\{+,-\}^E$ given by forgetting the value at $g$ (which is always $+$).  We may refer to the sign vectors in the image as \textit{feasible} sign vectors, and in a slight abuse of notation we identify feasible topes with the corresponding sign vectors.  When there is a risk of confusion, we will write $T_\a$ to denote the feasible tope of $\NC$ corresponding to a sign vector $\a$.

\begin{defi} (Directions and optimality)

\begin{enumerate}
\item We refer to covectors in the boundary of affine space $Z \in \AC^\infty$ as \textit{directions} in $\AC$.  We say that a direction is \textit{increasing} (resp. decreasing or constant) with respect to $f$ if $Z(f)=+$ (resp. $Z(f) = -$ or $Z(f)=0$).
\item For a feasible tope $T \in \FC$ and a feasible covector face $Y$ of  $T$, we say that the direction $Z \in \AC^\infty$ is \textit{feasible} for $Y$ in $T$ if $Y \circ Z$ is a face of $T$.
\item A feasible covector $Y$ that is a face of $T \in \FC$ is an \textit{optimal solution} for $T$ if there is no feasible increasing direction for $Y$ in $T$.
\end{enumerate}
\end{defi}

\begin{defi} (Bounded sign vectors)
A sign vector $\a\in\{+,-\}^E$ is \textit{unbounded} if there exists a increasing direction $Z\in\AC^\infty$ such that $Z|_{E}\circ \a=\a$.
If no such $Z$ exists, we say $\a$ is \emph{bounded}.  Similarly, a tope $T$ is bounded if the sign vector $T|_{E}$ is bounded.

Let $\BC$ denote the set of bounded sign vectors and $\PC = \FC \cap \BC$ denote the set of bounded and feasible sign vectors.
\end{defi}

\begin{remark} Let $Y$ be a bounded, feasible tope.  Note that if $Y_1,Y_2$ are optimal solutions for $Y$, then so is $Y_1 \circ Y_2$.  If $Y_1$ is an optimal solution and $Y_2 \in \AC$ is a face of $Y_1$, then $Y_2$ is also optimal.  It follows that if $Y$ has an optimal solution, then it has an optimal cocircuit.
\end{remark}

\begin{thm}\label{thm-optimal}
Assume $\PB$ is generic. Then every feasible, bounded tope $Y$ has a unique optimal solution (cocircuit) and the resulting map from $\PC$ to the set of feasible cocircuits for $\NC$ is a bijection.
\end{thm}

\begin{proof}
Let $Y \in \PC$.  Recall that an optimal solution exists, without any condition on $\PB$, by Bland and Lawrence's Main Theorem of Oriented Matroid Programming (see~\cite[10.1.13]{OMbook} for a survey).  

We first show that if $\PB$ is generic, then such a solution is unique. 

Suppose that $Y$ has two distinct optimal cocircuits $Y_1$ and $Y_2$.  Then $Y_1 \circ Y_2 \in \AC$ must also be an optimal solution.  Replacing $Y_2$ by another cocircuit face of $Y_1 \circ Y_2$ if necessary, without loss of generality, we may assume that $Y_1$ and $Y_2$ are joined by an edge (i.e., the rank of $Y_1 \circ Y_2$ is 2).  There then exist two cocircuits at the boundary $\pm Z \in \AC^\infty$ on the pseudoline $\overline{Y_1Y_2}$ such that $Y_1 \circ Y_2 = Y_1 \circ Z$ and $Y_1 \circ Y_2 = Y_2 \circ -Z$. (The cocircuit $Z$ can by obtained via elimination of $g$ from the pair $Y_1$, $Y_2$ and this elimination is unique up to sign as $Y_1,Y_2$ form a modular pair).  By optimality of $Y_1$ and $Y_2$, we conclude that $\pm Z$ must be constant directions.  Note that $z(Y_1\circ Y_2)$ is an independent set in $\un\MC$ of cardinality $d-1$. The contraction $\wt \MC/(\{g\}\cup z(Y_1\circ Y_2))$ is a rank 1 oriented matroid where $f$ is a loop since the cocircuits are both zero on $f$. Thus $z(Y_1\circ Y_2) \cup \{f\}$ contains a circuit $X$ of $\wt\MC/g$ such that $X(f)\neq0$.  This contradicts condition (2) since $X$ is zero on at least $n-d+1$ entries.

It remains to show that the map from $\PC$ to feasible cocircuits is a bijection.  Given a cocircuit $Y \in \AC$, we would like to show that $Y$ is the optimal solution for a unique tope $T \in \PC$. We construct such a $T$ as follows. For any $i\in z(Y)$, we know that $\PB_i:=\PB/(z(Y)\backslash\{i\})$ is a generic program whose affine space is one-dimensional. There is a unique cocircuit $\overline{Z}_i$ in $\PB_i$ such that $\overline{Z}_i(f)=-$, and this cocircuit is the restriction of a cocircuit $Z_i$ in $\wt\MC/g$. Then $Z_i(j)=0$ for $j\in z(Y)\backslash\{i\}$, while $Z_i(i)\neq 0$ and $Z_i(f)=-$.
Then $T$ is defined to be the composition $Y\circ Z$, where $Z$ is the composition of all $Z_i$, $i\in z(Y)$, taken in any order. 

This $T$ is feasible since $T(g)=Y(g)=+$, and unique since it agrees on $z(Y)$ with the unique bounded feasible tope of $\PB\backslash(E\backslash z(Y))$. To show that $T$ is bounded, recall that an equivalent definition for a feasible tope to be bounded is that it must be in the \emph{bounded cone} of some $b\in\BM$ (see \cite[Definition 10.1.8.ii, Corollary 10.1.10.ii]{OMbook}), meaning that it agrees on $b$ with a bounded tope of $\PB\backslash(E\backslash b)$ (see also Definition \ref{def-bounded cone}). Since $T$ is in the bounded cone of $z(Y)\in\BM$, we have that $T$ is bounded.
\end{proof}

Combining the bijections of Lemma~\ref{lem-BasFCocirc} and Theorem~\ref{thm-optimal} we obtain our desired correspondence.

\begin{cor}\label{cor-bij}
There is a natural bijection between the set $\BM$ of bases for $\un\MC$ and the set $\PC$ of bounded feasible topes:
$$ \mu : \BM \to \PC,$$
which takes a basis $b$ to the tope whose optimal cocircuit is $Y_b$.  (Recall that $Y_b$ is the feasible cocircuit with $z(Y_b)=b$.)
\end{cor}

We conclude this section with a discussion of the effect of duality on the bijection $\mu$.

Recall the following result about duality for oriented matroid programs.

\begin{prop}\cite[Corollary 10.1.11]{OMbook}\label{prop-FBP}
Let $\FC^\vee$, $\BC^\vee$, and $\PC^\vee$  respectively denote the sets of feasible, bounded, and bounded feasible sign vectors for the dual program $\PB^\vee=(\wt\MC^\vee,f,g)$. Then 
  \[\FC^\vee=\BC \quad\text{and}\quad \BC^\vee=\FC, \quad \text{and~so} \quad \PC^\vee=\PC.\]
\end{prop}

Let $\BM^\vee$ denote the set of bases of $\un\MC^\vee$.  Then $b \mapsto b^c :=E\backslash b$ defines a bijection $\BM \to \BM^\vee$.  Let $\mu^\vee: \BM^\vee \to \PC^\vee$ be the bijection for the dual program $\PB^\vee$ defined as above. Recall that $\PC^\vee = \PC$.

\begin{prop}[Complementary Slackness]\label{prop-comp}
For any $b\in\BM$, $\mu(b) = \mu^\vee(b^c)$.
\end{prop}

\begin{proof}
Recall that $\PB$ is generic if and only if $\PB^\vee$ is generic. This is then the ``Complementary Slackness'' theorem of Bland applied to generic programs (cf. \cite[Theorem 10.1.12]{OMbook}).
\end{proof}

\subsection{Cone relation and Euclidean oriented matroid programs}
\label{subsec-euclid}

In this section we consider a binary relation on the set $\PC$ of bounded feasible topes (or via $\mu$ equivalently on the set of bases $\BM$).

\begin{defi}\label{def-bounded cone}
For any basis $b \in \BM$ we define the negative cone as
\[ \BC_b = \{ \b \in \{+,-\}^E \;|\; \mu(b)(i) = \b(i)~\mathrm{for~all}~i\in b \}. \]
Notice that this set of sign vectors depends on $\MC$ and $f$ but does not depend on $g$, in the sense that the signs $\mu(b)(i)$ for $i\in b$ only depend on the cocircuits of $\wt\MC/g$.
\end{defi}

\begin{prop}[Complementary Slackness]\label{prop-compslack}
For any $b\in\BM$, let $X_{b^c}$ be the feasible cocircuit of $\PB^\vee$ with $z(X_{b^c})=b^c = E \backslash b$. Then for any $\a\in\PC=\PC^\vee$, $\a\in\BC_b$ if and only if $X_{b^c}$ is a face of $\a$. 
\end{prop}

\begin{proof}
By definition, the tope $\a$ is in the cone $\BC_b$ if and only if
\[\a(i) = \mu(b)(i) \quad \mathrm{for~all}~i\in b.\]
By Proposition~\ref{prop-comp}, we can rewrite the above condition as
\[\a(i) = \mu^\vee(b^c)(i) \quad \mathrm{for~all}~i\in b.\]
Under the bijection $\mu^\vee$ from Corollary~\ref{cor-bij} for the dual program $\PB^\vee$, we have that $\mu^\vee(b^c)(i) = X_{b^c}(i)$ for all $i \in b$.  Thus the previous condition becomes
\[\a(i) = X_{b^c}(i) \quad \mathrm{for~all}~i\in b,\]
which in turn is equivalent to $X_{b^c}$ is a face of $\a$ in the affine space associated to $\PB^\vee$. 
\end{proof}

\begin{defi}[Cone relation] For $\a,\b \in \PC$, we write $\b \preceq \a$ whenever $\b \in \BC_{\mu^{-1}(\a)}$.
\end{defi}

The binary relation $\preceq$ on $\PC$ is reflexive and anti-symmetric, but not necessarily transitive.  Let $\leq$ denote the transitive closure of $\preceq$.  In general, the binary relation $\leq$ on $\PC$ does not define a poset, as the closure may no longer be anti-symmetric.

In the following sections we will define algebras associated to the program $\PB$.  To ensure that these algebras are quasi-hereditary, we will need the relation $\leq$ to define a poset.  It turns out that this is equivalent to a well-known condition on the oriented matroid program $\PB$, namely we ask that $\PB$ be \textit{Euclidean}.  To recall the definition, we first define the following graph associated to $\PB$.

\begin{defi}\label{def-graph}
Let $G_\PB$ be the graph whose vertices are vertices in $\AC$ (i.e., the feasible cocircuits for $\NC$) and whose edges are the edges in $\AC$ (i.e., the feasible covectors of $\NC$ of rank 2).  By our genericity condition, the graph is naturally directed by orienting each edge in increasing direction with respect to $f$.
\end{defi}

\begin{defi}\label{def-graphrelation}
For $b,b' \in \BM$, we write $b \leq b'$ if there is a directed path from $Y_b$ to $Y_{b'}$ in the graph $G_\PB$.
\end{defi}

\begin{defi}
The program $\PB$ is \textit{Euclidean} if the directed graph $G_\PB$ contains no directed cycles.  Equivalently, $\PB$ is Euclidean if the binary relation on $\BM$ from Definition~\ref{def-graphrelation} is anti-symmetric.
\end{defi}

By a result of Edmonds and Fukuda, the Euclidean property is well-behaved under duality:

\begin{prop}\label{prop-dualEuc} \cite[Corollary 10.5.9]{OMbook}
An oriented matroid program $\PB = (\wt\MC,g,f)$ is Euclidean if and only if its dual program $\PB^\vee = (\wt\MC^\vee,f,g)$ is Euclidean.
\end{prop}

Importantly for us, for Euclidean programs, the transitive closure of the cone relations is a poset. In fact these two conditions are equivalent:

\begin{lem}\label{lem-EucCone}
The oriented matroid program $\PB = (\wt\MC,g,f)$ is Euclidean 
if and only if the transitive closure of the cone relation is anti-symmetric.
\end{lem}

\begin{proof}
By Proposition~\ref{prop-dualEuc}, it suffices to show that the dual program $\PB^\vee$ is Euclidean if and only if the transitive closure on the cone relation is anti-symmetric.  By definition, $\PB^\vee$ is Euclidean if and only if the binary relation $(\BM^\vee,\leq)$ is anti-symmetric.\footnote{Here by $(\BM^\vee,\leq)$ we are referring to the binary relation on $\BM^\vee$ coming from the directed graph $G_{\PB^\vee}$.}

Thus it suffices to show that the bijection $\mu^\vee: \BM^\vee \to \PC^\vee = \PC$ is order reversing.  In other words, we wish to show that for $b_1,b_2 \in \BM$, \[\mu(b_1) \geq \mu(b_2) \quad \textrm{if~and~only~if} \quad b_1^c \leq b_2^c.\]

As $(\PC,\leq)$ is the transitive closure of $\preceq$, without loss of generality we may suppose $\mu(b_1) \succeq \mu(b_2)$.  Then $\mu(b_2) \in \BC_{b_1}$ and by Proposition~\ref{prop-compslack}, $X_{b_1^c}$ is a face of $\mu(b_2) = \mu^\vee(b_2^c)$.  As $X_{b_2^c}$ is the (unique) optimal cocircuit of the tope $\mu^\vee(b_2^c)$, there is a directed path from $X_{b_2^c}$ to $X_{b_1^c}$.  Thus $b_1^c \leq b_2^c$ as desired.

For the other direction, it suffices to consider the case where there is a directed edge from $X_{b_2^c}$ to $X_{b_1^c}$.  Then $X_{b_1^c}$ is a cocircuit face of $\mu^\vee(b_2^c) = \mu(b_2)$.  By Proposition~\ref{prop-compslack} this implies $\mu(b_2)\in \BC_{b_1}$ and so $\mu(b_2) \leq \mu(b_1)$, which completes the proof.
\end{proof}

\subsection{On the existence of Euclidean generic programs} Unlike genericity, it is not clear (at least to the authors) that any oriented matroid $\MC$ can be extended to a Euclidean generic oriented matroid program $\wt \MC$.  

On the other hand, we have already seen in Example~\ref{ex-ringel} a non-realizable Euclidean program $(\wt\MC,g,f)$ for which $\MC = \wt\MC/g \backslash f$ is realizable.  There also appear to be many examples for which the oriented matroid $\MC$ is non-realizable, and so our setting significantly generalizes that of~\cite{GDKD}.

In this section we first give such an example and then give a criterion on $\MC$ for the existence of a Euclidean generic oriented matroid program lifting and extending $\MC$.

\begin{ex}
The V\' amos matroid $\VC$ (rank 4 on 8 elements) is not representable over any field, but is orientable by \cite{1978Orientability}.  We now define a Euclidean generic program $\wt\VC$ by adjoining elements to $\VC$ as labelled and orientated in \cite{1978Orientability}. Let $(\wt\VC,g,f)$ be the generic oriented matroid program obtained from $\VC$ by adjoining $g$ using the lexicographic one-element lift defined by the cobasis $\{3,6,7,8\}$, and adjoining $f$ by a lexicographic one-element extension defined by the basis $\{1,2,4,5\}$. One can verify that the graph $G_\PB$ of this program has no directed cycles, so this program is Euclidean and generic.
\end{ex}

We are grateful to Jim Lawrence for suggesting that we consider the following statement.

\begin{prop}\label{prop-lasvergnas}
Let $\MC$ be an oriented matroid.  If one can adjoin elements $g$ and $f$ to $\MC$ to obtain a Euclidean generic oriented matroid program $(\wt\MC,g,f)$, then $\MC$ has a simplicial tope.
\end{prop}

Before beginning the proof, we note that it is an open conjecture of Las Vergnas~\cite{LV} (see also~\cite[7.3.10]{OMbook}) that any oriented matroid $\MC$ should have a simplicial tope.  Thus, the proposition implies that if for any oriented matroid $\MC$ there is a Euclidean generic oriented matroid program $(\wt\MC,g,f)$ for which $\wt\MC/g\backslash f = \MC$, then Las Vergnas' conjecture holds.

\begin{proof}
Suppose $(\wt\MC,g,f)$ is a Euclidean generic oriented matroid program such that $(\wt\MC/g) \backslash f = \MC$.

As $\PB$ is Euclidean, there exists a minimal vertex of $G_\PB$.  Let $Y_{\mathrm{min}}$ be the feasible cocircuit of $\NC$ corresponding to such a minimal vertex.  
Let $T_{\mathrm{min}}$ be the tope whose optimal cocircuit is $Y_{\mathrm{min}}$, which exists by Theorem~\ref{thm-optimal}.  Then $Y_{\mathrm{min}}$ is the only feasible cocircuit face of $T_{\mathrm{min}}$, so for all other cocircuit faces $Y$, $Y(g)=0$.  Consider the subtopes covered by $T_{\mathrm{min}}$.  There are $d$ such subtopes that have $Y_{\mathrm{min}}$ as a face, all of which are feasible.  Any other subtope $Z$ covered by $T_{\mathrm{min}}$ cannot be feasible and as all non-feasible cocircuit faces $Y$ of $T$ satisfy $Y(g)=0$, then $z(Z)=g$.  There can be at most one such subtope covered by $T_{\mathrm{min}}$.  By~\cite[Exercise 4.4]{OMbook}, any tope of $\wt\MC$ covers at least $d+1$ subtopes, so there does exist a subtope $Z$ and $T_{\mathrm{min}}$ covers exactly $d+1$ subtopes.  By the same exercise, it follows that $T_{\mathrm{min}}$ is simplicial, hence the subtope $Z$ is as well, which is in turn a tope of $\MC$.
\end{proof}

\subsection{Linear systems of parameters}
Recall that $M=\un\MC$ denotes the underlying matroid of $\MC$.
Consider the \emph{matroid complex} $\Delta$ of $M$ - the simplicial complex of independent sets of $M$. 

Fix a field $k$.  Let $k^E$ be the standard $n$-dimensional vector space with basis $\{t_i~|~i\in E\}$, so we may identify the symmetric algebra $\Sym k^E$ with the polynomial algebra $k[t_i~|~i\in E]$.

\begin{defi} For any field $k$, the \textit{face ring} of  the matroid complex $\Delta$ of $M$ is defined as
  \eq{k[M]:&= \Sym k^E/(t_{S} \mid S\not\in\Delta)\\
    &=\Sym k^E/(t_{\un X} \mid X\in C)}
where $t_{S}:=\prod_{i\in S}t_i$ for any $S\subset E$. We give $k[M]$ a $\ZM_{\geq0}$-grading by setting $\deg(t_i)=2$ for all $i$.
\end{defi}

\begin{defi} Recall $d$ denotes the rank of $M$.  A \textit{linear system of parameters} (l.s.o.p.) for $k[M]$ is a set $$\Omega = \{\omega_1, \omega_2,\ldots,\omega_d\} \subset k^E$$ such that $k[M]$ is a finitely-generated $k[\omega_1, \omega_2,\ldots,\omega_d]$-module.  Equivalently $k[M]/(\overline{\Omega})$ is a zero-dimensional ring, where $\overline{\omega}$ is the image of $\Omega$ in $k[M]_2$.
\end{defi}

\begin{remark}
Stanley~\cite{StanleyComb} defines an l.s.o.p.\ as a subset of $k[M]_2 \subset k[M]$.  The set $\Omega$ is a l.s.o.p.\ in the sense we define above if its image in $k[M]_2$ is an l.s.o.p.\ in the sense used by Stanley.
\end{remark}

We introduce the following perhaps non-standard definition:
\begin{defi}\label{def-param}
We call a subspace $U\subset k^E$ a \textit{parameter space} for $k[M]$ if the composition $U\into k^E \fib \textrm{Span}\{t_i\mid i\in b\}$ is an isomorphism for any basis $b\in\BM$.
\end{defi}

\begin{ex}
If $M$ is realizable as a hyperplane arrangement coming from a $k$-vector subspace $V \subset k^E$, then $V$ is a parameter space for $k[M]$.
\end{ex}

\begin{lem}
If $U\subset k^E$ is a parameter space for $k[M]$, then any basis of $U$ is a l.s.o.p.\ for $k[M]$.  If $\Omega$ is an l.s.o.p.\ for $k[M]$, then its span $\textrm{Span}(\Omega) \subset k^E$ is a parameter space for $k[M]$.
\end{lem}

\begin{proof}
Suppose $U$ is a parameter space.  By definition $$\dim U= \dim \textrm{Span}\{t_i\mid i\in b\} = d.$$  Suppose $\omega_1,\ldots,\omega_d$ is a basis of $U$.   By~\cite[Lemma 2.4(a)]{StanleyComb}, $\omega_1,\ldots,\omega_d$ is an l.s.o.p.\ for $k[M]$ if and only if for every facet of $\Delta$, that is basis $b \in \BM$, the list $\omega_1,\ldots,\omega_d \in k^E$ projects to a spanning set of $\textrm{Span}\{t_i\mid i\in b\}$.  This is true by the definition of $U$.

Similarly, if $\omega_1,\ldots,\omega_d \in k^E$ is an l.s.o.p., then for any basis $b\in \BM$, the projections to $\textrm{Span}\{t_i\mid i\in b\}$ of $\omega_1,\ldots,\omega_d$ are a spanning set.  Thus the projection from $\textrm{Span}(\Omega)$ to $\textrm{Span}\{t_i\mid i\in b\}$ is an isomorphism and $\textrm{Span}(\Omega)$ is a parameter space.
\end{proof}

By the Noether Normalization Lemma, if $k$ is an infinite field, then a l.s.o.p.\ for $k[M]$ exists.  From now on, we assume that an l.s.o.p.\ exists and fix a choice of l.s.o.p.\ $\Omega$ and its span $U$.

\subsection{Linear systems of parameters and duality}
\label{lsopduality}
Let $\{u_i~|~i\in E\}$ be the basis of $(k^E)^*$ dual to $\{t_i ~|~i\in E\}$.  It will be convenient for us to view the matroid complex $k[M^\vee]$ of the dual matroid $M^\vee$ as the appropriate quotient of $$\Sym (k^E)^* = k[u_i~|~i\in E].$$

Let $U^\perp$ be the kernel of the natural map $(k^E)^* \to U^*$.

\begin{lem}
$U^\perp$ is a parameter space for  the face ring of the dual matroid $k[M^\vee]$.
\end{lem}

\begin{proof}
For any $b \in \BM$ a basis for $M$, let $b^c:=E\backslash b$ be its complement, which is a basis for $M^\vee$.  It suffices to show that the projection  $U^\perp \to \textrm{Span}\{u_i\mid i\in b^c\}$ is an isomorphism.  As the two vector spaces have the same dimension, it suffices to show the null space is trivial.  Suppose $x$ is in the null space, so $$x \in U^\perp \cap  \textrm{Span}\{u_i\mid i\in b\}.$$  Then for any $u \in U$, $$0=\langle x, u\rangle = \langle x, \mathrm{pr}_{b} u\rangle,$$
where $\mathrm{pr}_{b}: U \to \textrm{Span}\{t_i\mid i\in b\}$ is the projection.  As $U$ is a parameter space for $k[M]$, the projection $\mathrm{pr}_{b}$ is an isomorphism and so $\langle x, w\rangle=0$ for any $w \in \textrm{Span}\{t_i\mid i\in b\}$.  Thus $x=0$ and we may conclude that the projection $U^\perp \to \textrm{Span}\{t_i\mid i\in b^c\}$ is an isomorphism.
\end{proof}

We define the dual of the pair $(\PB,U)$ to be $(\PB,U)^* = (\PB^\vee,U^\perp)$.

\begin{prop}
$k[M]$ is a free algebra over $k[\omega_1,\ldots,\omega_d] = \Sym U \subset \Sym k^E$, while
$k[M^\vee]$ is a free algebra over $\Sym U^\perp \subset \Sym (k^E)^*$. Both have rank $|\BM|$.
\end{prop}

\begin{proof} Recall that the face ring of a matroid complex is shellable (e.g.,~\cite[Proposition III.3.1]{StanleyComb}) and so the result follows~\cite[Theorem III.2.5]{StanleyComb}.
\end{proof}

\section{The algebra $A$}\label{sec-defA}

\subsection{The definition of $A$} \label{subsec-defA}
Recall that $\PB=(\wt\MC,g,f)$ is a generic oriented matroid program and $U \subset k^E$ is a parameter space for $M = \un\MC$.

Let $Q$ denote the quiver with vertex set $\FC$, the set of feasible topes, and arrows between topes that differ by exactly one sign.  We say that two topes $\a,\b$ that differ by exactly one sign are \textit{adjacent} and write $\a \leftrightarrow \b$.  If $\a \leftrightarrow \b$ and differ in the $i$-th component, we write $\b=\a^i$.

Let $P(Q)$ denote the path algebra for $Q$, which is generated by orthogonal idempotents $\{e_\a \mid \a \in \FC\}$ and edge paths $\{p(\a,\b)\}$ where $\a$ and $\b$ are adjacent and $p(\a,\b)$ is the path from $\a$ to $\b$.
We write $p(\a_1, \ldots, \a_k)$ for the element in the quiver algebra obtained as the composition $p(\a_1,\a_2) \cdot \ldots \cdot p(\a_{k-1},\a_k)$.

\begin{defi} \label{def-A}
  Let $\wt A=\wt A(\PB) = \wt A(\wt\MC,g,f)$ be the quotient of $P(Q)\otimes_k \Sym(k^E)^* = P(Q) \otimes_k k[u_i~|~i\in E]$ by the two-sided ideal generated by the following relations:
  \begin{enumerate}
   \item[$(A1)$] $e_\alpha=0$ for $\alpha\not\in\PC$,
   \item[$(A2)$] $p(\a,\g,\b)=p(\a,\d,\b)$ for any four distinct topes $\a,\b,\g,\d \in \FC$ where $\a$ and $\b$ are each connected to $\g$ and $\d$ by an edge, and 
   \item[(${A3}$)] $p(\a,\a^i,\a)=e_\alpha u_i$ whenever $\a,\a^i \in \FC$ differ only in the sign of $i\in E$.
  \end{enumerate}
We let $A=A(\PB,U) = A(\wt\MC,g,f,U)$ be
	\[ A:=\wt A \otimes_{\Sym (k^E)^*} \Sym((k^E)^*/(U^\perp)) =\wt A\otimes_{\Sym(k^E)^*} \Sym U^*, \]
or equivalently, the quotient of $\wt A$ by the additional relations
 \begin{enumerate}
   \item[($A4$)] $x=0$ for any $x \in U^\perp \subset \Sym(k^E)^*$.
  \end{enumerate}
\end{defi}

\begin{remark}
When the pair $(\PB,U)$ comes from a polarized arrangement as in Examples~\ref{ex-polar1} and~\ref{ex-polar2}, there is an equality $A(\PB,U) = A(V,\eta,\xi)$.
\end{remark}

As in~\ref{lsopduality}, for bookkeeping, we will use the dual coordinates for the dual matroid program, so we view $\wt A(\PB^\vee)$ and $A(\PB^\vee,U^\perp)$ as the analogous quotients of $P(Q)\otimes_k \Sym(k^E)= P(Q)\otimes_k k[t_i ~|~ i\in E]$.

\subsection{Expressions for elements of $A$}

We first introduce some terminology for paths in the quiver $Q$.  To distinguish between paths in $Q$ and their images in $A$, we will use the notation $\alpha_0\to \alpha_1 \to \cdots\to\alpha_r$ for a path of length $r$ in the quiver $Q$.

A path $\alpha_0\to\cdots\to\alpha_r$ in $Q$ is \emph{taut} if $\alpha_0$ and $\alpha_r$ differ in exactly $r$ coordinates. 
To relate paths, we use the following notion.

\begin{defi}
Two paths $P,P'$ in $Q$ are related by an \emph{elementary homotopy} (a symmetric relation) if
\begin{enumerate}
\item[(i)] $P$ is the path $\alpha_0\to\cdots\to\alpha_j\to\alpha_{j+1}\to\cdots\to\alpha_r$ of length $r$ while $P'$ is the path $\alpha_0\to\cdots\to\alpha_j\to\beta\to\alpha_j\to\alpha_{j+1}\to\cdots\to\alpha_r$ of length $r+2$ for some $\beta$ adjacent to $\alpha_j$, or
\item[(ii)] $P$ is the path $\alpha_0\to\cdots\to\alpha_{j-1}\to\alpha_j\to\alpha_{j+1}\to\cdots\to\alpha_r$ while $P'$ is the path $\alpha_0\to\cdots\to\alpha_{j-1}\to\alpha_j'\to\alpha_{j+1}\to\cdots\to\alpha_r$ of the same length, where we assume $\alpha_{j-1}$ and $\alpha_{j+1}$ differ in exactly 2 coordinates.
\end{enumerate}
\end{defi}

\begin{remark}
Under our assumption that $(\wt\MC,g,f)$ is generic, any feasible covector for $\NC = \wt\MC\backslash f$ of rank $d-1$ has exactly two zero coordinates. Thus, in this setting, this definition of elementary homotopy coincides with that given in~\cite[Section 4.4, page 184]{OMbook} for paths in the tope graph, because every elementary homotopy of type (ii) in the sense of~\cite{OMbook} between feasible paths is of the form above.
\end{remark}

We will use the following result:

\begin{prop}\label{prop-taut}
Let $P$ and $P'$ be any two taut paths in $Q$ with the same start and end points.  Then $P$ and $P'$ are related by a sequence of elementary homotopies of type (ii) such that every intermediate path is also taut.
\end{prop}

\begin{proof}
If we consider instead paths in the entire tope graph and the more general notion of elementary homotopy, this is a result of Cordovil-Moreira~\cite{CordMor} (see also~\cite[Proposition 4.4.7]{OMbook}). Recall that a subset $\RC \subset \TC$ is $T$-convex if it contains every shortest path between any two of its members and the set of feasible topes is $T$-convex (see \cite[Definition 4.2.5]{OMbook} and the discussion that follows). The result then follows for paths in $Q$.
\end{proof}

\begin{prop} \label{prop-tautification}
Given a path $P=(\alpha_0\to\cdots\to\alpha_s)$ in $Q$, let $d_i$ be the number of times the $i$-th coordinate changes twice. Then for any taut path $P'=(\alpha_0=\beta_0\to\cdots\to\beta_r=\alpha_s)$, we have
\[p(\alpha_0,\dots,\alpha_s)=p(\beta_0,\cdots,\beta_r)\cdot\prod_{i\in E}u_i^{d_i}\]
in $\wt A$.
\end{prop}

\begin{proof}
Note that $s-r \geq 0$ with equality if and only if $P$ is a taut path.  We prove the proposition by induction on $s-r$. If $s=r$, then both paths are taut, and so by Proposition~\ref{prop-taut} they are related by a sequence of elementary homotopies of type (ii).  But an elementary homotopy of type (ii) descends to an equality in the path algebra by definition, so
\[p(\alpha_0,\dots,\alpha_r)=p(\beta_0,\dots,\beta_r),\]
as desired.

Assume that the statement holds whenever $s-r < k$ for some positive integer $k$. Suppose that $s-r=k$. There exists a minimal $\ell$ such that $\alpha_0\to\cdots\to\alpha_\ell$ is taut, while $\alpha_0\to\cdots\to\alpha_{\ell+1}$ is not taut. Then $\alpha_0$ and $\alpha_{\ell+1}$ differ in $\ell-1$ coordinates and for some $i$, $\alpha_0(i)=\alpha_{\ell+1}(i)\neq\alpha_\ell(i)$. 

Notice that any taut path between $\a_0$ and $\a_\ell$ will have length exactly one more than the length of a taut path from $\a_0$ to $\a_{\ell+1}$. Therefore, by Proposition~\ref{prop-taut}, using a sequence of elementary homotopies of type (ii) we can replace $\a_0\to\cdots\to\a_{\ell-1}\to\a_\ell$ with a taut path $\a_0\to\a_1'\to\cdots\to\a_{\ell-2}'\to\a_{\ell+1}\to\a_\ell$. 

This gives an equality in the algebra $A$: 
\begin{align*}
p(\alpha_0,\a_1,\dots,\a_{\ell-2},\alpha_{\ell-1},\alpha_{\ell},\a_{\ell+1})
&=p(\a_0,\a_1',\dots,\a_{\ell-2}',\a_{\ell+1},\a_\ell,\a_{\ell+1}) \\
&=p(\a_0,\a_1',\dots,\a_{\ell-2}',\a_{\ell+1})u_i.
\end{align*}
 We are then reduced to considering the path \[\a_0\to\a_1'\to\cdots\to\a_{\ell-2}'\to\a_{\ell+1}\to\a_{\ell+2}\to\cdots\to\a_s\] of length $s-1$ and the number of times the $i$-th coordinate changes twice is $d_i-1$, while the number of times every $j$-th coordinate changes twice remains $d_j$ for all $j\neq i$.  We can then invoke the induction hypothesis to complete the proof.
\end{proof}

The following two corollaries are analogous to~\cite[Corollary 3.10]{GDKD}.

\begin{cor}\label{cor-thru}
Consider an element 
\[a=p\cdot\prod_{i\in E}u_i^{d_i}\in \wt A,\]
where $p$ is a taut path in $Q$ from $\alpha$ to $\beta$.
Suppose $\gamma$ denotes a feasible tope such that if $\alpha(i)=\beta(i)$ and $d_i=0$, then $\gamma(i)=\alpha(i)=\beta(i)$. Then 
\[a=a' \cdot m,\]
where $a'$ is the concatenation of a taut path from $\alpha$ to $\gamma$ with a taut path from $\gamma$ to $\beta$ and $m$ is a product of $u_i$'s.

In particular, if $\gamma$ is not bounded, then $a=0$ in both $\wt A$ and $A$.
\end{cor}

\begin{proof}
For all $j \in E$, either:
\begin{enumerate}
\item $\alpha(i) \neq \beta(j)$ and the $j$-th coordinate changes in the concatenation $a'$ exactly once.
\item $\gamma(j)=\alpha(j)=\beta(j)$ and the $j$-th coordinate does not change in the concatenation $a'$.
\item $\gamma(j)\neq\alpha(j)=\beta(j)$ and so the $j$-th coordinate changes exactly twice in the concatenation $a'$.
\end{enumerate}
Proposition~\ref{prop-tautification} then says:
\[a' = p\cdot\prod_{i\in E}u_i^{d_i'},\]
where $d_i'\in\{0,1\}$ and by our assumption on $\g$, we have $d_i \geq d_i'$.  Thus 
\[a = a' \cdot \prod_{i \in E} u_i^{d_i-d_i'},\] 
as desired.
\end{proof}

\begin{cor}\label{cor-lin combo}
Let $b$ be the zero set of any feasible cocircuit face of $\a\in \PC$.  For any $j\in E$, $e_\a u_j \in A$ can be written as a $k$-linear combination of paths $\{p(\a,\a^i,\a)\mid i\in b\}$.

In particular, the image in $A$ of the element $a \in \wt A$ described in Corollary~\ref{cor-thru} can be expressed a linear combination of paths in $Q$ that pass through $\g$.
\end{cor}

\begin{proof}
For $j\in b$, the tope $\a^j$ is feasible, so $e_\a u_j = p(\a,\a^j,\a)$.  On the other hand, as $U$ is a parameter space for $k[M]$, the set $\{u_i| i\in b\}$ restricts to a basis of $U^*$. Thus for any $j\not\in b$, $u_j \in A$ can be expressed as a linear combination of $\{u_i\mid i\in b\}$.
\end{proof}

\subsection{Alternative description of $A$}\label{section-quiver}
We conclude this section with a slightly different description of our algebra $A$, which will make it easier to describe how $A$ changes when we modify the choice of generic oriented matroid program.

Let $D$ be the path algebra over $k$ of the quiver with two vertices labelled by $+$ and $-$ and an arrow in each direction.  Let $D_E=D^{\otimes E}$ denote the $E$-fold tensor product of $D$ with itself.  In particular, $D_E$ is the path algebra on the quiver with vertices labelled by the set $\{+,-\}^E$ of sign vectors, or equivalently vertices of an $|E|$-cube, and edges connecting any two sign vectors that differ in exactly one position, modulo the relations that whenever $\a,\b \in \{+,-\}^E$ differ in exactly two positions $i$ and $j$, we have an equivalence of paths in $D_E$:
\[p(\a,\a^i,\b) = p(\a , \a^j , \b).\]
(As before, $\a^i\in\{+,-\}^E$ denotes the sign vector that differs from $\a$ in exactly the $i$-th coordinate.)

For any sign vector $\a$, we again let $e_\a$ denote the idempotent defined as the trivial path at the vertex labelled by $\a$.  Let $e_\PC = \sum_{\a\in \PC} e_\a$ be the sum of idempotents corresponding to bounded, feasible topes and $e_f = \sum_{\a \not\in \BC} e_\a$ be the sum of idempotents corresponding to unbounded sign vectors.

For each $i \in E$, we consider the element $\theta_i \in D_E$ defined as the sum $$\theta_i  = \sum_{\a \in \{+,-\}^E} p(\a,\a^i,\a).$$  Note that the center $Z(D_E)$ of $D_E$ is a polynomial algebra with generators $\theta_i$.  Let $\vartheta: (k^E)^* \to k\{\theta_i \mid i\in E\}$ be the isomorphism sending $u_i$ to $\theta_i$.

\begin{lem}\label{lem-defv2}
The algebra $\wt A$ from Definition~\ref{def-A} is isomorphic to the quotient of $e_\PC D_E e_\PC$ by the relation:
\begin{enumerate}
\item[$(A1')$] $e_\a = 0$ for $\a \not\in \BC$,
\end{enumerate}
and the algebra $A$ is obtained by adding the additional relation:
\begin{enumerate}
\item[$(A4')$] $\vartheta(x)=0$ for $x \in U^\perp$.
\end{enumerate}

Equivalently, there are isomorphisms $$\wt A \cong e_\PC D_E e_\PC/\langle e_f e_\PC \rangle, \quad  \quad
A \cong e_\PC D_E e_\PC/\langle e_f e_\PC \rangle + \langle \vartheta(U^\perp)e_\PC \rangle. $$
\end{lem}

\begin{proof}
To distinguish between the two definitions, let $\wt A_1$ and $A_1$ denote the original algebras and $\wt A_2$ and $A_2$ be the algebras defined as in the lemma.  Note that there is an injective homomorphism from $\wt A_1$ to $\wt A_2$ uniquely defined by sending a path in $Q$ to the corresponding path in the cube quiver $\{+,-\}^E$ and sending $u_i$ to $\theta_i$.  To see that it is surjective, it suffices to observe that for any two topes $\a,\b \in \PC$, there exists a taut path in $Q$ from $\a$ to $\b$, which follows from~\cite[Lemme 3.7]{Cord}.
\end{proof}

\section{The quadratic dual of $A$}\label{sec-quaddual}
In this section we observe that $A(\PB,U)$ is a quadratic algebra and that its quadratic dual is isomorphic to $A(\PB^\vee,U^\perp)$.  To ease notational clutter, in this section we will write $A$ for $A(\PB,U)$ and $A^\vee$ for $A(\PB^\vee,U^\perp)$.

To state and prove the results of this section, we will need some additional notation.  Let $Q_\PC \subset Q$ be the full subquiver with vertices $\PC \subset \FC$.  For $\a \in \PC$, we let 
   $$J_\a := \{i \mid \a^i \in \PC\},$$
$$I_\a := \{i \mid \a^i \in \FC\},$$
$$K_\a := \{i \mid \a^i \in \BC\}.$$

\begin{lem}\label{lem-quad}
The algebra $A = A(\PB,U)$ is the quotient of the path algebra $P(Q_\PC)$ by the following relations:
  \begin{enumerate}
   \item[$(A2'')$] For distinct topes $\a,\b,\g \in \PC$, $\d \in\FC$, where $\a$ and $\b$ are each connected to $\g$ and $\d$ by an edge,
   $$p(\a,\g,\b)= \begin{cases}
        p(\a,\d,\b) & \mathrm{if~}\d\in\PC \\
        0 & \mathrm{otherwise}
        \end{cases}$$
   \item[(${A3/4''}$)] 
For any $\a \in \PC$ and $w \in U^\perp \cap \mathrm{Span}\{u_i \mid i\in I_\a\}$, if $w = \sum_{i\in I_\a} w_i u_i$ for some $w_i \in k$, then
\[ \sum_{i\in J_\a} w_i p(\a,\a^i,\a)= 0.\]
  \end{enumerate}
In particular, it follows that $A$ is a quadratic algebra.
\end{lem}

\begin{proof}
  Note that there is a surjection from $P(Q)$ to $P(Q_\PC)$ (by setting $e_\a = 0$ for all $\a\not\in \PC$) and that this map factors through the natural map from $P(Q)$ to $A$.

We first show that the map from $P(Q)$ to $A$ (and hence from $P(Q_\PC)$ to $A$) is surjective.  To do so, it suffices to show that the image in $A$ of any element of $w\in U^*$ is in the image of the map from $P(Q)$ to $A$.  In $A$, we can express $w$ as the sum $w= \sum_{\a\in\FC} w e_\a$, so it is enough to show that for any $\a \in \FC$, $we_\a$ is in the image.

Because $\a$ is feasible, there exists a feasible cocircuit face $Y$ of $\a$ and $z(Y)$ is a basis of $M$.  Thus
$z(Y) \subset I_\a$ and so by our assumption that $U$ is a parameter space, the image of $\{u_i | i\in I_\a\} \supset \{u_i | i \in z(Y)\}$ is a spanning set of $U^*$.  For any $w \in U^*$, we can therefore write $w e_\a$ as a linear combination of elements of the form $p(\a,\a^i,\a)$ where $\a^i \in \FC$.  We conclude that map from $P(Q)$ to $A$ is surjective.

To identify the kernel, note that $(A2'')$ is simply the image of the relation $(A2)$ in $P(Q_\PC)$ and similarly the relations $(A3)$ and $(A4)$ combine to give $(A3/4'')$.

As $P(Q_\PC)$ is generated over its degree zero component by its degree one component and the relations above are quadric, we conclude that $A$ is a quadratic algebra.  
\end{proof}

Recall that $A_1$ is an $A_0$-bimodule.  Since $A$ is quadratic, $A$ is a quotient of the tensor algebra $T(A_1)$ over $A_0$ of the form
  \[A=T(A_1)/T(A_1)\cdot W\cdot T(A_1)\]
where $W\subset N:=(T(A_1)\otimes_{A_0} T(A_1))_2$ is the space of quadratic relations in $A$. The \emph{quadratic dual} of $A$ is defined to be
  \[A^!=T(A_1^*)/T(A_1^*)\cdot W^\perp\cdot T(A_1^*)\]
where $W^\perp\subset N^*:=(T(A_1^*)\otimes_{A_0} T(A_1^*))_2$ is the set of elements orthogonal to $W$. 

\begin{thm}\label{thm-quad}
There is an isomorphism $A(\PB^\vee,U^\perp)\simeq A(\PB,U)^!$.
\end{thm}

\begin{proof}
Mimicking the proof of~\cite[Theorem 3.11]{GDKD}, we define an isomorphism between $(A^\vee)_1$ and  $(A_1)^*$ and show that the space $W^\vee$ of quadratic relations in $A^\vee$ coincides with the space $W^\perp$ of quadratic relations in $A^!$ under this identification.

In degree zero, we have a canonical identification $$A_0=k\{e_\a\mid \a\in\PC=\PC^\vee\} = (A^\vee)_0.$$

In degree one, \[\{p(\a,\b) \mid \a,\b \in \PC~\mathrm{such~that}~\a \leftrightarrow \b\} \subset A_1\] is a natural basis for $A_1$.

As $\PC = \PC^\vee$, to distinguish elements of $A$ and $A^\vee$, we let $p^\vee(\a,\b)$ denote the element of $(A^\vee)_1$ associated to the arrow $\a \rightarrow \b$ in $\PC^\vee$.

We now identify $(A^\vee)_1$ with $(A_1)^*$ as follows.  First we attach a sign $\e(\a \leftrightarrow \b)$ to each pair $\a \leftrightarrow \b$ of adjacent topes in $\PC$ such that for distinct topes $\a,\a^i,\a^j,(\a^i)^j$, an odd number of the edges of the square 
\[
\begin{array}{c}
\xymatrix{
\a \ar@{<->}[r] \ar@{<->}[d]& \a^i \ar@{<->}[d] \\
\a^j \ar@{<->}[r] & (\a^i)^j
}
\end{array}
\]
 are attached a negative sign.\footnote{This can be done for the edges of the n-cube $\{+,-\}^E$ by identifying its vertices with monomials in the exterior algebra $\Lambda k\{e_1,\ldots,e_n\}$ and then using the standard differential to attach signs to edges.  Restricting to $\PC$ then gives a collection of signs as desired.}

We then identify $(A^\vee)_1$ with $(A_1)^*$ via the perfect pairing
$$ (A^\vee)_1 \times A_1 \to k$$
$$ \langle p^\vee(\a,\b),p(\g,\d) \rangle = \begin{cases}
\e(\a \leftrightarrow \b) & \mathrm{if}~\a=\d,\b=\g, \\
0 & \mathrm{otherwise.}
\end{cases}$$

For the remainder of the proof, we let 
  \eq{N:&=A_1 \otimes_{A_0} A_1=T(A_1)_2,\\
  N^\vee:&= (A^\vee)_1 \otimes_{A_0} (A^\vee)_1=T((A^\vee)_1)_2.}
We wish to show that our choice of perfect pairing induces an isomorphism between $W^\vee\subset N^\vee$ and $W^\perp\subset N^*$.

Note that the relations of type ($A2''$) lie in $e_\a W e_\b$ and the relations of type ($A3/4''$) lie in $e_\a W e_\a$.  As these relations are homogeneous with respect to the idempotents $e_\a$ for $\a\in\PC$, we have a direct sum decomposition $$W = \bigoplus_{\a,\b \in \PC} e_\a W e_\b.$$  Moreover, $e_\a W e_\b$ and $e_\g W^\vee e_\d$ are orthogonal unless $\a=\d$ and $\b=\g$.  Thus it is enough to check that $e_\b W^\vee e_\a \subset e_\b N^\vee e_\a$ is the perpendicular complement to $e_\a W e_\b \subset e_\a N e_\b$ for any $\a,\b\in \PC$.

Note that $e_\a N e_\b, e_\b N^\vee e_\a$ are zero unless $\a=\b$ or $\a$ and $\b$ differ in exactly two positions and there is a path from $\a$ to $\b$ in $\PC$.

We first deal with the latter case. Suppose the two elements of $E$ where $\a$ and $\b$ differ are $i\neq j$. We must have that at least one of $\a^i$ or $\a^j$ is in $\PC$ by assumption. We can assume that $a^j\in\PC$.

If $a^i\in\PC$, then we have that $e_\a Ne_\b$ is a two dimensional $k$-vector space with basis $\{p(\a,\a^i)\otimes p(\a^i,\b),p(\a,\a^j)\otimes p(\a^j,\b)\}$ and 
  \eq{
    e_\a We_\b&=k\{p(\a,\a^i)\otimes p(\a^i,\b)-p(\a,\a^j)\otimes p(\a^j,\b)\} \subset e_\a Ne_\b,,\;\text{ and} \\
    e_\b W^\vee e_\a&=k\{p^\vee(\b,\a^i)\otimes p^\vee(\a^i,\a)-p^\vee(\b,\a^j)\otimes p^\vee(\a^j,\a)\} \subset e_\b N^\vee e_\a.}
Pairing the two basis vectors together using the form defined above we get $${\e(\a \leftrightarrow \a^i)}{\e(\a^i \leftrightarrow \b)}+{\e(\a \leftrightarrow \a^j)}{\e(\a^j \leftrightarrow \b)}.$$  By our choice of signs, the terms cancel and we conclude that $e_\a W^\perp e_\b=e_\b W^\vee e_\a.$

If $\a^j\not\in\PC$, then either $\a^j\in\FC\backslash \PC$ and $\a^j\not\in\BC=\FC^\vee$, or $\a^j\in\BC\backslash\PC=\FC^\vee\backslash \PC^\vee$ and $\a^j\not\in \FC=\BC^\vee$. Assume that $\a^j\in\FC\backslash \PC$ and $\a^j\not\in\BC=\FC^\vee$. Then we have that 
  \[e_\a Ne_\b = k\{p(\a,\a^i)\otimes p(\a^i,\b)\} = e_\a W e_\b\]
since $p(\a,\a^i,\b)=p(\a,\a^j,\b)=0$ in $A$. On the other hand, $\a^j\not\in\FC^\vee$ means that $e_\b N^\vee e_\a=k\{p^\vee(\b,\a^i)\otimes p^\vee(\a^i,\a)\}$ and ($A2$) does not impose any relations, so $e_\b W^\vee e_\a=0$. Therefore $e_\b W^\perp e_\a=0=e_\b W^\vee e_\a$. The case $\a^j\in\BC\backslash\PC=\FC^\vee\backslash \PC^\vee$ and $\a^j\not\in \FC=\BC^\vee$ follows from the same argument on the dual side.

Finally, consider the case where $\a=\b$. Note that 
\[e_\a N e_\a = k\{p(\a,\a^i)\otimes p(\a^i,a) \mid i \in J_\a\}. \]
We identify $e_\a N e_\a$ with $k^{J_\a}$ by regarding $p(\a,\a^i)\otimes p(\a^i,a)$ as the standard basis element labelled by $i\in J_\a$. We also use the standard pairing on $k^E$ to view $U^\perp$ as a subspace of $k^E$.
From the relations ($A3/4''$), we find that $e_\a W e_\a$ is given by $pr_{J_\a}(U^\perp \cap k^{I_\a})$ or equivalently by $(pr_{K_\a} U^\perp) \cap k^{J_\a}$, where $pr_{S}$ denotes the orthogonal projection from $k^E$ to $k^S$ for any $S \subset E$.

Taking the orthogonal complement of $e_\a W e_\a \subset e_\a N e_\a$ using the first description gives:
\eq{(e_\a W e_\a)^\perp = (pr_{J_\a}(U^\perp \cap k^{I_\a}))^\perp &= (pr_{I_\a} U) \cap k^{J_\a} \\
&= (pr_{K_\a^\vee} U) \cap k^{J_\a^\vee}= e_\a W^\vee e_\a.}
Here the second equality follows from the fact that $(pr_S V)^\perp = V^\perp \cap k^S$, the third equality uses the fact that $I_\a = K_\a^\vee$ and $J_\a=J_\a^\vee$ (see Proposition~\ref{prop-FBP}) and the final equality follows from the second description of $W$ above.
\end{proof}

\section{The algebra $B$}\label{sec-defB}

Following Braden--Licata--Proudfoot--Webster, we define in this section another algebra $B=B(\PB,U)$ associated to the pair $(\PB,U)$ and prove that $B$ is isomorphic to the quadratic dual $A^!$ of $A(\PB,U)$. We also consider a deformed version $\wt B(\PB)$ such that $\wt B(\PB) \cong \wt A(\PB^\vee)$.
The algebras $B$ and $\wt B$ defined in~\cite[Section 4.1]{GDKD} coincide with those defined here in the special case when $(\PB,U)$ comes from a linear subspace as in Example~\ref{ex-linear}.  

\subsection{A topological lemma} In this section we assume some familiarity with regular cell complexes, posets and their geometric realizations and refer the reader to~\cite[Section 4.7]{OMbook} for more on these topics.

Recall from Section~\ref{subsec-P} that $\LC = \LC(\NC)$ denotes the poset of covectors for $\NC = \wt\MC\backslash f$. This poset is pure with a unique minimal element $0$ and a rank function $\r$. Every covector is uniquely determined by its cocircuit faces.

We will define the algebra $B(\PB,U)$ from the poset structure of $\LC$, using an affine version of the following notion~\cite[Definition 4.1.2(ii)]{OMbook}:

\begin{defi}
For any covector $Y\in\LC$, the \emph{Edmonds--Mandel face lattice of $Y$}, denoted $\flem(Y)$, is the set of all faces of $Y$ in $\LC$.
The opposite poset $\fllv(Y):=\flem(Y)^\op$ is called the \emph{Las Vergnas face lattice of $Y$}. 
\end{defi}

Both $\flem(Y)$ and $\fllv(Y)$ are graded posets and by theorems of Folkman--Lawrence and Edmonds--Mandel, they have the following topological interpretation~\cite[Theorem 4.3.5]{OMbook}:

\begin{thm}\label{thm-FLEM}
The lattices $\flem(Y)$ and $\fllv(Y)$ are each isomorphic to the face lattices (or augmented face posets) of PL regular cell decompositions of the $(\r(Y)-2)$-sphere.
\end{thm}

We will use the following affine (or feasible) version of the Edmonds--Mandel and Las Vergnas face lattice.

\begin{defi}\label{def-feaslattices} Let $Y\in\AC$.
The \emph{feasible Edmonds--Mandel face lattice of $Y$} is 
    \eq{\flem^\FC(Y):=(\flem(Y)\cap\AC)\cup \{0\},}
while the \emph{feasible Las Vergnas face lattice of $Y$} is
    \eq{\fllv^\FC(Y):=\flem^\FC(Y)^\op.}
\end{defi}

For any feasible covector $Y\in\AC$, let $Y^\infty\in\AC^\infty$ be the unique maximal face of $Y$ in the boundary. The face $Y^\infty$ is equal to the composition of all cocircuit faces of $Y$ in $\AC^\infty$.

Note that $\AC$ does not have a cellular interpretation, since the faces of a feasible covector need not be feasible. The same is true of $\flem^\FC(Y)$ when $Y^\infty\neq0$.
On the other hand, $\fllv^\FC(Y)$ will always have a cellular interpretation, even if the program $\PB$ is not generic.
 The assumption that $g$ is generic implies $\fllv^\FC(Y)$ is the face lattice of a pure simplicial complex, whose vertices correspond to the feasible facets of $Y$ and whose maximal simplices correspond to the feasible cocircuits faces of $Y$. 

We will use the following lemma on the topology of $\fllv^\FC(Y)$ to show that the algebra $B(\PB,U)$ we define below is finite-dimensional (which will imply that $A(\PB,U)$ is finite dimensional as well).

\begin{lem}\label{lem-feas Las Vergnas}
    Let $Y\in\AC$. 
    \begin{itemize}
\item If $Y^\infty=0$, then the geometric realization $\|\fllv^\FC(Y)\backslash\{0,Y\}\|$ of the proper part of the feasible Las Vergnas face lattice of $Y$ is a PL $(\r(Y)-2)$-sphere. 
\item If $Y^\infty\neq 0$, then $\|\fllv^\FC(Y)\backslash\{0,Y\}\|$ is a PL $(\r(Y)-2)$-ball.
\end{itemize}
\end{lem}

\begin{proof}
If $Y^\infty=0$, then $\fllv^\FC(Y)=\fllv(Y)$ and so the statement reduces to Theorem~\ref{thm-FLEM}.

   Assume that $Y^\infty\neq0$, and let 
        \eq{ \D:=\fllv^\FC(Y)\backslash\{0,Y\} } 
denote the proper part of the feasible Las Vergnas face lattice of $Y$.

The geometric realization $\|\D\|$ of $\D$ is homeomorphic to the geometric realization $\|\D_{\ord}(\D)\|$ of the order complex $\D_{\ord}(\D)$ of $\D$ (which is a subdivision of the former).  Further, using the canonical identification $\D_{\ord}(\D) \cong \D_{\ord}(\D^\op)$, we find that $\|\D\|$ is homeomorphic to $\|\D_{\ord}(\D^\op)\|$.

It will thus suffice to prove that $\|\D_{\ord}(\D^\op)\|$ is a PL ball.  By restricting to $\PB/z(Y)$ if necessary, we may assume that $Y$ is a tope, so $\r(Y)=d+1$ and $\D_{\ord}(\D^\op)$ is a $(d-1)$-dimensional simplicial complex. 

We first note that $\D_{\ord}(\D^\op)\cong \D_{\ord}(\D)$ is pure because $\D$ is the face poset of a pure simplicial complex.  By~\cite[Proposition 4.5.4]{OMbook}, the graded poset $\flem^\FC(Y)$ admits a recursive coatom ordering, which implies by~\cite[Lemma 4.7.19]{OMbook} that its open interval $\D^\op=(0,Y)$ is shellable (meaning that the order complex $\D_{\ord}(\D^\op)$ is shellable).
 
    To conclude that the shellable $(d-1)$-dimensional simplicial complex $\D_{\ord}(\D^\op)$ is a PL $(d-1)$-ball we use the criterion of~\cite[Proposition 4.7.22(ii)]{OMbook}: namely we must show that every $(d-2)$-simplex is the face of one or two $(d-1)$-simplices and at least one $(d-2)$-simplex is the face of exactly one $(d-1)$-simplex.

Note that a $(d-1)$-simplex of $\D_{\ord}(\D^\op)$ is a maximal chain $x_0<x_1<\dots<x_{d-1}$ of feasible covector faces of $Y$, where $x_i<x_{i+1}$ means that $x_i$ is a proper face of $x_{i+1}$.  Now consider any $(d-2)$-simplex of $\D_{\ord}(\D^\op)$. It will similarly be a chain of the form $x_0<x_1<\dots<x_{i-1}<x_{i+1}<\dots<x_{d-1}$ for some $i$ between $0$ and $d-1$.

If $i>0$, then the chain can be completed to a maximal chain in exactly two ways because $\D$ is a pure simplicial complex.

If $i=0$, then the feasible edge $x_1$ either has one or two feasible cocircuit faces and so the $(d-2)$-simplex can be completed to either one or two $(d-1)$-simplices.  As $Y^\infty \neq 0$, there exists an edge $x_1$ of $Y$ with only one feasible cocircuit face and thus a $(d-2)$-simplex of $\D_{\ord}(\D^\op)$ that is the face of a unique $(d-1)$-simplex.

We conclude that $\norm{\D_{\ord}(\D^\op)} \cong \norm{\D}$ is a PL $(d-1)$-ball.
\end{proof}

\subsection{The algebra $\wt B$}\label{section-B}
Recall that the face ring of a simplicial complex $\mathbf{\D}$ is defined as
\[k[\mathbf{\D}]=k[t_i~|~i\in E]/(t_S\mid S\not\in \mathbf{\D})=\Sym k^E/(t_S\mid S\not\in \mathbf{\D}),\]
where here and in what follows $t_S=\prod_{i\in S}t_i$ with the convention $t_\varnothing =1$. As before, we let $k[M]$ denote the face ring of the matroid complex of $M$.

For any bounded feasible topes $\a_1,\dots,\a_r\in\PC$, let $Y=\a_1\wedge\cdots\wedge\a_r$ denote the unique maximal common feasible covector face of all $\a_1,\dots,\a_r$. Then 
    \eq{\flem^\FC(Y)=\bigcap_{\ell=1}^r\flem^\FC(\a_\ell).}
  We define  
    \eq{\D_{\a_1 \cdots \a_r} := \fllv^\FC(Y) \backslash \{0\} }
and 
    \eq{z(\D_{\a_1\cdots\a_r}):=\{S\subset E\mid S\subset z(X)\;\text{for some}\:X\in\D_{\a_1\cdots\a_r}\}. } 
Note that $z(\D_{\a_1\cdots\a_r})$ is a simplicial complex and $\D_{\a_1\cdots\a_r}$ can be realized as the face poset of a simplicial complex.

\begin{defi} 
For $\a_1,\dots,\a_r\in\PC$, let \eq{\wt R_{\a_1\cdots \a_r} := k[z(\D_{\a_1\cdots\a_r})].}
\end{defi}

\begin{remark}
When $(\PB,U)$ comes from a linear subspace as in Example~\ref{ex-linear}, then each feasible tope $\alpha$ corresponds to a bounded feasible chamber in the the corresponding hyperplane arrangement and the ring $\wt R_{\a_1\cdots \a_r}$ defined here agrees with the corresponding ones defined in~\cite[Definition 4.1]{GDKD}.
\end{remark}

For any feasible covector $Y$, the zero set $z(Y) \subset E$ is an independent set of $M$, so there are natural quotient maps $k[M]\to\wt R_{\a_1\cdots \a_r}$.  Notice that for any $\b \in \FC$ there is also a natural quotient map $\wt R_{\a_1\cdots \a_r}\to \wt R_{\a_1\cdots \a_r \b}$ compatible with the maps from $k[M]$. Furthermore, the quotient $k[M]\to\wt R_{\a_1\cdots \a_r}$ makes $\wt R_{\a_1\cdots\a_r}$ a $\Sym U$-module.

\begin{lem}\label{lem-size}
For $\a_1,\dots,\a_r\in\PC$, let $Y=\a_1\wedge\cdots\wedge\a_r$.
The ring $\wt R_{\a_1\cdots \a_r}$ is a free $\Sym U$-module whose rank is equal to the number of feasible cocircuit faces of $Y$. 
\end{lem}

\begin{proof}
Lemma \ref{lem-feas Las Vergnas} tells us that $\|\fllv^\FC(Y) \backslash \{0,Y\}\|=\|\D_{\a_1\dots\a_r}\|$ is a $(\r(Y)-2)$-sphere or $(\r(Y)-2)$-ball. If $z(Y)=\varnothing$, then the posets $z(\D_{\a_1\cdots\a_r})$ and $\D_{\a_1\cdots\a_r}$ are isomorphic. More generally, the geometric realization $\|z(\D_{\a_1\cdots\a_r})\|$ is the $|z(Y)|$-fold cone over $\|\D_{\a_1\cdots\a_r}\|$.  In any case, $\|z(\D_{\a_1\cdots\a_r})\|$ is either a $(d-1)$-ball or $(d-1)$-sphere. 

By results of Hochster, Reisner, and Munkres (cf \cite[Section II.4]{StanleyComb}) it follows that $\wt R_{\a_1\cdots\a_r} = k[z(\D_{\a_1\cdots\a_r})]$ is a Cohen-Macaulay ring. Hence~\cite[Theorem I.5.10]{StanleyComb} implies that $\wt R_{\a_1\cdots\a_r}$ is a free module of finite rank over the symmetric algebra of any parameter space. Here $U$ is a parameter space for $\wt R_{\a_1\cdots\a_r}$ by \cite[Lemma III.2.4]{StanleyComb} because the composition $U\hookrightarrow k^E\fib k^{z(X)}$ is an isomorphism for any cocircuit $X\in\flem^\FC(Y)$.  Also by \cite[Lemma III.2.4]{StanleyComb}, the rank is equal to the number of maximal simplices of $z(\D_{\a_1\cdots\a_r})$, which are in bijection with the cocircuits of $\flem^\FC(Y)$.
\end{proof}

\begin{remark}\label{rmk-face ring basis}
While we will not need it in what follows, we note that if $\PB$ is Euclidean one can prove that $z(\D_{\a_1\cdots\a_r})$ is in fact shellable and then \cite[Theorem III.2.5]{StanleyComb} gives an explicit basis of $\wt R_{\a_1\cdots\a_r}$ as a free $\Sym U$-module. 
\end{remark}

For any $\a,\b\in\PC$, we let
\[ d_{\a\b}:=|\{i\mid\a(i)\neq\b(i)\}|,\]
which coincides with the length of any taut path from $\a$ to $\b$. For $\a,\b,\g$, we let 
\[ S^\b_{\a\g}:=\{i\mid \a(i)=\g(i)\neq\b(i)\}, \]
which is the set of $i\in E$ such that the concatenation of a taut path $\a$ to $\b$ and a taut path $\b$ to $\g$ changes the $i$-th coordinate exactly twice.

For a graded vector space (or module) $M$ and integer $k$, we write $M\langle k \rangle$ to denote the graded vector space shifted down by $k$, that is $(M\langle k \rangle)_i = M_{i+k}$.

\begin{defi}
Let $\wt B=\wt B(\PB)$, as a graded vector space in non-negative degree, be defined as 
	\[ \wt B:=\bigoplus_{(\a,\b)\in\PC\times \PC}\wt R_{\a\b}\langle -d_{\a\b} \rangle \]
	where the variables $t_i$ are given degree 2. 
	
Following~\cite{GDKD} we define a multiplication $\star \colon \wt B\otimes \wt B\to \wt B$ as zero on $\wt R_{\a\b}\otimes \wt R_{\d\g}$ if $\b\neq\d$ and for $\a,\b,\g\in\PC$,
by the composition 
	\[\wt R_{\a\b}\otimes \wt R_{\b\g}\to \wt R_{\a\b\g}\otimes \wt R_{\a\b\g}\to \wt R_{\a\b\g} \stackrel{f_{\a\b\g}}{\to} \wt R_{\a\g} \]
where the first map is the product of the restrictions, the second is multiplication in $\wt R_{\a\b\g}$, and the third map $f_{\a\b\g}$ is induced by multiplication by $t_{S^\b_{\a\g}}$.  
\end{defi}

\begin{lem}
The map $f_{\a\b\g}$ is well-defined.
\end{lem}
\begin{proof} 
It suffices to check that if $t_S = 0$ in $\wt R_{\a\b\g}$, then $t_{S^\b_{\a\g}} t_S= 0$ in $\wt R_{\a\g}$.  In other words, we want to show if $S\not\subset z(Y)$ for all $Y \in \D_{\a\b\g}$, then $S \cup S^\b_{\a\g}\not\subset z(Y')$ for all $Y' \in \D_{\a\g}.$  Suppose $Y' \in \D_{\a\g}$ and either: (1) $Y' \not\in \D_{\b}$ or (2) $Y' \in \D_{\b}$.

In the first case, $Y' \not\in \D_\b$ means that there exists an $i$ such that $Y'(i)=-\b(i)$, in particular $Y'(i)\neq 0$ and $i\not\in z(Y')$.  As $Y'$ is a face of $\a$ and $\g$, it follows that $Y'(i)=\a(i)=\g(i)$ and so $i \in S^\b_{\a\g}$. Thus $S \cup S^\b_{\a\g}\not\subset z(Y')$.  

In the second case, we have $Y' \in \D_{\a\b\g}$ hence by assumption $S \not\subset z(Y')$ and so again $S \cup S^\b_{\a\g} \not\subset z(Y')$.
\end{proof}

\begin{prop}
The multiplication $*$ gives $\wt B$ the structure of a graded ring.
\end{prop}

\begin{proof}
We need to check associativity and compatibility with grading. For associativity, the map $\wt R_{\a\b}\otimes \wt R_{\b\g}\otimes \wt R_{\g\d}\to \wt R_{\a\d}$ given by $x\otimes y\otimes z\mapsto (x\star y)\star z$ is equal to the map given by restricting each of the components to $\wt R_{\a\b\g\d}$, multiplying in order, and then multiplying by $t_{S^\b_{\a\g}}\cdot t_{S^\g_{\a\d}}$ to get back into $R_{\a\d}$. For $x\star(y\star z)$, the only change is that we multiply by $t_{S^\g_{\b\d}}\cdot t_{S^\b_{\a\d}}$. To show
  \[t_{S^\b_{\a\g}}\cdot t_{S^\g_{\a\d}}=t_{S^\g_{\b\d}}\cdot t_{S^\b_{\a\d}},\]
note that the power of $t_i$ appearing on each side is equal to the number of times a path given by the concatenation of taut paths from $\a$ to $\b$, $\b$ to $\g$, and $\g$ to $\d$ changes the $i$-th coordinate twice.

For the compatibility of gradings, note that $$d_{\a\b}+d_{\b\g}-d_{\a\g}=2|S^\b_{\a\g}|.$$ It follows that multiplication $\star$ gives a graded preserving map
  \[\wt R_{\a\b}\langle-d_{\a\b} \rangle \otimes_k\wt R_{\b\g} \langle -d_{\b\g} \rangle \to\wt R_{\a\g}\langle -d_{\a\g} \rangle.\] 
(Recall that $\deg(t_i)=2$ for all $i$.)
\end{proof}

Note that the map $\z\colon \Sym (k^E)\to\wt B$ given by the composition
\[\Sym (k^E)\into\bigoplus_{\a\in\PC}\Sym (k^E)\fib\bigoplus_{\a\in\PC} \wt R_{\a\a} \into \wt B\]
makes $\wt B$ into a graded $\Sym (k^E)$-algebra.  Moreover, this map factors through the projection $\Sym (k^E) \to k[M]$ and so we may also view $\wt B$ as a graded $k[M]$-algebra.

Let	\[ R_{\a_1\cdots \a_r}:= \wt R_{\a_1\cdots \a_r}\otimes_{\Sym U} k =\wt R_{\a_1\cdots \a_r}\otimes_{k[M]} k[M]/(U). \]

We define $B=B(\PB,U)$ via
\[ B:=\wt B\otimes_{\Sym U} k=\wt B\otimes_{k[M]} k[M]/(U).\]

Note, $B$ is itself a graded ring whose multiplication we will also denote by $\star$.

\begin{thm}\label{thm-A-B}
There is a natural isomorphism $\wt A(\PB^\vee)\to \wt B(\PB)$ as graded $\Sym (k^E)$-algebras, and this induces an isomorphism $A(\PB^\vee,U^\perp)\simeq B(\PB,U)$ as graded rings.
\end{thm}

In particular the theorem implies that $\wt A(\PB^\vee)$ is in fact a $k[M]$-module.

\begin{proof}
We define a map $\phi\colon \wt A(\PB^\vee)\to \wt B(\PB)$ by 
\begin{enumerate}
\item $e_\alpha\mapsto 1_{\a\a}\in \wt R_{\a\a}$ for all $\a\in\PC$,
\item $p^\vee(\a,\b)\mapsto 1_{\a\b}\in\wt R_{\a\b}\langle -1\rangle$ for adjacent $\a,\b\in\PC$, and 
\item $f\mapsto \z(f)$ for all $f\in \Sym (k^E)=k[t_i~|~i\in E]$.
\end{enumerate}
We will first show that this gives a well-defined homomorphism. We then prove surjectivity and conclude with injectivity.

We must check that the image of the relations ($A2$) and ($A3$) for $\wt\AC(\PB^\vee)$ hold in $\wt B(\PB)$. 

For ($A2$), we consider $\a\in\PC^\vee=\PC$ and $i\neq j$ in $E$ such that $\gamma=(\a^i)^j\in\PC^\vee$ and $\a^i,\a^j\in \FC^\vee$. This means that $S^{\a^i}_{\a\g}=\varnothing=S^{\a^j}_{\a\g}$, so that $t_{S^{\a^i}_{\a\g}}=1$.

If $\a^i$ and $\a^j$ are both in $\PC^\vee=\PC$, then we have $$1_{\a\a^i}\star 1_{\a^i\g}=1_{\a\g}=1_{\a\a^j}\star 1_{\a^j\g}.$$ 

Otherwise, by relabelling $i$ and $j$ if necessary, we may assume without loss of generality that $\a^i \not\in \PC^\vee$.  This means that in $\wt A(\PB^\vee)$, we  have 
\[0=p^\vee(\a,\a^i,\g)=p^\vee(\a,\a^j,\g).\]
On the other hand it also means that $\a^i \in \FC^\vee \backslash \PC^\vee = \BC \backslash \FC$.  But any common face of $\a$ and $\g$ must also be a face of $\a^i$, which is infeasible, hence $\a$ and $\g$ have no common feasible faces and so $\wt R_{\a\g}=0$, which means both products $1_{\a\a^i}\star 1_{\a^i\g}$ and $1_{\a\a^j}\star 1_{\a^j\g}$ must be zero.

We now check the relation (${A3}$). Let $\a\in\PC^\vee$ and $\a^i\in\FC^\vee$. 

If $\a^i\in\PC^\vee$, we have that 
	\eq{\phi(e_\a t_i)
		=1_{\a\a}\star \z(t_i)
		& = (t_i \in \wt R_{\a\a}) \\
		&=(t_{S_{\a\a}^{\a^i}} \in \wt R_{\a\a} )\\
		&=1_{\a(\a^i)}\star1_{(a^i)\a}
		=\phi(p(\a,\a^i,\a)).}

If $\a^i\not\in\PC^\vee$, then $\a^i\in\FC^\vee\backslash\PC^\vee$ and so $e_\a t_i = e_\a p^\vee(\a,\a^i,\a)=0$ in $\wt A(\PB^\vee)$. On the other hand, $\a^i \in \BC\backslash\PC$ implies that $i$ is not in the zero set of any feasible face of the feasible tope $\a$ of $\NC=\wt\MC\backslash f$, so $t_i = 0\in \wt R_{\a\a}$. This completes the proof that the relation (${A3}$) is satisfied.

Thus, the homomorphism $\phi$ is well-defined.

\medskip

To see that $\phi$ is surjective, note that $\phi(e_\a t_i)=1_{\a\a}t_i$ for all $i\in E$ and $\alpha\in\PC^\vee=\PC$. This means that $\bigoplus_\a \wt R_{\a\a}\subset \wt B$ is contained in the image of $\phi$. Since the natural quotient $\wt R_{\a\a}\to\wt R_{\a\a\b}=\wt R_{\a\b}$ is given by multiplication by $1_{\a\b}=\phi(p^\vee_{\a,\b})$ for any $\b\in\PC$ and $p^\vee_{\a,\b}$ representing a taut path from $\a$ to $\b$ in the quiver $Q^\vee$ associated to $\PB^\vee$, we have that $\phi$ is surjective.

\medskip

Finally, we must prove that $\phi$ is injective.  It suffices to show that the dimension of $\wt R_{\a\b}$ in each degree is at least the dimension of the corresponding graded part of $e_\a \wt A(\PB^\vee) e_\b$.   To do so, we construct a surjection of graded $\Sym(k^E)$-modules $\wt R_{\a\b} \to e_\a \wt A(\PB^\vee) e_\b$.

Let $\chi: \Sym(k^E)=k[t_i ~|~i\in E] \to e_\a \wt A(\PB^\vee) e_\b$ be the map that takes $1$ to a taut path $p$ from $\a$ to $\b$ in the quiver $Q^\vee$.  By Proposition~\ref{prop-tautification} and Proposition~\ref{prop-taut} this map is surjective.  It remains to show that $\chi$ factors through $\wt R_{\a\b}$, which is equivalent to showing that for any $S \not\subset z(Y)$ for all $Y \in \D_{\a\b}$, we have $\chi(t_S)=p \cdot t_S=0$. 
Notice that $S\not\subset z(Y)$ for all $Y\in\D_{\a\b}$ if and only if $S\cup S_{\a\a}^\b\not\subset z(Y)$ for all $Y\in\D_{\a}$, so we may reduce to the case $\a=\b$ if we replace $S$ with $S\cup S_{\a\a}^\b$.

By Corollary \ref{cor-thru} it suffices to prove the existence of $\g\in \FC^\vee\backslash \PC^\vee=\BC\backslash\PC$ such that $\g(i)=\a(i)$ for all $i\not\in S$. Since $S\not\subset z(Y)$ for any $Y\in \D_\a$, the image $\overline\a$ of $\a$ in $\PB/S=(\wt\MC/S,g,f)$ is not feasible. However, $\overline\a$ is bounded in $\PB/S$ since any cocircuit face $\overline X\in\AC^\infty(\PB/S)$ of $\overline\a$ comes from a cocircuit face $X\in\AC^\infty(\PB)$ of $\a$, and therefore has $\overline X(f)=X(f)=-$. Hence $\overline\a$ is a feasible but unbounded sign vector in $\PB^\vee\backslash S$, and lifts to at least one $\g\in\FC^\vee\backslash\PC^\vee$ with $\g(i)=\a(i)$ for all $i\not\in S$.
\end{proof}

\section{The center of $B$}\label{section-center} 

We continue to assume $\PB=(\wt\MC,g,f)$ is a generic oriented matroid program and $U \subset k^E$ is a parameter space for $M = \un\MC$. In this section we compute the centers of $\wt B:=\wt B(\PB)$ and $B:=B(\PB,U)$. 

Let $\overline\z$ be the composition 
 \eq{k[M]\hookrightarrow\bigoplus_{\a\in\PC}k[M]\to\bigoplus_{\a\in\PC}\wt R_{\a\a}\hookrightarrow\wt B} 
and recall that this map makes $\wt B$ a graded $k[M]$-algebra. 

\begin{thm}\label{thm-center}
The map $\overline\z\colon k[M]\to\wt B(\PB)$ is injective, and its image is the center of $\wt B(\PB)$. Furthermore, the quotient $\wt B(\PB)\to B(\PB,U)$ induces a surjection of centers and the center of $B(\PB,U)$ is isomorphic to $k[M]/(U)$.
\end{thm}

This is the natural generalization of~\cite[Theorem 4.16]{GDKD} and we imitate the structure of their proof, which makes use of the \emph{extended algebras}:
  \[\wt B_\ext=\wt B_\ext(\PB):=\bigoplus_{(\a,\b)\in\FC\times\FC} \wt R_{\a\b} \langle -d_{\a\b}\rangle
  \quad \text{and} \quad
  B_\ext=B_\ext(\PB,U):=\wt B_\ext\otimes_{\Sym U} k,\]
where the product $\star$ is defined as before but we use all feasible topes, not just the bounded feasible topes. 
Braden--Licata--Proudfoot--Webster first prove their result is true when $\wt B$ and $B$ are replaced by the extended algebras $\wt B_\ext$ and $B_\ext$.  This is done by studying a chain complex whose homology is the center $Z(\wt B_\ext)$. 
To get the theorem for $\wt B$ and $B$, they then use a categorical limit argument.

This section is split into two subsections. In the first subsection we define the necessary notation and lay the topological foundation for the proof. In the second subsection we adapt the arguments of \cite{GDKD} to our setting.

\subsection{The topology of affine space and feasible Edmonds--Mandel face lattices}
When $\PB$ realizable by a polarized arrangement $(V,\eta,\xi)$, it is possible to view $\AC$ as the cells of the coordinate hyperplane arrangement in $\eta+V \subset \RM^n$. This allows one to find tubular neighborhoods of intersections of hyperplanes, and compute the relative cellular Borel-Moore homology of these tubular neighborhoods using the decomposition by cells.   In this section we recall definitions and notation to generalize these notions to our setting.

Recall that $\AC$ and $\AC^\infty$ were defined respectively in Definition~\ref{def-covect} as the feasible and boundary covectors. We also define the \emph{core} of $\AC$ to be 
    \[\AC^0:=\{Y\in\AC\mid Y^\infty=0\}.\] 
For any $i\in E $ and $S\subset E $, let 
    \eq{H_i^\FC:=\{X\in\AC\mid X(i)=0\}
    \quad\text{and}\quad
        H_S^\FC:=\bigcap_{i\in S} H_i^\FC.
        }
Our genericity assumption on $g$ implies that $H_S^\FC\neq\varnothing$ if and only if $S$ is independent in the underlying matroid $M$ of $\MC$, in which case any maximal covector in $H_S^\FC$ has rank $d+1-|S|$. Note that $H_S^\FC$ does not have a cellular interpretation\footnote{As was discussed in the case of $H_\varnothing^\FC= \AC$ following Definition~\ref{def-feaslattices}.} unless $S$ is a basis of $M$, in which case $H_S^\FC$ consists of a single feasible cocircuit.

It is known~\cite[Theorem 4.5.7.i]{OMbook} that $\|\D_{\ord}(\AC)\|$ is a shellable $d$-ball. 
Thus its boundary $\|\D_{\ord}(\AC\backslash\AC^0)\|$ is a PL $(d-1)$-sphere.
In $\|\D_{\ord}(\AC)\|$, we have that $\|\D_{\ord}(H_S^\FC)\|$ is a $(d-|S|)$-ball (when nonempty), with boundary $\|\D_{\ord}(H_S^\FC\backslash\AC^0)\|$. 

For any $Y\in\AC$, let 
    \eq{\s_Y:=\|\D_{\ord}(\flem^\FC(Y)\backslash \{0\})\| \subset \| \D_{\ord}(\AC)\|.}
    
\begin{lem}\label{lem-feas EM}
For any $Y \in \AC$, the order complex $\D_{\ord}(\flem^\FC(Y)\backslash \{0\})$ is a shellable $(\rho(Y)-1)$-ball.
\end{lem}

\begin{proof}
Recalling the notation from the proof of Lemma~\ref{lem-feas Las Vergnas}, note that $\D_{\ord}(\flem^\FC(Y)\backslash \{0\})$ is the cone on $\D_{\ord}(\D^\op)$ with vertex $Y$.  If $Y^\infty = 0$, then $\D_{\ord}(\D^\op)$  is a shellable sphere (see~\cite[Theorem 4.3.5(i)]{OMbook}) and if $Y^\infty \neq 0$, then in the proof of Lemma~\ref{lem-feas Las Vergnas} we showed that $\D_{\ord}(\D^\op)$ is a shellable ball.  In both cases, the cone is a shellable ball as claimed.
\end{proof}

The boundary of each $\s_Y$ is the union of cells $\s_X$ for the proper faces $X$ of $Y$ and the geometric simplices in $\D_{\ord}(\flem^\FC(Y)\backslash \{0\})$ corresponding to chains that do not begin with a cocircuit. 
Let $\Xi$ be the regular cell complex of cells $\{\s_Y\}_{Y\in \AC}$ together with the set of (geometric) simplices $\{\s\}_{\s\in\D_{\ord}(\AC\backslash\AC^0)}$.  In particular, $\D_{\ord}(\AC)$ is a subdivision of $\Xi$.  The cells $\s_Y$ for $Y \in \AC^0$ define a subcomplex of $\Xi$, as do those for $Y \in H_S^\FC\cap\AC^0$ when $S$ is independent in $M$.

\begin{remark}
Figure \ref{fig-ordercomplex} shows an example to illustrate some of these definitions. The reason we consider $\D_{\ord}(\AC)$ and $\Xi$ is because $\AC$ does not have a cellular interpretation unless we include a boundary. The natural boundary would be $\AC^\infty$, but this gives us undesirable topology at the boundary. By introducing $\D_{\ord}(\AC)$ and $\Xi$ we resolve this issue. 
\end{remark}

\begin{figure}
\begin{center}
\begin{tikzpicture}
  \def\radius{1cm}
  \draw[blue] circle[radius=\radius];
\path (0:\radius) coordinate (1+) node [circle,fill,inner sep=1pt]{};
\path (90:\radius) coordinate (2+) node [circle,fill,inner sep=1pt]{};
\path (135:\radius) coordinate (3+) node [circle,fill,inner sep=1pt]{};
\path (180:\radius) coordinate  (1-) node [circle,fill,inner sep=1pt]{};
\path (270:\radius) coordinate (2-) node [circle,fill,inner sep=1pt]{};
\path (315:\radius) coordinate (3-) node [circle,fill,inner sep=1pt]{};
\path (250:\radius) coordinate (g) node [below,blue] {g};
\draw[->,blue,thick] (g) -- (250:.85*\radius);

\draw (1-) arc [start angle=180, end angle=360, x radius=\radius, y radius=.4*\radius];
\draw (1+) -- (1-);
\draw (2-) -- (2+);
\draw[rotate=-45] (3-) arc [start angle=0, end angle=180, x radius=\radius, y radius=.4*\radius];
\end{tikzpicture}
\quad
\begin{tikzpicture}
  \def\x{.6}
\coordinate (V14) at (0,-1*\x);
\coordinate (V13) at (2*\x,-1*\x);
\coordinate (V24) at (0,0);
\coordinate (V34) at (0,1.25*\x);
\coordinate (V23) at (\x,0);

\coordinate (E11) at (-1*\x,-1*\x);
\coordinate (E12) at (\x,-1*\x);
\coordinate (E13) at (3*\x,-1*\x);

\coordinate (E21) at (-1*\x,0);
\coordinate (E22) at (.5*\x,0);
\coordinate (E23) at (2*\x,0);

\coordinate (E31) at (-1*\x,1.5*\x);
\coordinate (E32) at (.6*\x,.6*\x);
\coordinate (E33) at (1.5*\x,-.5*\x);
\coordinate (E34) at (2.25*\x,-2*\x);

\coordinate (E41) at (0,2*\x);
\coordinate (E42) at (0,.5*\x);
\coordinate (E43) at (0,-.5*\x);
\coordinate (E44) at (0,-2*\x);

\coordinate (T1) at (-.75*\x,-2*\x);
\coordinate (T2) at (1.25*\x,-2*\x);
\coordinate (T3) at (3*\x,-1.5*\x);
\coordinate (T4) at (-1*\x,-.5*\x);
\coordinate (T5) at (.75*\x,-.5*\x);
\coordinate (T6) at (2.5*\x,-.5*\x);
\coordinate (T7) at (-1*\x,.5*\x);
\coordinate (T8) at (.4*\x,.4*\x);
\coordinate (T9) at (1.5*\x,1.5*\x);
\coordinate (T10) at (-.5*\x,2*\x);

\node [circle,fill,inner sep=1pt] at (V14) {};
\node [circle,fill,inner sep=1pt] at (V13) {};
\node [circle,fill,inner sep=1pt] at (V24) {};
\node [circle,fill,inner sep=1pt] at (V23) {};
\node [circle,fill,inner sep=1pt] at (V34) {};
\node [circle,fill,inner sep=1pt] at (E11) {};
\node [circle,fill,inner sep=1pt] at (E12) {};
\node [circle,fill,inner sep=1pt] at (E13) {};
\node [circle,fill,inner sep=1pt] at (E21) {};
\node [circle,fill,inner sep=1pt] at (E22) {};
\node [circle,fill,inner sep=1pt] at (E23) {};
\node [circle,fill,inner sep=1pt] at (E31) {};
\node [circle,fill,inner sep=1pt] at (E32) {};
\node [circle,fill,inner sep=1pt] at (E33) {};
\node [circle,fill,inner sep=1pt] at (E34) {};
\node [circle,fill,inner sep=1pt] at (E41) {};
\node [circle,fill,inner sep=1pt] at (E42) {};
\node [circle,fill,inner sep=1pt] at (E43) {};
\node [circle,fill,inner sep=1pt] at (E44) {};
\node [circle,fill,inner sep=1pt] at (T1) {};
\node [circle,fill,inner sep=1pt] at (T2) {};
\node [circle,fill,inner sep=1pt] at (T3) {};
\node [circle,fill,inner sep=1pt] at (T4) {};
\node [circle,fill,inner sep=1pt] at (T5) {};
\node [circle,fill,inner sep=1pt] at (T6) {};
\node [circle,fill,inner sep=1pt] at (T7) {};
\node [circle,fill,inner sep=1pt] at (T8) {};
\node [circle,fill,inner sep=1pt] at (T9) {};
\node [circle,fill,inner sep=1pt] at (T10) {};

\draw[very thick] (E11)--(E12)--(E13);
\draw[very thick] (E21)--(E22)--(E23);
\draw[very thick] (E31)--(V34)--(E32)--(V23)--(E33)--(V13)--(E34);
\draw[very thick] (E41)--(E42)--(E43)--(E44);
\draw (T1)--(E11)--(T4)--(E21)--(T7)--(E31)--(T10) --(E41)--(T9)--(E23)--(T6)--(E13)--(T3)--(E34)--(T1);
\draw (T1)--(V14);
\draw (T2)--(V14)--(V13)--(T2)--(E12);
\draw (T3)--(V13);
\draw (V14)--(T4)--(V24)--(E43)--(T4);
\draw (T5)--(E43)--(V24)--(T5)--(E22)--(V23)--(T5)--
(E33)--(V13)--(T5)--(E12)--(V14)--(T5);
\draw (V13)--(T6)--(E33)--(V23)--(T6);
\draw (V24)--(T7)--(E42)--(V34)--(T7);
\draw (V24)--(T8)--(E42)--(V34)--(T8)--(E32)--(V23)--(T8)--(E22);
\draw (V23)--(T9)--(V34)--(E32)--(T9);
\draw (T10)--(V34);
\end{tikzpicture}
\quad
\begin{tikzpicture}
  \def\x{.6}
\coordinate (V14) at (0,-1*\x);
\coordinate (V13) at (2*\x,-1*\x);
\coordinate (V24) at (0,0);
\coordinate (V34) at (0,1.25*\x);
\coordinate (V23) at (\x,0);

\coordinate (E11) at (-1*\x,-1*\x);
\coordinate (E12) at (\x,-1*\x);
\coordinate (E13) at (3*\x,-1*\x);

\coordinate (E21) at (-1*\x,0);
\coordinate (E22) at (.5*\x,0);
\coordinate (E23) at (2*\x,0);

\coordinate (E31) at (-1*\x,1.5*\x);
\coordinate (E32) at (.6*\x,.6*\x);
\coordinate (E33) at (1.5*\x,-.5*\x);
\coordinate (E34) at (2.25*\x,-2*\x);

\coordinate (E41) at (0,2*\x);
\coordinate (E42) at (0,.5*\x);
\coordinate (E43) at (0,-.5*\x);
\coordinate (E44) at (0,-2*\x);

\coordinate (T1) at (-.75*\x,-2*\x);
\coordinate (T2) at (1.25*\x,-2*\x);
\coordinate (T3) at (3*\x,-1.5*\x);
\coordinate (T4) at (-1*\x,-.5*\x);
\coordinate (T5) at (.75*\x,-.5*\x);
\coordinate (T6) at (2.5*\x,-.5*\x);
\coordinate (T7) at (-1*\x,.5*\x);
\coordinate (T8) at (.4*\x,.4*\x);
\coordinate (T9) at (1.5*\x,1.5*\x);
\coordinate (T10) at (-.5*\x,2*\x);

\node [circle,fill,inner sep=1pt] at (V14) {};
\node [circle,fill,inner sep=1pt] at (V13) {};
\node [circle,fill,inner sep=1pt] at (V24) {};
\node [circle,fill,inner sep=1pt] at (V23) {};
\node [circle,fill,inner sep=1pt] at (V34) {};
\node [circle,fill,inner sep=1pt] at (E11) {};
\node [circle,fill,inner sep=1pt] at (E13) {};
\node [circle,fill,inner sep=1pt] at (E21) {};
\node [circle,fill,inner sep=1pt] at (E23) {};
\node [circle,fill,inner sep=1pt] at (E31) {};
\node [circle,fill,inner sep=1pt] at (E34) {};
\node [circle,fill,inner sep=1pt] at (E41) {};
\node [circle,fill,inner sep=1pt] at (E44) {};
\node [circle,fill,inner sep=1pt] at (T1) {};
\node [circle,fill,inner sep=1pt] at (T2) {};
\node [circle,fill,inner sep=1pt] at (T3) {};
\node [circle,fill,inner sep=1pt] at (T4) {};
\node [circle,fill,inner sep=1pt] at (T6) {};
\node [circle,fill,inner sep=1pt] at (T7) {};
\node [circle,fill,inner sep=1pt] at (T9) {};
\node [circle,fill,inner sep=1pt] at (T10) {};

\draw[thick] (E11)--(E12)--(E13);
\draw[thick] (E21)--(E22)--(E23);
\draw[thick] (E31)--(V34)--(E32)--(V23)--(E33)--(V13)--(E34);
\draw[thick] (E41)--(E42)--(E43)--(E44);
\draw (T1)--(E11)--(T4)--(E21)--(T7)--(E31)--(T10) --(E41)--(T9)--(E23)--(T6)--(E13)--(T3)--(E34)--(T1);
\end{tikzpicture}
\end{center}
\caption{$\|\AC\cup\AC^\infty\|$, $\|\D_{ord}(\AC)\|$, $\|\Xi\|$}\label{fig-ordercomplex}
\end{figure}

\begin{defi}\label{def-tubular}
For an independent set $S\subset E$ of $M$, we define
    \eq{\Sigma_S:= \{A \in \D_{\ord}(\AC)| A \subset (x_0<x_1<\dots<x_{i}) \in \D_{\ord}(\AC) \text{~and~} x_j \in H_S^\FC \text{~for~some~}j \} }
and let $N_S:=\|\Sigma_S\|^\circ$ be the interior of $\|\Sigma_S\|$.  In particular:
    \eq{N_S:= \{A \in \D_{\ord}(\AC)| A = (x_0<x_1<\dots<x_{i}) \text{~and~} x_j \in H_S^\FC \cap \AC^0 \text{~for~some~}j \} }
\end{defi}

\begin{prop} \label{prop-tubular}
    The Borel-Moore homology of $N_S$ is 
        \eq{\HBM_m(N_S;\ZM)=\begin{cases} \ZM & m=d\\ 0& m\neq d\end{cases}} and can be computed in $\|\Xi\|$ using the relative cellular Borel-Moore homology of $N_S$ via the decomposition by cells $N_S\cap\s_Y$ for $Y\in\AC$. These $N_S\cap\s_Y$ are nonempty exactly when $\flem^\FC(Y)\cap H_S^\FC\neq\varnothing$.
\end{prop}

\begin{proof}
We first show that $N_S$ is a $d$-ball, from which the first statement follows.

By~\cite[Theorem 4.5.7(i)]{OMbook}, $\| \Delta_{\ord} (H^\FC_S) \|$ is a shellable ball.  In particular, $\| \Delta_{\ord} (H^\FC_S) \|$ is collapsible  (i.e., it collapses to a point).

Now $N_S$ is a regular neighborhood of $\| \Delta_{\ord} (H^\FC_S) \|$ in the $d$-ball $\|\Delta_{\ord}(\AC)\|$ and so by \cite[Corollary 3.27]{RoSa}, $N_S$ is a $d$-ball.

To see that each intersection $N_S\cap\s_Y$ for $Y\in\AC$ is a cell, note that $N_S \cap \s_Y$ is a regular neighborhood of the shellable ball $H_S \cap \s_Y$ in $\s_Y$.  The same collapsibility argument from above then implies $N_S\cap\s_Y$ is a ball.

The cells $N_S\cap\s_Y$ provide a cellular decomposition of $N_S$ modulo its boundary and the space $N_S \cap \s_Y$ is nonempty if and only if a face of $Y$ is contained in $H_S^\FC$ or equivalently, $\fllv^\FC(Y)\cap H_S^\FC\neq\varnothing$.
\end{proof}

\begin{prop}\label{prop-tubular tope}
    Let $Y\in\AC\backslash\AC^0$ such that $Y$ has a face in $H_S^\FC\cap\AC^0$. Then 
    \eq{\HBM_m(N_S\cap\s_Y;\ZM)=0} for all $m$. 
    This can be computed in $\|\Xi\|$ using the decomposition of $N_S\cap\s_Y$ by cells $N_S\cap\s_X$ for $X\in\flem^\FC(Y)$, which are nonempty exactly for $X$ which have a face in $H_S^\FC$.
\end{prop}

\begin{proof}
The idea of this proof is to again use relative homology, this time on the pair $(\s_Y,\s_Y\backslash N_S)$. Recall that $\s_Y$ is a PL $(\r(Y)-1)$-ball by Lemma \ref{lem-feas EM}. Then $\s_Y$ is contractible, so we are done if we can prove the same for $\s_Y\backslash N_S$.

There is a unique maximal element of $H_S^\FC\cap\AC^0\cap\flem^\FC(Y)$, which we call $X$.
The complement of $N_S$ in $\s_Y$ is $\|\D\|$, where 
    \eq{\D:=\ord(\flem^\FC(Y)\backslash (H_S^\FC\cap\AC^0))\subset\ord(\flem^\FC(Y)).}
 Notice that $\D$ is nonempty since $Y^\infty\neq0$, and $\D$ is equivalent to the cone over $\ord(\flem^\FC(Y)\backslash\flem^\FC(X))$ with cone point $Y$. Thus $\|\D\|$ is contractible, and we are done.
\end{proof}

\subsection{The proof of Theorem~\ref{thm-center}}
We will construct a chain complex with homology isomorphic to the center $Z(\wt B_\ext)$. 

Let 
  \[\AC_r=\{Y\in\AC\mid \r(Y)-1=r\}=\{Y\in\AC\mid \dim(\s_Y)=r\},\]
where, as in Section~\ref{subsec-P}, $\r$ is the poset rank of $Y$ in $\LC$ and  by Lemma~\ref{lem-feas EM}, $\s_Y=\|\ord(\flem^\FC(Y))\|\subset\|\Xi\|$ is a ball.   
For example, $\AC_{d}$ is the set of feasible topes $\FC$ and $\AC_0$ is the set of feasible cocircuits. 

For all $Y\in\AC_r$, the space of orientations of $\s_Y$ is a one-dimensional vector space 
  \[or(Y):=\HBM_{r}(\s_Y^\circ;k),\]
  where $\s_Y^\circ$ denotes the interior of $\s_Y$.
There is a natural boundary map 
  \[\partial_Y: \;or(Y)\to \bigoplus_{X\in\AC_{r-1}\cap \flem^\FC(Y)} or(X).\]
Assembling all such maps, we obtain a chain complex on $\bigoplus_{Y\in\AC} or(Y)$, graded by $\dim(\s_Y)$, which computes the cellular Borel-Moore homology of $\|\Xi\|^\circ$. As $\|\Xi\|$ is a closed PL $d$-ball, this homology is one-dimensional in degree $d$ and zero in all other degrees.

For $Y \in \AC_r$, let 
    \eq{\wt R_Y :=k[z(\fllv^\FC(Y)\backslash\{0\})] = \wt R_{\a_1\cdots\a_\ell},} for any choice of $\a_1,\dots,\a_\ell \in\FC$ such that $Y=\a_1\wedge\cdots\wedge\a_\ell$. We define a chain complex $C_\bullet$ such that 
  \[C_r=\bigoplus_{Y\in\AC_r}\wt R_Y\otimes_k or(Y)\]
with differentials for each $Y \in \AC_r$
  \[\wt{R}_Y\otimes_k or(Y)\to \bigoplus_{X\in\AC_{r-1}\cap\flem^\FC(Y)} \wt R_X\otimes_k or(X)\]
   induced by the natural boundary maps $or(Y)\to or(X)$ and the quotients $\wt R_Y \to \wt R_X$ for each facet $X$ of $Y$.

\begin{lem}\label{lem-center homology extended}
Fix an orientation class $\O\in H_d^{BM}(\|\Xi\|^\circ;k)$, and let $\O_\a \in or(\a)$ be the restriction of $\O$ for any $\a\in\FC$. Let $\psi_\a : k[\un{\MC}]\to \wt{R}_\a$ denote the natural quotient map.
Then the homology of $C_\bullet$ is zero outside of degree $d$ and $k[\un\MC]\simeq H_d(C_\bullet)$ via the map 
\eq{x\mapsto \sum_{\a\in\FC}\psi_\a(x)\otimes\O_\a.}
\end{lem}

\begin{proof} Following the proof of \cite[Lemma 4.17]{GDKD}, for a monomial $m=\prod_i t_i^{s_i}$, let  $S=\{i\mid s_i>0\}$ and $C_\bullet^m\subset C_\bullet$ be the subcomplex consisting of all images of $m$, namely
	\[C_r^m=\bigoplus_{\tiny \begin{array}{c} Y\in\AC_r \\ H_S^\FC\cap\flem^\FC(Y)\neq\varnothing \end{array}} or(Y). \]

Note that the complex $C_\bullet$ decomposes as a direct sum $\oplus_m C^m_\bullet$ of subcomplexes because the terms of $C_\bullet$ are direct sums of quotients of $\Sym k^E = k[t_i | i\in E]$ by monomial ideals, while the differentials are induced by the identity map on $k[t_i | i\in E]$, up to sign.

If the set $S=\{i\mid s_i>0\}$ is dependent in $M$, then $H_S =\varnothing$ and so $C^m_\bullet=0$. 
If $S$ is independent in $M$, then $C_\bullet^m$ is the cellular Borel-Moore complex of the neighborhood $N_S\subset\|\Xi\|$, the homology of which is one-dimensional and concentrated in degree $d$  by Proposition \ref{prop-tubular}.
\end{proof}

\begin{prop}\label{prop-center extended}
    The obvious map $\overline\z_\ext\colon k[M]\to \wt B_\ext$ is injective, and its image is the center of $\wt B_\ext$. The quotient homomorphism $\wt B_\ext\to B_\ext$ induces a surjection of centers, and yields an isomorphism $Z(B_\ext)\cong k[M]/(U)$.
\end{prop}

\begin{proof}
 As in the proof of \cite[Proposition 4.18]{GDKD}, for an element $z \in \wt B_\ext$ to be in the center, it must commute with each idempotent $1_{\a\a}$ and thus be of the form $z= \sum_{\a \in \FC} z_\a$ where $z_\a \in R_\a$.  Similarly, using the fact that $z$ must commute with $1_{\a\b}$ for $\a$ and $\b$ adjacent, we find that 
$$\psi_{\a\b}(z_\a)=\psi_{\b\a}(z_\b)$$
where $\psi_{\a\b}: \wt R_\a \to \wt R_{\a\b}$ and $\psi_{\b\a}: \wt R_\b \to \wt R_{\a\b}$ denote the canonical quotient homomorphisms.  As $\wt B_\ext$ is generated by the elements $1_{\a\a}$ and $1_{\a\b}$ for adjacent $\a,\b$ and the image of $\overline\z_\ext$, it follows that:
    \begin{align}\label{center equation}
        Z(\wt B_{ext})\cong \left\{(z_\a)\in\bigoplus_{\a\in\FC} \wt R_\a\mid \psi_{\a\b}(z_\a)=\psi_{\b\a}(z_\b) \;\text{for all}\;\a \leftrightarrow \b \in \FC \right\}.
    \end{align}
On the other hand, $$H_d(C_\bullet) = \left\{y = \sum_{\a \in \FC} y_\a \otimes \O_\a \in C_d \mid \partial y=0\right\}$$ and the cycle condition $\partial y=0$ is equivalent to $\psi_{\a\b}(y_\a)=\psi_{\b\a}(y_\b)$ for all $\a \leftrightarrow \b \in \FC $.  We conclude that $ \overline\z_\ext$ induces an isomorphism
$$k[M] \cong H_d(C_\bullet) \cong  Z(\wt B_{ext}).$$

Finally, we can define a chain complex $\hat C_\bullet$ of free $\Sym U$-modules with $\hat C_m= C_m$ for $0\leq m\leq d$ and $\hat C_{d+1}=\ker(\partial_{d})\cong k[M]$. This is acyclic, and thus so is $\hat C_\bullet\otimes_{\Sym U} k$. Arguments analogous to above prove that 
$$k[M]/(U)\cong H_d(C_\bullet \otimes_{\Sym U} k)\cong Z(B_\ext)$$ through an isomorphism compatible with the quotient $\wt B_\ext\to B_\ext$.
\end{proof}

We now include boundedness into our considerations.
Define 
    \eq{\AC^\PC:=\bigcup_{Y\in\PC\subset\AC_d}\flem^\FC(Y),}
and notice that we have a chain of proper inclusions $\AC^0\subset\AC^\PC\subset\AC$.

Note that the description of $Z(\wt B_{ext})$ in (\ref{center equation}) can be rewritten as asking that $\psi_{\a\b}(z_\a)=\psi_{\b\a}(z_\b)$ for all $\a,\b \in \FC$, not necessarily adjacent.  This can be rephrased as a limit: 
 \begin{align}\label{center limit 1}
 Z(\wt B_\ext)\cong \varprojlim_{X\in\AC}\wt R_X.
 \end{align}
By the same sort of argument, we find
    \begin{align}\label{center limit 2}
    Z(B_\ext)\cong \varprojlim_{X\in\AC} R_X,
    \quad
    Z(\wt B)\cong \varprojlim_{X\in\AC^\PC}\wt R_X,
    \quad
    \text{and}
    \quad
    Z(B)\cong \varprojlim_{X\in\AC^\PC} R_X.
    \end{align}
The next lemma allows us to conclude that these centers only depend on $\AC^0$.

\begin{lem}\label{lem-limits}
    For any $Y\in\AC^\PC$ and any subcomplex $\DG\subset\flem^\FC(Y)$ with $\flem^\FC(Y)\cap \AC^0\subset\DG$, the restrictions
    \begin{align}\label{center limits lemma equation}
    \varprojlim_{X\in\flem^\FC(Y)}\wt R_X \to\varprojlim_{X\in\DG}\wt R_X
    \quad \text{and} \quad
    \varprojlim_{X\in\flem^\FC(Y)} R_X \to\varprojlim_{X\in\DG} R_X
    \end{align}
are isomorphisms.
\end{lem}

\begin{proof} 
This proof is simply a rephrasing of \cite[Lemma 4.20]{GDKD} in our setting.

This is trivial if $Y\in \flem^\FC(Y)\cap \AC^0$ (equivalently $Y^\infty=0$), which includes the case where $Y$ is a feasible cocircuit. So we may assume $Y^\infty\neq0$, and we may also inductively assume the statement is true if $Y$ is replaced by any $X\in\flem^\FC(Y)\backslash\{Y\}$. 

We will prove the statement first for $\DG=\flem^\FC(Y)\backslash\{Y\}$. Let $C_\bullet^Y$ be the subcomplex of $C_\bullet$ consisting only of the summands $\wt R_X\otimes_k or(X)$ for $X\in\flem^\FC(Y)$. As in the proof of Lemma \ref{lem-center homology extended}, this complex splits into a direct sum of complexes $C_\bullet^{Y,m}=C_\bullet^Y\cap C_\bullet^m$ for each monomial $m=\prod_i x^{s_i}$. The summand $C_\bullet^{Y,m}$ computes the cellular Borel-Moore homology of $N_S\cap\s_Y\subset \s_Y$, for $S=\{i\mid s_i>0\}$. By Proposition \ref{prop-tubular tope}, this Borel-Moore homology is trivial, so every $C_\bullet^{Y,m}$ is acyclic and so is $C_\bullet^Y$. 

We have that $\varprojlim_{X\in\DG}\wt R_X$ is isomorphic to the kernel of the boundary map $C_{d-1}^Y\to C_{d-2}^Y$. Since $C^Y_\bullet$ is acyclic, the first map of (\ref{center limits lemma equation}) is therefore an isomorphism. Similarly, the second map is an isomorphism since $C^Y_\bullet$ is an acyclic complex of free $\Sym U$-modules, which implies $C_\bullet^Y\otimes_{\Sym U}k$ is an acyclic complex of vector spaces.

For a general $\DG$ containing $\flem^\FC(Y)\cap \AC^0$, pick an ordering $X_1,\dots,X_r$ of the elements of $\flem(Y)\backslash \DG$ such that their ranks are nondecreasing, and let $\DG_\ell=\DG\cup\{X_1,\dots,X_\ell\}$. Then for $1\leq \ell\leq r$, we have $\flem^\FC(X_\ell)\backslash\{X_\ell\}\subset\DG_{\ell-1}$, so an identical argument shows that
    \eq{\varprojlim_{X\in\DG_\ell}\wt R_X \to\varprojlim_{X\in\DG_{\ell-1}}\wt R_X
    \quad \text{and} \quad
    \varprojlim_{X\in\DG_{\ell}} R_X \to\varprojlim_{X\in\DG_{\ell-1}} R_X}
are isomorphisms.
\end{proof}

\begin{proof}[Proof of Theorem \ref{thm-center}]
Put the equations (\ref{center limit 1}) and (\ref{center limit 2}) together with Lemma \ref{lem-limits} to get 
    \eq{
    Z(\wt B_\ext)
    \cong \varprojlim_{X\in\AC}\wt R_X 
    \cong\varprojlim_{X\in\AC^\PC}\wt R_X
    \cong Z(\wt B)}
and 
   \eq{
    Z( B_\ext)
    \cong \varprojlim_{X\in\AC} R_X 
    \cong\varprojlim_{X\in\AC^\PC} R_X
    \cong Z(B).}
All of these isomorphisms are compatible with $\overline\z_\ext$, $\overline\z$ and the natural quotients $\wt B_\ext\to B_\ext$ and $\wt B\to B$, so we are done.
\end{proof}

\section{The module category of $A$}\label{sec-modules}

In this section, we study the simple modules for $A=A(\PB,U)$ and their projective covers using a class of standard modules.  

\begin{defi} For any $\alpha\in\PC$, let 
\[L_\alpha:=A/(e_\beta\mid\beta\neq\alpha).\]
Then $L_\alpha$ is the simple one-dimensional $A$-module supported at $\alpha$ and each simple $A$-module is isomorphic to $L_\alpha$ for some $\a\in \PC$.  Let
\[P_\alpha:=e_\alpha A\]
be the projective cover of $L_\alpha$. We also define $V_\alpha$ to be $P_\alpha/K_\alpha$, where
\[K_\alpha:=\sum_{i\in b}p(\alpha,\alpha^i)\cdot A \subset P_\a \] and $b$ is the basis of $M$ such that $\mu(b) = \a$ under the bijection of Corollary~\ref{cor-bij}.
We refer to $V_\a$ as the \emph{standard module} and $L_\a$ as the \emph{simple module} associated to $\a$.
\end{defi}

\begin{lem}\label{lem-standard}
Let $\a\in\PC$.
 The standard module $V_\alpha$ has a basis consisting of a taut path from $\alpha$ to each $\beta\preceq \a$.  
\end{lem}

\begin{proof}
We simply copy the argument from~\cite[Lemma 5.21]{GDKD}.  It is clear that the collection of such taut paths is linearly independent.  We now show that the image of any other path is trivial in $V_\a$.  

Let $b = \mu^{-1}(\a) \in \BM$.

Suppose $p$ is a taut path from $\a$ to $\g \in \FC$ and $\g \not \preceq \a$.  Then for some $i\in b$, $\g(i) \neq \a(i)$ and $\a^i \in \FC$.  By Corollary~\ref{cor-thru} $p$ can be replaced by one of the form $p(\alpha,\alpha^i) \cdot x \in K_\alpha$. Thus $p=0$ as an element of $V_\a$.

If $p$ is a non-taut path, then by Proposition~\ref{prop-tautification}, we can write 
    \eq{p = p' \cdot \prod_{i} u_i^{a_i}=\left(\prod_{i} e_\a u_i^{a_i}\right)\cdot p'}
where $p'$ is taut (with the same endpoint as $p$) and $a_i>0$ for some $i$. Corollary~\ref{cor-lin combo} implies that for all $i\in E$ and some $c_{i,j}\in k$, we have
    \eq{e_\a u_i=\sum_{j\in b}c_{i,j} p(\a,\a^j,\a)\in K_\alpha.}
Thus, $p \in K_\alpha$ and we are done.
\end{proof}

\begin{cor}\label{cor-hwStandard}
The kernel of $V_\a\fib L_\a$ has a filtration with subquotients isomorphic to $L_\b$ for $\b\prec \a$, each appearing exactly once.
\end{cor}

\subsection{$A$ is a quasi-hereditary algebra when $\PB$ is Euclidean} 
Recall that a finite-dimensional algebra is \textit{quasi-hereditary} if its category of finitely generated modules is \textit{highest weight} in the following sense.

\begin{defi}\label{def-hwc}
Let $\CC$ be an abelian, artinian $k$-linear category and let $\IC$ be the set indexing the isomorphism classes of simple objects $\{L_\a \mid \a \in \IC\}$ and indecomposable projective objects $\{P_\a \mid \a \in \IC\}$.  Then $\CC$ is a \textit{highest weight category} if the set $\IC$ can be endowed with a partial order $\leq$ and there exists a collection of objects $\{V_\a \mid \a \in \IC\}$ with surjections
$$ P_\a \to V_\a \to L_\a $$
that satisfy:
\begin{itemize}
  \item[(i)] the kernel of $V_\a\to L_\a$ has a filtration for which each subquotient is isomorphic to $L_\g$ for some $\g<\a$, and
  \item[(ii)] The kernel of $P_\a\to V_\a$ has a filtration for which each subquotient is isomorphic to $V_\b$ for some $\b>\a$.
\end{itemize}
\end{defi}

Now consider the category of finitely generated $A(\PB,U)$-modules.  As discussed above, the isomorphism classes of simple modules are indexed by the set $\PC$ of bounded feasible topes.
For $A=A(\PB,U)$ to be a quasi-hereditary algebra, we will assume that the oriented matroid program $\PB$ is Euclidean.  Recall from  Section~\ref{subsec-euclid}, that this implies there is partial order $\leq$ on $\PC$ defined by: $\a\leq\b$ when there exists a directed sequence of edges from $\mu^{-1}(\a)$ to $\mu^{-1}(\b)$ in the graph $G_\PB$ of the program.  By Lemma~\ref{lem-EucCone}, this is the same partial order as that defined by the transitive closure of the cone relation $\preceq$.

Suppose $\PB$ is Euclidean.  Then the category of $A(\PB,U)$-modules and the poset $\IC=(\PC,\leq)$ satisfies condition (i) of Definition~\ref{def-hwc} by Corollary~\ref{cor-hwStandard}.  To show that $A$ is quasi-hereditary it remains to show condition (ii).

We will use the following simple lemma.

\begin{lem}\label{lem-icross}
Suppose $\b \in \PC$ and $i \in \mu^{-1}(\b)$.  Then the feasible sign vector $\b^i$ is either unbounded or $\b^i \succeq \b$.
\end{lem}

\begin{proof}
Suppose $\b^i$ is bounded.  As $\b^i(j)=\b(j)$ for all $j\neq i$, if $i \not\in \mu^{-1}(\b^i)$, then for all $j \in \mu^{-1}(\b^i)$, $\b^i(j) = \b(j)$.  Thus $\b \in \BC_{\mu^{-1}(\b^i)}$ or equivalently $\b \preceq \b^i$.

If $\b^i$ is bounded and $i \in \mu^{-1}(\b^i)$, then the optimal solutions of $\b$ and ${\b^i}$ are also optimal solutions of their common subtope $Y = \b \wedge \b^i$.  But $Y$ is a tope of $\PB/\{i\}$ and so $Y$ has a unique optimal solution.  Thus $\mu^{-1}(\b)=\mu^{-1}(\b^i)$, which contradicts the fact that $\mu$ is bijection.  
\end{proof}

\begin{thm}\label{thm-projfilt}
Assume $\PB$ is Euclidean. Then the kernel of the quotient homomorphism $P_\a\fib V_\a$ has a filtration with each successive subquotient isomorphic to $V_\b$ for $\b\succ \a$. Each of these standard modules appears exactly once.

In particular, $A$ is quasihereditary.
\end{thm}

\begin{proof} 
For any $\g\in\PC$, we define $P_\a^\g\subset P_\a$ to be the submodule generated by paths which pass through $\g$. For any $\b\in\PC$, let 
\[K_\a^\b:=\sum_{\g>\b}P_\a^\g.\] 
After choosing a total order on $\{\b\in\PC\mid\a\leq\b\}$ refining $\leq$, the set of submodules $P_\a^\b+K_\a^\b$ with $\b\geq\a$ forms a filtration of $K_{\a}$ with successive subquotients  $$M_\a^\b:= \left(P_\a^{\b}+K_\a^\b\right)/K_\a^{\b}.$$ 

We pause to note that to make the definitions above, it is essential that $\PB$ is Euclidean, for if $\PB$ were not Euclidean there would be no partial (or total) order refining the cone relation on $\PC$.

Our goal now is to prove that $M_\a^\b$ is zero if $\b\not\succeq \a$, and is isomorphic to $V_\b$ if $\b\succeq \a$. Notice that $M_\a^\a=V_\a$. 

If $\b\not\succeq \a$, then there is an index $i\in \mu^{-1}(\b)$ such that $\a(i)\neq\b(i)$. By Proposition \ref{prop-tautification}, any path starting at $\a$ and passing through $\b$ can be written as $p_{\a,\b} \cdot r$ in $P_\a^{\geq\b}$, where $p_{\a,\b}$ represents a taut path from $\a$ to $\b$ and $r$ represents a path starting at $\b$.  We may then apply Corollary \ref{cor-thru} to the taut path $p_{\a,\b}$ and $\g = \b^i$, to show that $p_{\a,\b}$ can be chosen to pass through $\b^i\in\FC$. By Lemma~\ref{lem-icross}, $P_\a^\b\subset P_\a^{\b^i}\subset K_\a^\b$, so $M_\a^\b=0$. 

On the other hand, assume that $\b\succ \a$. There is a natural map $P_\b\to P_\a^\b$ given by composing any element of $P_\b$ with a fixed taut path $p_{\a,\b}$ from $\a$ to $\b$.  This induces a homomorphism $V_\b \to M_\a^\b$ that we wish to show is an isomorphism. 

By Proposition~\ref{prop-tautification}, any path starting at $\a$ and passing through $\b$ can be expressed as a product of an element in $P_\b$ with some taut path from $\a$ to $\b$ and by Proposition~\ref{prop-taut} the taut path can be chosen to be the one we have fixed. It follows that the map $P_\b\to P_\a^\b$ is surjective and thus the induced map $V_\b \to M_\a^\b$ is surjective as well.

Finally, we need to show that $V_\b\to M_\a^\b$ is injective. We proceed by showing that they have the same dimension. The surjectivity of the map implies $$\dim_k M_\a^\b\leq\dim_k V_\b=|\{\g \in \PC \mid \g \preceq\b\}|,$$ so that 
	\[\dim_k P_\a=\sum_{\b \succeq \a}\dim_k M_\a^\b \leq |\{(\g,\b)\in\PC\times\PC \mid \a,\g \preceq \b\}|.\]
We're done if we can show this is an equality. As $A=\sum_{\a\in\PC} P_\a$, it suffices to prove that 
	\[ \dim_k A = \sum_{\a\in\PC} \dim_k P_\a = \{(\a,\g,\b)\in\PC\times\PC\times\PC \mid \a,\g\preceq \b\}. \]
But recall that $$A=A(\PB,U)\simeq B(\PB^\vee,U^\perp)=\bigoplus_{(\a,\g)\in\PC^\vee\times\PC^\vee} R_{\a\g}^\vee$$ and so by Lemma \ref{lem-size}
	\[ \dim_k B(\PB^\vee,U^\perp)
	=\sum_{(\a,\g)\in\PC^\vee\times\PC^\vee}|\{\text{common feasible cocircuit faces of $\a$ and $\g$}\}|. \]
 We are then reduced to showing that the number of common feasible cocircuit faces in $\PB^\vee$ of $\a$ and $\g$ is equal to the number of bounded feasible topes $\b$ of $\PB$ such that $\a\preceq \b$ and $\g\preceq \b$.
This follows from Complementary Slackness (Proposition~\ref{prop-compslack}).
\end{proof}

\subsection{The structure of projectives when $\PB$ is not Euclidean}\label{subsec-quasihereditary needs Euclidean}
Note that the definition of the standard modules makes sense for any $\PB$ and Lemma \ref{lem-standard} holds even in the non-Euclidean case.  However, when $\PB$ is not Euclidean, the transitive closure of the cone relation is not a poset and so the standard modules are not part of a highest weight structure.

Nonetheless, one might still hope for a version of Theorem~\ref{thm-projfilt}: that the kernel of $P_\a\fib V_\a$ has a filtration with successive subquotients isomorphic to $V_\b$ for $\b\succ \a$.

In this section we observe that this is too optimistic a hope, but that it does hold on the level of graded Grothendieck groups.

\medskip

Recall from Lemma~\ref{lem-size}, that for any $\a\in \PC$, the dimension $\dim_k R_\a$ is equal to the number of feasible cocircuit faces of $\a$.  We begin with a graded refinement of this statement.

\begin{lem}\label{lem-gradedsize}
Let $(h_0,h_1 \ldots, h_{d-1})$ denote the $h$-vector of $z(\D_{\a})$ or equivalently $h_i$ is equal to the dimension of the graded piece of $R_\a$ of degree $2i$.  Then $h_i$ is equal to the number of feasible vertices of $\a$ with $i$ outgoing edges.
\end{lem}

\begin{proof}
We proceed by showing that $z(\D_{\a})$ is partitionable.  Recall that a pure simplicial complex $\D$ is \textit{partitionable}, if it can be expressed as a disjoint union of closed intervals of the form
\[ \D = [G_1,F_1] \sqcup \ldots \sqcup [G_s,F_s], \]
where each $F_i$ is a facet of $\D$.  By \cite[Proposition III.2.3]{StanleyComb} the $h$-polynomial of such a simplicial complex is given by
\[ h_i = \#\{ j: |G_j| = i \}. \] 

Recall that $z(\a)=\varnothing$ and $z(\D_\a)$ is isomorphic as a poset to $\fllv^\FC(\a) \backslash \{0\}$.  Thus we may identify a face of $\a$ with the faces of the abstract simplicial complex $z(\D_\a)$.

The facets $F_1,\ldots,F_s$ of $z(\D_\a)$ are the zero sets of the feasible vertices (i.e., feasible cocircuit faces) of $\a$.
If $F_i=z(X_i)$ for a feasible vertex $X_i$, let $G_i$ be (the zero set of) the meet in $\D_\a$ of the incoming edges of $X_i$.

Recall that each face of $\a$ has a unique optimal solution (this follows from Theorem~\ref{thm-optimal}).  For each feasible face $Y$ of $\a$, the face $z(Y) \in \D_\a$ is in the interval $[G_j,F_j]$ if and only if $X_j$ is the optimal solution of $Y$.  Thus $[G_1,F_1] \sqcup \ldots \sqcup [G_s,F_s]$ is a partition of $z(\D_{\a})$.  Note that 
\[|G_j| = d - \#\{\textrm{incoming~edges~to~}X_j\} =  \#\{\textrm{outgoing~edges~from~}X_j\}.\]
We conclude that
\[ h_i = \#\{ j: |G_j| = i \} = \#\{\textrm{feasible~vertices~of~}Y \textrm{~with~}i \textrm{~outgoing~edges}\}.\] 
\end{proof}

\begin{cor}\label{cor-gradedsize}
Let $(h_0,h_1 \ldots, h_{d-1})$ denote the $h$-vector of $z(\D_{\a\b})$ or equivalently $h_i$ is equal to the dimension of the graded piece of $R_{\a\b}$ of degree $2i$.  Then $h_i$ is equal to the number of feasible vertices of $\a\wedge \b$ with $i$ outgoing edges of $\a\wedge\b$.
\end{cor}

\begin{proof}
Let $\g$ be the tope in $\PB/z(\a\wedge \b)$ given by the restriction of $\a\wedge\b$.  Then the simplicial complex $z(\D_{\a\b})$ is equal to the simplicial join $z(\D_\g)* \G$ of $z(\D_\g)$ with the $(d_{\a\b}-1)$-simplex $\G$ on the set $z(\a\wedge \b)$.  By standard properties of the $h$-polynomial, we have:
\[ h(z(\D_{\a\b}),x) = h(z(\D_\g)* \G, x) = h(z(\D_\g),x) h(\G,x) = h(z(\D_\g),x).\]
We conclude that the $h$-vector of $z(\D_{\a\b})$ is equal to that of $z(\D_\g)$.  The result then follows from Lemma~\ref{lem-gradedsize}.
\end{proof}

For an $A$-module $M$, let $[M]$ denote the class of $M$ in the Grothendieck group of $A$-modules.  We will consider the Grothendieck group of the category of graded $A$-modules as a $\ZM[q,q^{-1}]$-module, where 
\[ [M\langle -k\rangle ] = q^k [M]. \]
For a graded vector space $V = \oplus_i V_i$, we denote the graded dimension of $V$ by
\[ \grdim V = \sum_i (\dim V_i)~q^i. \]

\begin{thm}\label{thm-Kgroup}
For any generic oriented matroid program $\PB$ and any $\a \in \PC$, the class of the indecomposable projective $P_\a$ in the Grothendieck group can be expressed as the sum:
\[ [P_\a] = \sum_{\g \succeq \a} q^{d_{\a\g}}[V_\g].\]
\end{thm}

\begin{proof}
For any $\b\in \PC$ the graded composition series multiplicity of the simple $L_\b$ in the projective $P_\a$ is equal to the graded dimension of the space of paths in $A$ that start at $\a$ and end at $\b$.  In other words we have: 
\[[P_\a] = \sum_{\b\in \PC} (\grdim ~P_\a e_\b) \cdot [L_\b] =  \sum_{\b\in \PC} (\grdim ~e_\a A e_\b)\cdot [L_\b].\]
By Theorem~\ref{thm-A-B}, 
\[\grdim~e_\a A e_\b = \grdim ~R^\vee_{\a\b}\langle -d_{\a\b} \rangle = q^{ d_{\a\b}} \cdot \grdim~ R^\vee_{\a\b}.\]

By Corollary~\ref{cor-gradedsize}, we can express the graded dimension of $R^\vee_{\a\b}$ as
\[ \grdim ~R^\vee_{\a\b} = \sum_{i=0}^d  \#\{ \mathrm{feasible~vertices~of~}\a\wedge\b \mathrm{~in~}\PB^\vee \mathrm{~with~}i\mathrm{~outgoing~edges}\} \cdot q^{2i}. \]

 Observe that by Proposition~\ref{prop-compslack}  the feasible vertices of $\a\wedge\b$ in $\PB^\vee$ (i.e., common feasible vertices of both $\a$ and $\b$) are in bijection with the bounded feasible topes $\d$ of $\PB$ such that $\a \preceq \d$ and $\b\preceq \d$.  We claim that the number of outgoing edges of $\a\wedge\b$ of the feasible vertex corresponding to $\d$ is equal to $|S^\d_{\a\b}|$. As in the proof of Corollary~\ref{cor-gradedsize}, let $\g$ be the bounded feasible tope $(\a\wedge\b) |_{\un{\a\wedge \b}}$ in the contraction $\PB^\vee/z(\a\wedge \b)$.  Then the number of outgoing edges of the vertex of $\g$ corresponding to $\d$ is equal to the distance between $\g$ and the restriction $\ov\d$ of $\d$ to $\un{\a\wedge\b}$.  But $\un{\a\wedge\b}= \{i \in E \mid \a(i)=\b(i)\}$ and so the distance between $\g$ and $\ov\d$ is equal to the cardinality of the difference set
 $$S(\g,\ov\d) = \{ i \in E \mid \a(i)=\b(i) \neq \d(i) \} = S_{\a\b}^\d.$$

Rewriting the sum over topes $\d$ of $\PB$ such that $\a \preceq \d$ and $\b\preceq \d$ and using the formula $d_{\a\d}+d_{\d\b} = d_{\a\b}+2|S_{\a\b}^\d|$, we find:
\begin{equation} \label{eq-grdim}
 \grdim~e_\a A e_\b = q^{d_{\a\b}} \grdim ~R^\vee_{\a\b} = \sum_{\d \succeq \a,\b}  q^{d_{\a\b} + 2|S_{\a\b}^\d|} =  \sum_{\d \succeq \a,\b}  q^{d_{\a\d}+d_{\d\b}}.
 \end{equation}

Putting it all together, 
\[[P_\a] = \sum_\b   \sum_{\d \succeq \a,\b}  q^{d_{\a\d}+d_{\d\b}} [L_\b] \]
\[= \sum_{\d \succeq \a} q^{d_{\a\d}} \sum_{\b \preceq \d} q^{d_{\b\d}}[L_\b] \]
\[= \sum_{\g \succeq \a} q^{d_{\a\d}} [V_\g] ,\]
as we wished to show.
\end{proof}

We conclude this section with an example of a generic non-Euclidean program and sign vector $\a$ for which the kernel of $P_\a\fib V_\a$ does not admit a filtration with successive standard subquotients.

\begin{ex}\label{ex-EFM(8)}
Let $\PB=(\text{EFM}(8),g,f)$ be the generic non-Euclidean program defined in \cite[Section 10.4]{OMbook}.
Then $M=\un{\text{EFM}(8)/g\backslash f}$ is the uniform matroid of rank 3 on $E_6$. As short hand, we simply write $ijk$ for the basis $\{i,j,k\}$ of $M$.  We denote the sign vector of a bounded feasible tope $\a: E_6\to\{0,+,-\}$ using the string of signs $$\a(1)\a(2)\a(3)\a(4)\a(5)\a(6).$$ The bijection $\mu$ between $\BM$ and $\PC$ can described as follows, where we have listed the pairs $(b,\mu(b))\in \BM\times\PC$ for $\PB$:
\eq{(123,++++++)\
    (124,+++-+-)\
    (123,++-+++)\
    (126,+-+++-)\\
    (134,+-++++)\
    (135,-+++-+)\
    (136,++++--)\
    (145,-++---)\\
    (146,+-+-+-)\
    (156,+++++-)\
    (234,++--++)\
    (235,+++--+)\\
    (236,-+++++)\
    (245,++---+)\
    (246,+++-++)\
    (256,+-+---)\\
    (345,++++-+)\
    (346,++----)\
    (356,-+++--)\
    (456,+++---).
  }
Using this table one can deduce the cone relation $\preceq$ on $\PC$ from the fact that $\mu(b) \preceq \mu(b')$ if $\mu(b)(i)=\mu(b')(i)$ for any $i \in b'$.  For example, if $\mu(b) \prec \mu(456)$, then $b=346,145$ or $256$.

Let $\a=+++--- \in\PC$. Recall the notation $\a^S$ denoting the sign vector of a tope which differs from $\a$ on exactly the set $S\subset E_6$. Using the above list, we find that
 \eq{\{\b \in \PC \mid \b \succeq \a\}=\{\a,\a^4,\a^5,\a^6,\a^{\{1,4\}},\a^{\{2,5\}},\a^{\{3,6\}},\a^{\{4,5,6\}}\}
    }

Suppose there were a filtration $0 \subset F_1 \subset \ldots \subset F_6 \subset K_\a$ of the kernel $K_\alpha:=\sum_{i\in b}p(\alpha,\alpha^i)\cdot A$ of $P_\a\to V_\a$ with nonzero successive standard subquotients $\{V_\g\mid \g\in\CC_\a \backslash \{\a\} \}$ as in the proof of Theorem \ref{thm-projfilt}.  Let $V_\g=K_\a/F_6$ be the final standard subquotient and suppose that $\g=\a^S$.

Let $p$ be a taut path from $\a$ to $\a^S$.  If $p\in F_6$, then $(K_\a/F_6)e_{\a^S}=0$, which is a contradiction.  Thus we may assume that $p \not\in F_6$.  For any $i \in S$, $p=p(\a,\a^i)q$ where $q$ is a taut path from $\a^i$ to $\a^S$.  As $p \not \in F_6$, it follows that $p(\a,\a^i) \not\in F_6$ and $L_{\a^i}$ is a quotient of $V_\g$.  This is a contradiction unless $S=\{i\}$.  Thus $S$ is either $\{4\}, \{5\}$ or $\{6\}$. 

Suppose $S=\{4\}$, so $\g=\a^4=++++--$. Note that $\a^6 \prec \a^4$, so $(V_{\a^4})e_{\a^6}=(K_\a/F_6)e_{\a^6} \neq 0$ and thus $p(\a,\a^6) \not\in F_6$.  But this would mean that $L_{\a^6}$ is a quotient of $V_{\a^4}$, which is a contradiction.

After permuting indices, the same argument shows that neither $V_{\a^5}$ nor $V_{\a^6}$ is a quotient of $K_\a$.  We conclude that $P_\a$ does not have the expected filtration.  More generally, we will see below in the proof of Theorem~\ref{thm-numKoszul} that the change of basis matrix between the standard and simple bases for the Grothendieck group is invertible, so $[P_\a]$ cannot be expressed as a different sum of standard classes.  Thus $P_\a$ does not admit a filtration by standards.
\end{ex}

\subsection{$A$ is a Koszul algebra when $\PB$ is Euclidean}

Recall the notion of a Koszul algebra:

\begin{defi}
Let $M=\bigoplus_{\ell\geq0}M_\ell$ be a graded $k$-algebra. A complex
  \[\dots\to P_3\to P_2\to P_1\to P_0\] 
of graded projective right $M$-modules is \emph{linear} if each $P_\ell$ is generated in degree $\ell$. We say that $M$ is \emph{Koszul} if every simple right $M$-module has a linear projective resolution.
\end{defi}

\begin{thm}\label{thm-KoszulStan}
Assume $\PB$ is Euclidean. Then for all $\a\in\PC$, the standard module $V_\a$ has a linear projective resolution.
\end{thm}

\begin{proof}
We follow the proof of \cite[Theorem 5.24]{GDKD}.

Let $a$ be the basis corresponding to the optimal cocircuit for $\a$. We will define the promised resolution as the total complex of the following multicomplex. 

For any $S\subset a$, let $\a^S\in\{+,-\}^E$ be the sign vector which disagrees with $\a$ on exactly the entries in $S$. For example, $\a^\varnothing=\a$, and $\a^{\{i\}}=\a^i$ for any $i\in a$. Notice that if $i\in S\subset a$ and $\a^S,\a^{S\backslash i}\in\PC$, then there is a degree one map 
\[
\begin{array}{cccc}
\phi_{S,i}\colon & P_{\a^{S}} &\longrightarrow & P_{\a^{S\backslash i}}, \\
 & q & \mapsto & p(\a^{S\backslash i},\a^S)\cdot q.
\end{array}
\]
We extend this to all $S\subset a$ and $i\in a$ by declaring that $P_{\a^S}=0$ if $\a^S\not\in \PC$ and $\phi_{S,i}=0$ if $i\not\in S$. Consider the module
 \[ \Pi_\a:=\bigoplus_{S\subset a} P_{\a^S}, \]
which we view as being graded by the free abelian group $\ZM\{\e_i\mid i\in a\}$ where the summand $P_{\a^S}$ is given degree $\e_S:=\sum_{i\in S} \e_i$. For each $i \in a$, consider the differential $\partial_i\colon \Pi_\a\to\Pi_\a$ of degree $-\e_i$ defined as the sum
 \[ \partial_i:=\sum_{S\subset a}\phi_{S,i}. \]
Observe that $\partial_i\partial_j=\partial_j\partial_i$ for any $i,j\in a$ by relation ($A2$) and so we can view $\Pi_\a$ as a multi-complex with differentials $\partial_i$ for each $i \in a$. 

Let $\mathbf{\Pi}_\a^\bullet$ denote the total complex of $\Pi_\a$.  Then $\mathbf{\Pi}_\a^\bullet$ is a linear complex of projective modules and $H_0(\mathbf{\Pi}_\a^\bullet)=V_\a$. It remains to show that the complex $\mathbf{\Pi}_\a^\bullet$ is exact in positive degrees.

To do so, we will filter the multicomplex $\Pi_\a$.  For each $\b \in \PC$, let $(\Pi_\a)^\b \subset \Pi_\a$ be the submodule whose $\e_S$-graded part is defined as 
\[ \sum_{\g \geq \b, \a^S} P^{\g}_{\a^S} \subset P_{\a^S}, \]
that is, the submodule consisting of all paths from $\a^S$ passing through some $\g\in \PC$ where $\g\geq \a^S$ and $\g\geq \b$.
Observe that the differentials $\partial_i$ for $i\in a$ are compatible with the submodules $(\Pi_\a)^\b$ and so we have defined a filtration of $\Pi_\a$ by the poset $\PC$.

Computing the associated graded of this filtration yields a multi-complex 
\[\wt\Pi_\a = \bigoplus_{\beta \in \PC} (\Pi_\a)^\b/(\Pi_\a)^{>\b} = \bigoplus_{\b \in \PC} \left( \bigoplus_S M^\b_{\a^S}\right),\]
where $M^\b_{\a^S}$ is the subquotient of $P_{\a^S}$ defined as in the proof of Theorem~\ref{thm-projfilt}.  

Consider the resulting quotient multi-complexes for each $\b \in \PC$.  Let $b=\mu^{-1}(\b)$.  Recall from the proof of Theorem~\ref{thm-projfilt} that $M^\b_{\a^S}$ is non-zero if and only if $\a^S \in \BC_b$.

If $\b =\a$, then $M^\b_{\a^S}=M^\a_{\a^S}=0$ for any non-empty $S \subset a=b$.  Thus the only non-zero summand of the $\alpha$-subquotient is $M^\a_\a=V_\a$ in total degree zero.

If $\b \neq \a$, choose an element $i \in a$ such that $i \not\in b$.  Consider those subsets $S \subset a$ such that $i\not\in S$.  Then we have $\a^S \in \BC_b$ if and only if $\a^{S\cup\{i\}} \in \BC_b$.  If $\a^S \in \BC_b$, then $M_{\a^{S\cup i}}^\b\simeq V_\b\simeq M_{\a^S}^\b$ and the differential induced by $\partial_i$ is the isomorphism given by left multiplication with $p(\a^S,\a^{S\cup \{i\}})$.  On the other hand, if $\a^S \not\in \BC_b$, then $M_{\a^{S\cup i}}^\b =0= M_{\a^S}^\b$. In particular, the differential induced by $\partial_i$ on the $\beta$-component of the associated graded multi-complex is exact.

Recall that if any differential of a multi-complex is exact then the total complex of the multi-complex is also exact.  We conclude that the total complex of the associated graded multi-complex is exact in positive degree.  It then follows that the total complex of the original multi-complex must also be exact in positive degrees.
\end{proof}

\begin{thm}\label{thm-Koszul}
Assume $\PB$ is Euclidean.  Then $A$ and $B$ are Koszul algebras and $A$ is Koszul dual to $B$.
\end{thm}

\begin{proof}
By \cite[Theorem 1]{quasi-hered2003} a quasi-hereditary algebra is Koszul if the standard modules have linear projective resolutions. Such resolutions exist for $A$ by the previous theorem. Theorem~\ref{thm-A-B} implies that $B \cong A(\PB^\vee,U^\perp)$ must also be Koszul.  Finally the Koszul duality follows from the quadratic duality statement of Theorem~\ref{thm-quad}.
\end{proof}

\subsection{Numerical identity for Hilbert polynomials}

We do not know whether or not the Euclidean condition on $\PB$ is necessary for $A$ to be Koszul.  In this section we prove that for any generic oriented matroid program $\PB$ the Hilbert polynomial of the algebra $A=A(\PB,U)$ satisfies the following numerical identity.  

Let $H(A,q)$ denote the Hilbert polynomial of $A$, which is the $\PC \times \PC$-matrix with entries
\[ H(A,q)_{\a,\b} = \grdim~e_\a A e_\b. \]
Recall~\cite[Lemma 2.11.1]{BGS} that if $A$ is Koszul, then there is an equality of matrices
\[ H(A,q) H(A^!,-q)^T = I.\]

\begin{thm}\label{thm-numKoszul}
For any generic oriented matroid program $\PB$, the algebra $A=A(\PB,U)$ satisfies the numerical identity above, that is, the Hilbert polynomials of $A$ and its quadratic dual $A^!$ satisfy the matrix equation
\[ H(A,q) H(A^!,-q)^T = I.\]
\end{thm}

\begin{remark} This identity does not necessarily imply that $A$ is Koszul.  See~\cite{Posits} for an example of a non-Koszul quadratic algebra whose Hilbert series satisfies the numerical identity.
\end{remark}

\begin{proof}
Using equation (\ref{eq-grdim}) in the proof of Theorem~\ref{thm-Kgroup}, the $(\a,\b)$-entry of $H(A,q)$ is given by
\[ H(A,q)_{\a,\b} = \grdim~e_\a A e_\b = \sum_{\g \succeq \a,\b}  q^{d_{\a\g}+d_{\g\b}}.\]
In particular, $H(A,q)$ factors as the product $H(A,q) = X X^T$, where $X$ is the $\PC\times\PC$-matrix with $(\a,\b)$-entry given by
\[X_{\a,\b} = \begin{cases}
q^{d_{\a\b}} & \mathrm{if~}\b \succeq \a \\
0 & \mathrm{otherwise}.
\end{cases}
\]

Dually, using Proposition~\ref{prop-compslack} we find that the $(\a,\b)$-entry of $H(A^!,-q)$ is given by
\[ H(A^!,-q)_{\a,\b} = \sum_i (-q)^i \dim e_\a A^!_i e_\b = \sum_{\g \supset Y_{\mu^{-1}(\a)},Y_{\mu^{-1}(\b)}}  (-q)^{d_{\a\g}+d_{\g\b}},\]
in other words the sum runs over all $\g \in \PC$ for which the optimal solution (cocircuit) of both $\a$ and $\b$ are faces of $\g$.  Again this factors as a product $H(A^!,-q) = Y Y^T$, where $Y$ is the $\PC \times \PC$-matrix with $(\a,\b)$-entry given by
\[Y_{\a,\b} = \begin{cases}
(-q)^{d_{\a\b}} & \mathrm{if~}Y_{\mu^{-1}(\b)} \mathrm{~is~a~face~of~}\a \\
0 & \mathrm{otherwise}.
\end{cases}\]

We wish to show that 
\[ H(A,q) H(A^!,-q)^T = XX^T YY^T= I.\]
Note that it suffices to show $X^T Y = I$.

Computing the product $X^T Y$, we find that its $(\a,\b)$-entry is given by
\[ (X^T Y)_{\a,\b} = \sum_{\g\in Q} q^{d_{\a\g}} (-q)^{d_{\g\b}},\] where $Q$ is the set of all $\g \in \PC$ such that $\a \succeq \g$ and $Y_{\mu^{-1}(\b)}$ is a face of $\g$.  In other words, $Q$ consists of all $\g\in \PC$ such that 
\begin{equation} \label{eq-Q}
 \g(i) = \a(i) \mathrm{~if~} i \in \mu^{-1}(\a) \quad \mathrm{~and}  \quad
 \g(i) = \b(i)  \mathrm{~if~} i \not\in \mu^{-1}(\b). \end{equation}

We wish to show that 
\[(X^T Y)_{\a,\b} = \begin{cases}
1 & \mathrm{if}~\a=\b \\
0 & \mathrm{otherwise}.
\end{cases} \]

If $\a=\b$, then $Q=\{\a\}$ and the sum is equal to $q^{d_{\a\a}}(-q)^{d_{\a\a}}=1$.

\medskip

Now assume that $\a \neq \b$ and let
$$J :=\mu^{-1}(\b) \backslash \mu^{-1}(\a) \quad \mathrm{and} \quad J' :=\mu^{-1}(\a) \backslash \mu^{-1}(\b)$$
so that
$$ J \sqcup J' = (\mu^{-1}(\a) \cup \mu^{-1}(\b)) \backslash (\mu^{-1}(\a) \cap \mu^{-1}(\b)).$$
As we have assumed that $\a \neq \b$, $J$ and $J'$ are nonempty.
    
 Note that if $\a(i)\neq \b(i)$ for some $i \in J'$, then by the conditions $(\ref{eq-Q})$ $Q$ is empty  and $(X^T Y)_{\a,\b}=0$ as desired.  Thus we will assume that $\a(i)= \b(i)$ for all $i\in J'$.

Let $$K:=\{ i \in \mu^{-1}(\a)\cap \mu^{-1}(\b) \mid \a(i)\neq \b(i)\}.$$

For $\d \in \PC$, $Y_{\mu^{-1}(\b)}$ is a face of $\d$ if and only if $\d = \b^W$ for some subset $W \subset \mu^{-1}(\b)$.  On the other hand, $\d= \b^W \preceq \a$ if and only if $K \cup J \supset W \supset K$.  Thus $Q = \{(\b^K)^S \mid S \subset J \}$ and
\[ (X^T Y)_{\a,\b} = \sum_{S \subset J} q^{d_{\a,(\b^K)^S}} (-q)^{d_{(\b^K)^S,\b}} = (-1)^{|K|} \sum_{S \subset J} (-1)^{|S|} q^{d_{\a,(\b^K)^S} + d_{(\b^K)^S,\b}}\] 
\[ = (-1)^{|K|} \sum_{S \subset J} (-1)^{|S|} q^{d_{\a,\b} + 2|S^{(\b^K)^S}_{\a,\b}|}= (-1)^{|K|} q^{d_{\a,\b}} \sum_{S \subset J} (-1)^{|S|} q^{2|S^{\b^S}_{\a,\b}|},\] 
where in the last line we have used $S^{(\b^K)^S}_{\a,\b} = S^{\b^S}_{\a,\b} = \{i\in S \mid \a(i) = \b(i)\}$. We will need the following lemma to finish this proof.

\begin{lem}
Assume as above that $\a\neq\b$ and $\a(i)= \b(i)$ for all $i\in J'$.  Then there exists an element $t\in J$ such that $\a(t)\neq \b(t)$.
\end{lem}
\begin{proof}
Suppose for the sake of contradiction that $\a(i)=\b(i)$ for all $i \in J$, then $\a(i)= \b(i)$ for all $i\in J \sqcup J'$.   In the deletion-contraction program 
$$\left(\PB/(\mu^{-1}(\a)\cap \mu^{-1}(\b))\right)\setminus (\mu^{-1}(\a) \cup \mu^{-1}(\b))^c$$
defined on the set $J \sqcup J'$,
 the restrictions of the sign vectors of $\a$ and $\b$ are then equal and so describe the same tope $T$.  Now $Y_{\mu^{-1}(\a)}$ is the optimal solution for $\a$ and $Y_{\mu^{-1}(\b)}$ is the optimal solution for $\b$, so the restrictions $Y_\a$ and $Y_\b$ of $Y_{\mu^{-1}(\a)}$ and $Y_{\mu^{-1}(\b)}$ to $J \sqcup J'$ should both be the unique optimal solution of the tope $T$.  But $z(Y_\a) = J' \neq J = z(Y_\b)$, which is a contradiction.  Thus there exists a $t \in J$ such that $\a(t)\neq \b(t)$ as desired.
\end{proof}

In particular if $S \subset J\backslash\{t\}$, we have $S^{\b^{S\sqcup\{t\}}}_{\a,\b} = S^{\b^{S}}_{\a,\b}$. 

Using this fact we rewrite the sum above:
\[ (X^T Y)_{\a,\b} = (-1)^{|K|} q^{d_{\a,\b}} \sum_{S \subset J} (-1)^{|S|} q^{ 2|S^{\b^S}_{\a,\b}|},\]
\[ = (-1)^{|K|} q^{d_{\a,\b}} \sum_{S \subset J \backslash \{t\} } \left((-1)^{|S|} q^{ 2|S^{\b^S}_{\a,\b}|} + (-1)^{|S\sqcup\{t\}|}  q^{2|S^{\b^{S\cup\{t\}}}_{\a,\b}|}\right)\]
\[ = (-1)^{|K|} q^{d_{\a,\b}} \sum_{S \subset J \backslash \{t\} } \left((-1)^{|S|} q^{2|S^{\b^S}_{\a,\b}|} - (-1)^{|S|} q^{2|S^{\b^{S}}_{\a,\b}|}\right)=0.\] 
\end{proof}

\subsection{Self-dual projectives}
Consider the duality functor 
\[\textrm{d}:A\text{-mod}\to A\text{-mod}\] 
defined by composing the equivalence $A^{\text{op}}\text{-mod} \simeq A\text{-mod}$ induced by the isomorphism $A\cong A^{\text{op}}$ given by reversing the arrows of the quiver $D_E$ in Section \ref{section-quiver} with the induced functor $A^{\text{op}}$-mod$\to A$-mod coming from vector space duality.

In the following result, for a fixed sign vector $\a\in \PC= \PC^\vee$ we will need to refer to both the corresponding bounded feasible tope in the affine space of $\PB$ and the corresponding bounded feasible tope in the affine space of $\PB^\vee$.  To distinguish these two topes, we write $T_\a$ to denote the tope in $\PB$ and $T_\a^\vee$ for the tope in $\PB^\vee$.

\begin{thm}\label{thm-selfdual}
For any generic oriented matroid program $\PB$ and $\a\in\PC$. The following are equivalent:
\begin{enumerate}
 \item The projective $P_\a$ is injective.
 \item The projective $P_\a$ is self-dual.
 \item The simple $L_\a$ is contained in the socle of some standard module $V_\b$.
 \item The bounded feasible tope $T_\a$ covers an infeasible subtope $X$, meaning $X(g)=0$.
 \item The bounded feasible tope $T_\a^\vee$ in the dual program $\PB^\vee$ is in the core of the affine space for $\PB^\vee$.  In other words $(T_\a^\vee)^\infty=0$ or equivalently the cocircuit faces of the tope $T_\a^\vee$ are all feasible.
\end{enumerate}
When $\PB$ is Euclidean, and so $A$ is quasi-hereditary by Theorem~\ref{thm-projfilt}, then the statements above are also equivalent to the following: 
\begin{enumerate}
\item[(6)] The projective $P_\a$ is tilting.
\end{enumerate}
 \end{thm}

\begin{proof}
The implications $(1)\iff(2)~ (~\iff(6)$, if $A$ is quasi-hereditary) are standard facts.

$(2) \implies (3):$  If $P_\a$ is self-dual, then the socle of $P_\a$ is isomorphic to the cosocle of $P_\a$, which is $L_\a$.  Therefore when expressing $[P_\a]$ as a sum of simple classes times powers of $q$ in the Grothendieck group, $L_\a$ is the only simple to appear in the top degree.  On the other hand, by Theorem~\ref{thm-Kgroup}, 
$$[P_\a] = \sum_{\b \in C_\a} q^{d_{\a\b}} [V_\b], $$
so the unique simple class appearing in the highest degree must also appear as the highest degree term of some $[V_\b]$.  We conclude that $L_\a$ is the socle of $V_\b$.

$(3)\implies(4):$ Let $b=\mu^{-1}(\b)$. Lemma ~\ref{lem-standard} says that $V_\b$ is spanned as a vector space by taut paths $p_\g$ from $\b$ to $\g \preceq \b$.  A taut path $p_\a$ is in the socle of $V_\b$ if there does not exist a longer taut path $p_\g$ that factors through $p_\a$.  Note that this is equivalent to the condition: if $i\not\in b$ and $\a$ has a feasible face $Y$ such that $Y(i)=0$, then $\a(i)\neq\b(i)$. 
 
Recall that $Y_b \in \AC$ denotes the feasible cocircuit of $\NC = \wt\MC\backslash f$ that is the optimal solution of the tope $T_\b$.  By the covector axioms of an oriented matroid, the composition $T := (-Y_b)\circ T_\a$ is also a covector of $\NC$ and in particular an infeasible tope such that $T(i)=\a(i)$ for all $i\in b$ and for all $i$ which are zero on a feasible face of $\a$. 
 A taut path $p$ from $T_\a$ to $T$ exists in the tope graph of $\wt\MC\backslash f$, and this path cannot change the sign corresponding to any feasible facet of $T_\a$. Thus the first sign change of the path $p$ must be infeasible, which means that $T_\a$ covers a subtope $X$ with $X(g)=0$. 

$(4)\implies(5)$: 
If the bounded tope $T_\a$ in $\wt\MC\backslash f$ covers a subtope $X$ with $X(g)=0$ then $\a$ is a bounded feasible sign vector for both the original program $\PB$ as well as the reoriented program ${}_{-g}\PB = ({}_{-g}\wt\MC,-g,f)$. Dually, this means that the tope $T_\a^\vee$ in $\wt\MC^\vee\backslash g=({}_{-g} \wt\MC^\vee)\backslash (-g)$ is bounded and feasible in both dual programs $\PB^\vee$ and ${}_{-g}\PB^\vee=({}_{-g}\wt\MC^\vee,f,-g)$. In particular, the tope $\a\in\PC^\vee$ does not have any cocircuit face $Y$ with $Y(f)=0$, since this would imply $T_\a^\vee$ was unbounded in one of these generic programs.

$(5)\implies(2)$: If $(T_\a^\vee)^\infty=0$, then $e_\a A e_\a\simeq R_\a^\vee=k[z(\D_\a^\vee)]/(U^\perp)$, where $\D_\a^\vee = \fllv(T_\a^\vee) \backslash \{0\}$. By Lemma~\ref{lem-feas Las Vergnas}, $\|\D_\a^\vee\|=\|\fllv(T_\a^\vee) \backslash \{0\} \|$ is homeomorphic to a $(n-d-1)$-sphere and so a result of Munkres (see \cite[Theorem II.4.3]{StanleyComb}) implies that $R_\a^\vee$ is Gorenstein, meaning that there is an isomorphism
  $\int : (e_\a A e_\a)_{n-d-1}\to k$
such that $\langle x,y\rangle=\int xy$ defines a perfect pairing on $e_\a A e_\a$. 

We wish to produce an isomorphism of $A$-modules $\textrm{d}(P_\a) \cong P_\a$.  To do so, we will show that the map
\[
\langle -,-\rangle: e_\a A\times Ae_\a \to k 
\]
\[
\qquad \qquad \qquad (p,q) \mapsto \int pq
\]
defines a perfect pairing and so it will follow that $\textrm{d}(P_\a)=(e_\a A)^*\cong Ae_\a= P_\a$ as right $A$-modules.

To prove that $\langle -,-\rangle$ is perfect, we first observe that it suffices to show that the map $\cdot p_{\b,\a}:e_\a A e_\b\to e_\a A e_\a$ is injective for some taut path $p_{\b\a}$ from $\b$ to $\a$.  This is because for any nonzero $x\in e_\b Ae_\a$, if $x\cdot p_{\b\a}\neq 0$, then there exists $y\in e_\a A e_\a$ such that $\int (xp)y=\int x(py)=\langle x,py\rangle \neq 0$. On the $B$ side, this means showing that 
  \eq{ \cdot u_{S_{\a\a}^\b}: R_{\a\b}^\vee\to R_\a^\vee }
is injective. We proceed by showing 
 \eq{ \cdot u_{S_{\a\a}^\b}:\wt R_{\a\b}^\vee\to\wt R_\a^\vee}
 is the injective map in a split short exact sequence of $\Sym U^\perp$-modules, which proves the claim by applying the functor $-\otimes_{\Sym U^\perp} k$. 

 The claim is obvious if $\a=\b$ or $T^\vee_\a\wedge T^\vee_\b$ is not feasible, so we assume $\a\neq\b$ and $T^\vee_\a\wedge T^\vee_\b$ is a proper non-empty face of $T_\a^\vee$. 
To see the monomial map $\cdot u_{S_{\a\a}^\b}:\wt R_{\a\b}^\vee\to\wt R_\a^\vee$ is injective and to determine its cokernel, consider the image of any non-zero monomial $m=\prod_{i\in S} u_i^{s_i}$ in $\wt R^\vee_{\a\b}$, where $s_i >0$ for any $i\in S$.  As $m$ is non-zero in $\wt R^\vee_{\a\b}$, there exists $Y \in \D_{\a\b}^\vee$ such that $S \subset z(Y)$.  Note that $S^\b_{\a\a} = z(T^\vee_\a\wedge T^\vee_\b) \subset z(X)$ for any $X \in \D^\vee_{\a\b}$.  Thus $S \cup S^\b_{\a\a} \subset z(Y)$ and the product $m\cdot u_{S_{\a\a}^\b}=\prod_{i\in S\cup S^\b_{\a\a}} u_i^{t_i}$, where $t_i>0$ for $i\in S\cup S^\b_{\a\a}$,  is non-zero in $\wt R^\vee_\a$.

The computation above also shows that the cokernel of the map $\cdot u_{S_{\a\a}^\b}:\wt R^\vee_{\a\b}\to\wt R_\a^\vee$ is the face ring $k[\D]$ of 
 \[ \D=  \{S\subset E\mid S\subset z(Y)\;\text{for some}\:Y\in\D_{\a}\;\text{and}\:S^\b_{\a\a} \not\subset S\}. \] 
\[=z(\D_\a)\backslash \{S \in z(\D_\a) \mid S^\b_{\a\a} \subset S\}. \]
Recall from Lemma~\ref{lem-feas Las Vergnas} that the geometric realization of the simplicial complex $z(\D_\a)$ is a PL $(d-1)$-sphere.  The subset of simplices $\{S \in z(\D_\a) \mid S^\b_{\a\a} \subset S\}$ is the open star of the simplex on the set $S^\b_{\a\a}$ and thus its complement $\D$ in $z(\D_\a)$ is a PL $(d-1)$-ball.  It follows that $k[\D]$ is Cohen-Macaulay, again with parameter space $U^\perp$. Thus, $k[\D]$ is a free $\Sym U^\perp$-module, and therefore the exact sequence 
 \eq{\wt R_{\a\b}^\vee\hookrightarrow\wt R_\a^\vee \fib k[\D]}
 splits.
\end{proof}

\section{Derived Morita equivalence}\label{sec-derived}

We conclude with a proof of Theorem~\ref{thm-derivedequiv}.  Recall that $\MC$ is an oriented matroid, $U$ a parameter space for $M=\un\MC$, and $\PB_1=(\wt\MC_1,g_1,f_1), \PB_2=(\wt\MC_2,g_2,f_2)$ and $\PB_{\text{mid}}=(\wt\MC_{\text{mid}},g_2,f_1)$ are Euclidean generic oriented matroid programs such that 
\[ \MC= (\wt\MC_1/g_1)\backslash f_1 = (\wt\MC_2/g_2) \backslash f_2\]
and
\[\wt\MC_{\text{mid}}/g_2=\wt\MC_1/g_1, \qquad \wt\MC_{\text{mid}}\backslash f_1=\wt\MC_2\backslash f_2.\]
We wish to show there is an equivalence of categories
\[D(A(\PB_1,U))\cong D(A(\PB_2,U)),\]
where $D(A)$ denotes the bounded derived category of graded finitely generated $A$-modules.

\begin{remark}
Note that if $\PB_1$ and $\PB_2$ are Euclidean, it is not automatic that $\PB_{\text{mid}}$ will be Euclidean as well.  For example, one could take $\text{EFM}(8)$ (see Example~\ref{ex-EFM(8)}) and then change the choice of $g$ and $f$ separately to obtain two realizable (and hence Euclidean) generic oriented matroid programs $\PB_1$ and $\PB_2$ such that $\PB_{\text{mid}}$ is the non-Euclidean program $\text{EFM}(8)$.
\end{remark}

We will prove Theorem~\ref{thm-derivedequiv} by reducing it to the following claim.

\begin{prop}\label{prop-equiv}
Suppose $\PB_1=(\wt\MC_1,g_1,f)$ and $\PB_2=(\wt\MC_2,g_2,f)$ are generic Euclidean programs extending $\MC$ such that $\MC_1/g_1=\MC_2/g_2$. Then there is an equivalence of categories $$D(A(\PB_1,U))\simeq D(A(\PB_2,U)).$$
\end{prop}

Before giving a proof of this Proposition, we will use it to deduce Theorem~\ref{thm-derivedequiv}.

\begin{proof}[Proof that Proposition~\ref{prop-equiv} implies Theorem~\ref{thm-derivedequiv}]
Under the assumptions of  Theorem~\ref{thm-derivedequiv}, $\wt\MC_{\text{mid}}/g_2=\wt\MC_1/g_1$. Then by Proposition~\ref{prop-equiv} it follows that 
\begin{equation}\label{eq-equiv1}
D(A(\PB_1,U)) \simeq D(A(\PB_{\text{mid}},U).
\end{equation}
On the other side, duality together with the assumptions of  Theorem~\ref{thm-derivedequiv} give:  $$\wt\MC_{\text{mid}}^\vee/ f_1 = (\wt\MC_{\text{mid}}\backslash f_1)^\vee = (\wt\MC_2\backslash f_2)^\vee = \wt\MC_2^\vee/ f_2.$$
Viewing $f_1$ and $f_2$ as playing the role of $g$ in the Euclidean programs $\wt\MC_{\text{mid}}^\vee$ and $\wt\MC_2^\vee$ respectively, we can again apply Proposition~\ref{prop-equiv} to find: 
\begin{equation}\label{eq-equiv2}
D(A(\PB_2^\vee,U^\perp)) \simeq D(A(\PB_{\text{mid}}^\vee,U^\perp)).
\end{equation}
Putting these equivalences~(\ref{eq-equiv1}) and (\ref{eq-equiv2}) together with the equivalences from Koszul duality:
$$D(A(\PB_{\text{mid}},U)) \simeq D(A(\PB_{\text{mid}}^\vee,U^\perp)) \quad \text{and} \quad D(A(\PB_2,U)) \simeq D(A(\PB_2^\vee,U^\perp)) ,$$
gives the desired result:
$$D(A(\PB_1,U))\simeq D(A(\PB_2,U)).$$
\end{proof}

\subsection{The definition and properties of the functor}\label{subsec-functor}
It remains to prove Proposition \ref{prop-equiv}.  For the remainder of the paper we will let
 \[\PB_1=(\wt\MC_1,g_1,f) \quad \text{and} \quad \PB_2=(\wt\MC_2,g_2,f)\]
be two Euclidean generic oriented matroid programs such that $\wt\MC_1/g_1=\wt\MC_2/g_2$.

For $\ell=1,2$ let $A_\ell = A(\PB_\ell,U)$, $B_\ell = B(\PB_\ell,U)$ and $\PC_\ell$ be the set of bounded, feasible sign vectors of $\PB_\ell$.  Note that the set of bounded sign vectors $\BC$ is the same for $\PB_1$ and $\PB_2$.

As in \cite[Section 6]{GDKD}, the desired equivalence will come from a derived tensor product with a certain bimodule $N$.

It is slightly easier to define the bimodule $N$ on the $B$-side, using the isomorphisms $A_\ell\simeq B_\ell^\vee$ of Theorem \ref{thm-A-B} for $\ell=1,2$. Namely, let
 \[N=\bigoplus_{(\a,\b)\in\PC_1\times\PC_2} R_{\a\b}^\vee[-d_{\a\b}]\]
with the natural left $B_1^\vee$-action and right $B_2^\vee$-action given by the $\star$ operation. 

To translate this to the $A$-side, recall the following the alternative definition of $A$ from Section \ref{section-quiver}, 
$$ A(\PB,U) = e_{\PC} D_E e_{\PC}/\langle e_f e_{\PC} \rangle + \langle \vartheta(U^\perp)e_{\PC} \rangle. $$
We consider an extended version of $A$ that only depends on $\wt\MC/g$ by replacing $e_{\PC}$ by $e_\BC=\sum_{\alpha\in\BC}e_\alpha$. That is, let 
$$ A_\ext(\PB,U) = e_\BC D_E e_\BC/\langle e_f e_\BC \rangle + \langle \vartheta(U^\perp)e_\BC \rangle. $$

As $A_\ext(\PB,U)$ only depends on $\wt\MC/g$ and we have assumed that $\wt\MC_1/g_1=\wt\MC_2/g_2$, let
  \[A_\ext:=A_\ext(\PB_1,U)=A_\ext(\PB_2,U).\]

When viewed as an $(A_1,A_2)$-bimodule, $N$ can be described as
 \[ N=e_{g_1} A_\ext e_{g_2},\]
where $e_{g_\ell}=\sum_{\a\in\PC_\ell}e_\a$ for $\ell=1,2$.

To check that these definitions of $N$ coincide, consider the graded vector space
  \[B_\ext(\PB,U)=\bigoplus_{(\a,\b)\in\FC\times\FC} R_{\a\b}[-d_{\a\b}], \]
made into an algebra via $\star$, as in the definition of $B(\PB,U)$ from Section \ref{section-B}. Then an easy extension of the proof of Theorem \ref{thm-A-B} gives us the following lemma.

\begin{lem}
There is an isomorphism $A_\ext(\PB,U)\simeq B_\ext(\PB^\vee,U^\perp)$. Combining this isomorphism with the isomorphisms $A_\ell\simeq B_\ell^\vee$, we obtain an equivalence between the two definitions of $N$.
\end{lem}

We define the functor $\Phi:D(A_1)\to D(A_2)$ via
  \eq{\Phi(M)=M\overset{L}{\otimes}_{A_1} N.}
For $\ell=1,2$ and any $\a\in\PC_\ell$, let $P_\a^\ell$ and $V_\a^\ell$ be the corresponding projective and standard $A_\ell$-modules. Define $\nu:\PC_1\to\PC_2$ to be the composition 
  \[ \PC_1\overset{\mu_1^{-1}}{\longrightarrow}\BM\overset{\mu_2}{\longrightarrow}\PC_2. \]

\begin{prop}\label{prop-alpha in intersection}
  If $\a\in\PC_1\cap\PC_2$, then $\Phi(P_\a^1)=P_\a^2$.
\end{prop}

\begin{proof}
Consider the natural map
  \[\G:P_\a^2=e_\a A_2\to e_\a A_1 \otimes_{A_1} e_{g_1} A_\ext e_{g_2}=P_\a^1\otimes_{A_1} N=\Phi(P_\a^1)\]
taking $e_\a$ to $e_\a\otimes e_{g_1}e_{g_2}$. For paths $p,q$ in the quiver $Q$ with $p$ only passing through nodes in $\PC_1$, the equality of the simple tensors
  \[e_\a p\otimes e_{g_1} q e_{g_2}=e_\a\otimes e_\a p e_{g_1} q e_{g_2}
  =e_\a\otimes e_{g_1}e_{g_2}\cdot e_\a p e_{g_1} q e_{g_2}\]
implies that $\G$ is an isomorphism.
\end{proof}

\begin{remark}
Note that the proposition above and its proof are valid without the assumption that $\PB_1$ and $\PB_2$ be Euclidean.
\end{remark}

\begin{lem}\label{lem-projfiltfunctor}
  For any $\a\in\PC_1$, the $A_2$-module $\Phi(P_\a^1)$ has a filtration with standard subquotients. For $b \in \BM$, if $\a\in\BC_b$ then the standard module $V^2_{\mu_2(b)}$ appears with multiplicity 1 in the associated graded, and otherwise it does not appear.\footnote{Recall that the set $\BC_b$ defined in Definition~\ref{def-bounded cone} only depends on $\wt\MC/g$.}
\end{lem}

\begin{proof}
We have 
  \[\Phi(P_\a^1)=P_\a^1\otimes_{A_1} N=e_\a A_1\otimes_{A_1} e_{g_1} A_\ext e_{g_2},\]
so elements of $\Phi(P_\a^1)$ can be represented as linear combinations of paths in $\BC$ which begin at $\a$ and end at elements of $\PC_2=\BC\cap\FC_2$. For $\b\in\PC_2$, let $\Phi(P_\a^1)_\b$ be the submodule generated by paths $p$ such that $\b$ is the maximal element of $\PC_2$ (with respect to the ordering $\leq_2$ on $\PC_2$ coming from our Euclidean assumption on $\PB_2$) through which $p$ passes. Then let
  \[ \Phi(P_\a^1)_{>_2~ \b}=\bigcup_{\g>_2~ \b}\Phi(P_\a^1)_\g \quad\text{and}\quad  \Phi(P_\a^1)_{\geq_2~\b}=\bigcup_{\g\geq_2~\b}\Phi(P_\a^1)_\g.\]
We then obtain a filtration 
  \[\Phi(P_\a^1)=\bigcup_{\b\in\PC_2}\Phi(P_\a^1)_{\geq_2~ \b}.\]
Suppose $b\in \BM$ and let $\b=\mu_2(b)$. It suffices to show that the quotient $\Phi(P_\a^1)_{\geq_2~\b}/\Phi(P_\a^1)_{>_2~\b}$ is isomorphic to $V^2_\b$ if $\a\in\BC_b$, and is zero otherwise.

Our argument follows the proof of Theorem \ref{thm-projfilt}.

If $\a\not\in\BC_b$, then for some $i \in b$, $\a(i)\neq \b(i)$. Thus any path from $\a$ to $\b$ can be represented by one passing through $\b^i$.  By Lemma~\ref{lem-icross}, $\b^i >_2~\b$ and therefore $$\Phi(P_\a^1)_{\geq_2~\b}/\Phi(P_\a^1)_{>_2~\b}\cong 0.$$ 

Otherwise, precomposition with a taut path from $\a$ to $\b$ gives a surjective map
  \[V_\b^2\fib \Phi(P_\a^1)_{\geq_2~\b}/\Phi(P_\a^1)_{>_2~\b}.\]
Thus
\[\dim V_\b^2 \geq \dim~ \Phi(P_\a^1)_{\geq_2~\b}/\Phi(P_\a^1)_{>_2~\b}\]
and it suffices to show equality holds. After summing over all $\b\in \PC_2$:

\begin{equation*}
\begin{split}
\dim(\Phi(P^1_\a)) &= \sum_{\b \in \PC_2} \dim ~~\Phi(P_\a^1)_{\geq_2~\b}/\Phi(P_\a^1)_{>_2~\b} \\
& \leq \sum_{\a \in \BC_{b}}  \dim V^2_{\mu_2(b)} = \#\{(\g,b)\in \PC_2\times \BM\mid \a,\g\in\BC_b\} .
\end{split}
\end{equation*}

It suffices to show that equality holds after summing over all $\a \in \PC_1$.
As $N=\oplus_{\a\in\PC_1} \Phi(P_\a^1)$, we have
\[\sum_{\a\in\PC_1} \dim \Phi(P_\a^1) =\dim N = \sum_{(\a,\g)\in\PC_1\times\PC_2} \dim R_{\a\g}^\vee \]
\[ = \#\{(\a,\g,b)\in\PC_1\times\PC_2\times \BM\mid \a,\g\in\BC_b\}.\]
Here we are using Lemma \ref{lem-size} and Proposition \ref{prop-compslack} on each $(\a,\g)\in\PC_1\times\PC_2$.
\end{proof}

\begin{remark}
Notice that Theorem \ref{thm-projfilt} can be viewed as the special case $\PB_1=\PB_2$. 
\end{remark}

\begin{remark}\label{rem-Kgroup} We note that the above proof does not use the assumption that $\PB_1$ is Euclidean, and so for this result we need only assume that $\PB_2$ is Euclidean.  More generally, when $\PB_2$ is not Euclidean one can prove the result on the level of Grothendieck groups with a nearly identical proof as was given for the analogous Theorem~\ref{thm-Kgroup}.
\end{remark}

\begin{prop}\label{prop-groth}
For all $\a\in\PC_1$, we have $[\Phi(V_\a^1)]=[V_{\nu(\a)}^2]$ in the Grothendieck group of (ungraded) right $A_2$-modules. Thus $\Phi$ induces an isomorphism of Grothendieck groups.
\end{prop}

\begin{proof}
For any $\a\in\PC_1$, the equalities
  \[\sum [V_{\mu_2(b)}^2]
  =[\Phi(P_\a^1)]
  =\sum[\Phi(V_{\mu_1(b)}^1)]\]
follow from Lemma \ref{lem-projfiltfunctor} and its special case Theorem \ref{thm-projfilt}, where both sums are taken over 
  \[\{b\in\BM\mid\a\in\BC_b\}.\]
The first claim then follows by induction on the poset $\PC_1$ with base case $\a\in\PC_1$ maximal, so 
  \[P_\a^1=V_\a^1 \quad \text{and} \quad \Phi(P_\a^1)=V^2_{\nu(\a)}.\]
The second statement follows from the fact that the classes of standard modules in a highest weight category form a $\ZM$-basis for the Grothendieck group.
\end{proof}

\begin{remark}
One can show that this result holds without the Euclidean condition on $\PB_1$ and $\PB_2$ by the second part of Remark~\ref{rem-Kgroup} and the fact that the matrix $X$ from the proof of Theorem~\ref{thm-numKoszul} is invertible.
\end{remark}

\begin{prop}
Let $\a\in\PC_1$. Then $\Phi(V_\a^1)$ is the quotient of $\Phi(P_\a^1)$ by the submodule generated by all paths changing their $i$-th coordinate for some $i\in\mu_1^{-1}(\a)$. In particular, $\Tor_m^{A_1}(V_\a^1,N)=0$ for all $m>0$.
\end{prop}

\begin{proof}
Apply $\Phi$ to the linear projective resolution of $V_\a^1$ of Theorem \ref{thm-KoszulStan}. The degree zero homology of the resulting complex is the quotient promised. We wish to show that the resulting complex is a resolution of $V_\a^1\otimes_{A_1}N$. This claim follows from argument is analogous to the proof of Theorem \ref{thm-KoszulStan}, where for each $S\subset \mu_1^{-1}(\a)$ we filter each $A_2$-module $\Phi(P_{\a^S}^1)=P_\a^1\otimes_{A_1}N$ by standards as in the proof of Lemma \ref{lem-projfiltfunctor}.
\end{proof}

\begin{cor}\label{cor-torstd}
If a right $A_1$-module $M$ admits a filtration by standard modules, we have $\Tor_m^{A_1}(M,N)=0$ and therefore
\[\Phi(M)=M\otimes_{A_1}N.\]
\end{cor}

\subsection{Ringel duality and composition of functors}
Suppose that $\PB=(\wt\MC,g,f)$ is a generic Euclidean extension of $\MC$. Consider the program $\overline{\PB}= {}_{-g}\PB = ({}_{-g}\wt\MC,-g,f)$ obtained from $\PB$ by reorientation of $g$.  In other words $\overline{\PB}$ is the program on the oriented matroid $\MC$ with feasible cocircuits equal to the negative of the feasible cocircuits of $\PB$. This program is also generic and Euclidean. We let $\overline{A}=A(\overline{\PB},U)$ and denote by
  \[\Phi^-:D(A)\to D(\overline{A}),\]
 the functor $\Phi$ for $\PB_1=\PB$ and $\PB_2=\overline\PB$. We will prove that $\Phi^-$ is an equivalence and relate it to Ringel duality. 

\begin{thm}\label{thm-ringel}
$\Phi^-$ is an equivalence, the algebras $A$ and $\overline{A}$ are Ringel dual, and the Ringel duality functor is $\text{d}\circ\Phi^-=\Phi^-\circ\text{d}$.
\end{thm}

\begin{proof}
Notice that for any $\a\in\PC$, we have that the $B$-side description of $\Phi^-$ gives   
  \[\Phi^-(P_\a)=\bigoplus_{\overline\b\in\overline\PC}R_{\a\overline\b}^\vee[-d_{\a\overline\b}].\]
Then we have that the tope corresponding to the restrictions of both $\a$ and $\overline\b$ in the oriented matroid $\wt\MC^\vee/S_{\a\a}^{\overline\b}$ on $E\cup\{f\}$ has all cocircuit faces taking the value $+$ on $f$. As in the proof of Theorem \ref{thm-selfdual}, this implies that $R_{\a\overline\b}^\vee$ is Gorenstein and $\Phi^-(P_\a)$ is self-dual.  By Lemma \ref{lem-projfiltfunctor} it follows that $\Phi^-(P_\a)$ is tilting.  

It remains to show that $\Phi^-$ is an equivalence.  With Proposition \ref{prop-alpha in intersection} and Theorem \ref{thm-selfdual} in hand, one can repeat the proof of \cite[Theorem 6.10]{GDKD} word for word.
\end{proof}

To complete the proof of Proposition~\ref{prop-equiv} in the general case, we will need to study the composition of functors.

Let $\PB_1=(\wt\MC_1,g_1,f)$, $\PB_2=(\wt\MC_2,g_2,f)$, and $\PB_3=(\wt\MC_3,g_3,f)$ be generic Euclidean programs extending $\MC$ for which $\wt\MC_1/g_1=\wt\MC_2/g_2=\wt\MC_3/g_3$. We can then define the three functors 
  \[D(A_1)\overset{\Phi_{12}}{\longrightarrow} D(A_2) \overset{\Phi_{23}}{\longrightarrow} D(A_3) \quad \text{and} \quad D(A_1)\overset{\Phi_{13}}{\longrightarrow} D(A_3)\]
as before. We would like to compare $\Phi_{13}$ with the composition $\Phi_{23}\circ\Phi_{12}$. 

Notice that 
  \[N_{12}=\Phi_{12}(A_1)=\bigoplus_{\a\in\PC_1}\Phi_{12}(P_\a^1)\]
has a filtration by standard modules as a right $A_2$-module by Lemma \ref{lem-projfiltfunctor}. Then 
  \[\Phi_{23}\circ\Phi_{12}(M)
  =(M\overset{L}{\otimes}_{A_1} N_{12})\overset{L}{\otimes}_{A_2} N_{23}
  =M\overset{L}{\otimes}_{A_1}(N_{12}\otimes_{A_2} N_{23})\]
by Corollary \ref{cor-torstd}. The natural map $N_{12}\otimes_{A_2} N_{23}\to N_{13}$ given by concatenation of paths induces a natural transformation $\Phi_{23}\circ\Phi_{12}\to\Phi_{13}$. We also have that $\Phi_{23}\circ\Phi_{12}$ and $\Phi_{13}$ induce the same map on Grothendieck groups by Proposition \ref{prop-groth}. 
This implies that 
  \[\dim_k N_{12}\otimes_{A_2}N_{23}=\dim_k N_{13}\]
since the classes of $N_{12}\otimes_{A_2}N_{23}$ and $N_{13}$ agree in the Grothendieck group of $A_3$-modules.

We now combine this discussion with the equivalence we have already proved. Let $\PB_3=\overline{\PB_1}$, so that 
  \[\Phi_{13}=\Phi^-: D(A_1)\to D(\overline{A_1}) \quad \text{and} \quad \Phi_{23}=\Phi_{2\overline{1}}:D(A_2)\to D(\overline{A_1}).\]

\begin{lem}\label{lem-factor}
 $\Phi^-\cong\Phi_{2\overline{1}}\circ\Phi_{12}$.
\end{lem}

\begin{proof}
The conclusion follows from the discussion above if we can show that the map $N_{12}\otimes_{A_2}N_{2\overline1}\to N_{1\overline1}$ is an isomorphism. We have already observed that the source and target have the same dimension, so it will suffice to show that this map is surjective.  This means showing that for any $(\a,\b)\in\PC_1\times\overline{\PC_1}$, every path from $\a$ to $\b$ in $e_\a A_\ext e_\b$ can be represented as a path that passes through some sign vector $\g$ in $\PC_2$.  It suffices to do this for a taut path from $\a$ to $\b$.  Translating this to the $B$-side, we wish to show that $1_{\a\b}^\vee=1_{\a\g}^\vee\star 1_{\g\b}^\vee$ for some $\g \in \PC_2$.

Let $(\a,\b)\in\PC_1\times\overline{\PC_1}$ and suppose that $R_{\a\b}^\vee$ is nonzero. This means the maximal common covector face $T^\vee_\a\wedge T^\vee_\b$ of the topes $T^\vee_\a$ and $T^\vee_\b$ in 
  \[\wt\MC^\vee:=
  \wt\MC_1^\vee\backslash g_1
  =(\wt\MC_1/g_1)^\vee
  =(\wt\MC_{\overline1}/g_{\overline1})^\vee 
  =\wt\MC_{\overline1}^\vee\backslash g_{\overline1}\]
is nonzero and all of its nonzero cocircuit faces are feasible, i.e. they take the value $+$ on $f$.  Together with the fact that $z(T^\vee_\a\wedge T^\vee_\b)=S^\b_{\a\a}$, this implies that $T^\vee_\a\wedge T^\vee_\b$ restricts to a bounded feasible tope in the oriented matroid program $$\PB_2^\vee/S_{\a\a}^\b= (\wt\MC^\vee/S_{\a\a}^\b,f,g_2).$$

Let $Y$ be the optimal cocircuit face of $T^\vee_\a\wedge T^\vee_\b$ viewed as a covector in $\PB_2^\vee/S_{\a\a}^\b.$
  Then $Y$ lifts to a unique feasible cocircuit of $\PB^\vee_2$ and let $\g \in \PC_2^\vee = \PC_2$ be the corresponding sign vector.  By construction, $T^\vee_\a\wedge T^\vee_\b$ is a face of $T^\vee_\g$.

Thus $R^\vee_{\a\g\b}=R^\vee_{\a\b}, S^\g_{\a\b}=\varnothing$ and
\[1_{\a\b}^\vee=1_{\a\g}^\vee\star 1_{\g\b}^\vee.\]
\end{proof}

\begin{proof}[Proof of Proposition \ref{prop-equiv}]
We can set up everything as in Lemma \ref{lem-factor}, and we know that $\Phi^-$ is an equivalence by Theorem \ref{thm-ringel}. This gives us that $\Phi_{12}:D(A_1)\to D(A_2)$ is faithful while $\Phi_{2\overline{1}}$ is full and essentially surjective. 

Note that $\overline{\PB_2}$ is Euclidean if $\PB_2$ is Euclidean. Appealing to the same arguments as before, $\Phi_{12}\circ\Phi_{\overline{2}1}$ is an equivalence.  We conclude that $\Phi_{12}$ is also full and essentially surjective and thus an equivalence.
\end{proof}

\bibliographystyle{myalpha}
\bibliography{OM}{}

\newcommand{\etalchar}[1]{$^{#1}$}
\begin{thebibliography}{BLVS{\etalchar{+}}99}

\bibitem[{\'A}DL03]{quasi-hered2003}
I.~{\'A}goston, V.~Dlab, and E.~Luk{\'a}cs.
\newblock Quasi-hereditary extension algebras.
\newblock {\em Algebras and Representation Theory}, 6(1):97--117, Mar 2003.

\bibitem[BGS96]{BGS}
A.~Beilinson, V.~Ginzburg, and W.~Soergel.
\newblock Koszul duality patterns in representation theory.
\newblock {\em J. Amer. Math. Soc.}, 9(2):473--527, 1996.

\bibitem[BLPW10]{GDKD}
T.~Braden, A.~Licata, N.~Proudfoot, and B.~Webster.
\newblock Gale duality and {K}oszul duality.
\newblock {\em Adv. Math.}, 225(4):2002--2049, 2010.

\bibitem[BLPW12]{BLPWtorico}
T.~Braden, A.~Licata, N.~Proudfoot, and B.~Webster.
\newblock Hypertoric category $\mathcal{O}$.
\newblock {\em Adv. Math.}, 231(3-4):1487--1545, 2012.

\bibitem[BLPW16]{BLPWgco}
T.~Braden, A.~Licata, N.~Proudfoot, and B.~Webster.
\newblock Quantizations of conical symplectic resolutions {II}: category
  {$\mathcal O$} and symplectic duality.
\newblock {\em Ast\'{e}risque}, (384):75--179, 2016.
\newblock with an appendix by I. Losev.

\bibitem[BLVS{\etalchar{+}}99]{OMbook}
A.~Bj\"{o}rner, M.~Las~Vergnas, B.~Sturmfels, N.~White, and G.~M. Ziegler.
\newblock {\em Oriented matroids}, volume~46 of {\em Encyclopedia of
  Mathematics and its Applications}.
\newblock Cambridge University Press, Cambridge, second edition, 1999.

\bibitem[BM17]{BMMatroid}
T.~Braden and C.~Mautner.
\newblock Matroidal {S}chur algebras.
\newblock {\em J. Algebraic Combin.}, 46(1):51--75, 2017.

\bibitem[BM19]{BMHyperRingel}
T.~Braden and C.~Mautner.
\newblock Ringel duality for perverse sheaves on hypertoric varieties.
\newblock {\em Adv. Math.}, 344:35--98, 2019.

\bibitem[BPW16]{BPW1}
T.~Braden, N.~Proudfoot, and B.~Webster.
\newblock Quantizations of conical symplectic resolutions {I}: local and global
  structure.
\newblock {\em Ast\'{e}risque}, (384):1--73, 2016.

\bibitem[BV78]{1978Orientability}
R.~G. Bland and M.~L. Vergnas.
\newblock Orientability of matroids.
\newblock {\em J. Comb. Theory, Ser. B}, 24:94--123, 1978.

\bibitem[CM93]{CordMor}
R.~Cordovil and M.~L. Moreira.
\newblock A homotopy theorem on oriented matroids.
\newblock volume 111, pages 131--136. 1993.
\newblock Graph theory and combinatorics (Marseille-Luminy, 1990).

\bibitem[Cor82]{Cord}
R.~Cordovil.
\newblock Sur les matro\"{\i}des orient\'{e}s de rang {$3$} et les arrangements
  de pseudodroites dans le plan projectif r\'{e}el.
\newblock {\em European J. Combin.}, 3(4):307--318, 1982.

\bibitem[Los17]{Losev}
I.~Losev.
\newblock On categories {$\mathcal O$} for quantized symplectic resolutions.
\newblock {\em Compos. Math.}, 153(12):2445--2481, 2017.

\bibitem[LV80]{LV}
M.~Las~Vergnas.
\newblock Convexity in oriented matroids.
\newblock {\em J. Combin. Theory Ser. B}, 29(2):231--243, 1980.

\bibitem[Pos95]{Posits}
L.~E. Positselski.
\newblock The correspondence between {H}ilbert series of quadratically dual
  algebras does not imply their having the {K}oszul property.
\newblock {\em Funktsional. Anal. i Prilozhen.}, 29(3):83--87, 1995.

\bibitem[Rin55]{Ringel}
G.~Ringel.
\newblock Teilungen der {E}bene durch {G}eraden oder topologische {G}eraden.
\newblock {\em Math. Z.}, 64:79--102 (1956), 1955.

\bibitem[RS72]{RoSa}
C.~P. Rourke and B.~J. Sanderson.
\newblock {\em Introduction to piecewise-linear topology}.
\newblock Springer-Verlag, New York-Heidelberg, 1972.
\newblock Ergebnisse der Mathematik und ihrer Grenzgebiete, Band 69.

\bibitem[Sta04]{StanleyComb}
R.~Stanley.
\newblock {\em Combinatorics and Commutative Algebra}.
\newblock Combinatorics and Commutative Algebra. Birkh{\"a}user Boston, 2004.

\end{thebibliography}
\end{document}